\documentclass[reqno, 12pt, reqno]{amsart}

\usepackage{amsfonts, amsthm, amsmath, amssymb}
\usepackage{hyperref}
\hypersetup{colorlinks=false}

\usepackage[margin=1.5in]{geometry}

\usepackage{helvet}

\RequirePackage{mathrsfs} \let\mathcal\mathscr

\numberwithin{equation}{section}

\newtheorem{theorem}{Theorem}[section]
\newtheorem{lemma}[theorem]{Lemma}

\theoremstyle{definition}
\newtheorem*{ack}{Acknowledgements}
\newtheorem{remark}[theorem]{Remark}
\newtheorem{definition}[theorem]{Definition}

\renewcommand{\d}{\mathrm{d}}
\renewcommand{\phi}{\varphi}
\renewcommand{\rho}{\varrho}

\newcommand{\sumstar}{\sideset{}{^*}\sum}

\newcommand{\0}{\mathbf{0}}

\newcommand{\PP}{\mathbb{P}}
\renewcommand{\AA}{\mathbb{A}}
\newcommand{\A}{\mathbf{A}}
\newcommand{\FF}{\mathbb{F}}
\newcommand{\ZZ}{\mathbb{Z}}

\newcommand{\NN}{\mathbb{Z}_{>0}}
\newcommand{\QQ}{\mathbb{Q}}
\newcommand{\RR}{\mathbb{R}}
\newcommand{\CC}{\mathbb{C}}

\newcommand{\TT}{\mathbb{T}}

\newcommand{\cO}{\mathcal{O}}
\newcommand{\cF}{\mathcal{F}}

\newcommand{\Gal}{{\rm Gal}}

\renewcommand{\leq}{\leqslant}

\renewcommand{\geq}{\geqslant}

\renewcommand{\bar}{\overline}

\newcommand{\h}{\mathbf{h}}
\newcommand{\bH}{\mathbf{H}}

\newcommand{\x}{\mathbf{x}}
\newcommand{\y}{\mathbf{y}}
\renewcommand{\c}{\mathbf{c}}

\renewcommand{\u}{\mathbf{u}}
\newcommand{\z}{\mathbf{z}}
\newcommand{\w}{\mathbf{w}}
\renewcommand{\b}{\mathbf{b}}
\renewcommand{\a}{\mathbf{a}}
\renewcommand{\k}{\mathbf{k}}

\renewcommand{\r}{\mathbf{r}}

\newcommand{\ve}{\varepsilon}

\newcommand{\bal}{\boldsymbol{\alpha}}
\newcommand{\bbe}{\boldsymbol{\beta}}

\DeclareMathOperator{\rank}{rank}

\DeclareMathOperator{\Pic}{Pic}

\DeclareMathOperator{\meas}{meas}

\DeclareMathOperator{\Spec}{Spec}

\DeclareMathOperator{\tr}{Tr}

\DeclareMathOperator{\ord}{ord}
\DeclareMathOperator{\sw}{swan}

\DeclareMathOperator{\fr}{Fr}
\DeclareMathOperator{\ch}{char}
\DeclareMathOperator{\Hom}{Hom}

\renewcommand{\=}{\equiv}

\newcommand{\vp}{\varpi}

\renewcommand{\hat}{\widehat}
\newcommand{\Ki}{K_\infty}
\newcommand{\ki}{K_\infty}
\newcommand{\hQ}{\widehat{Q}}

\newcommand{\hJ}{\widehat J}

\newcommand{\het}{H_{\text{\'et}}}

\begin{document}


\date{\today}

\title[Cubic hypersurfaces over $\FF_q(t)$]{Rational points on\\ cubic hypersurfaces over $\FF_q(t)$}

\author{T.D. Browning}
\address{School of Mathematics\\
University of Bristol\\ Bristol\\ BS8 1TW
\\ UK}
\email{t.d.browning@bristol.ac.uk}

\author{P. Vishe}
\address{
Department of Mathematics\\
University of York\\
York\\ YO10 5DD\\ UK}
\email{pankaj.vishe@york.ac.uk}

\begin{abstract}
The Hasse principle and weak approximation is established for non-singular cubic hypersurfaces $X$ over the function field $\FF_q(t)$, provided that $\ch(\FF_q)>3$ and $X$ has dimension at least $6$.
\end{abstract}

\maketitle

\setcounter{tocdepth}{1}
\tableofcontents

\thispagestyle{empty}

\section{Introduction}

Let $K$ be a global field and let $X\subset \PP_K^{n-1}$ be  a  cubic hypersurface defined over $K$.
A ``folklore'' conjecture predicts that the set $X(K)$ of $K$-rational points on $X$ is non-empty as soon as $n\geq 10$. 
When $K=\FF_q(C)$ 
is the function field of a smooth and projective curve $C$ over the finite field $\FF_q$
this conjecture follows from 
the Lang--Tsen theorem 
(see \cite[Thm.~3.6]{greenberg}), since  $K$ has transcendence degree $1$ over a $C_1$-field. Alternatively, when $K$ is a number field, it follows from recent work of the authors \cite{skinner} provided that  $X$ is assumed to be non-singular.  We record this observation as follows.

\begin{theorem}\label{t:10}
Let $K$ be a global field and let $X\subseteq \PP_K^{n-1}$ be a non-singular cubic hypersurface defined over $K$. If $n\geq 10$ then $X(K)\neq \emptyset$.
\end{theorem}

The main goal of this paper is to  improve this result in the special case 
$K=\FF_q(t)$.   Compared to the situation over number fields, there are relatively few results in the literature which deal with the Hasse principle and weak approximation for cubic hypersurfaces defined over $K$. 
One notable exception is found in work of Colliot-Th\'el\`ene \cite[\S 3]{bud}, which establishes the Hasse principle for  the diagonal threefolds 
$$
a_1x_1^3+\dots+ a_5 x_5^3=0, \quad (a_1,\dots,a_5\in K^*),
$$
provided that $q$ is odd and $q\equiv 2 \bmod{3}$. Furthermore, subject to a collection of explicit constraints on the coefficients, he is able to draw the same conclusion for diagonal cubic surfaces in $\PP_K^3$.
These results are established by  adapting to $K$ work of Swinnerton-Dyer \cite{swd} on this problem over number fields.  
It is worth highlighting that Swinnerton-Dyer's approach relies on a delicate analysis of  certain Selmer groups and this  leads to a final result which is conditional on the conjecture that the Tate--Shafarevich group of an elliptic curve is finite. 
The advantage of working over the function field $K$ is that the analagous statements can be made {\em unconditional} --- a feature that will resurface in the present investigation.

Turning to weak approximation, in the setting 
$n=4$ of non-singular cubic surfaces it follows from work of Hu \cite[Thm.~5]{Hu} that $X$ satisfies weak approximation at the places of good reduction,  provided that $\ch(\FF_q)>3$ and $q>47$.
For larger values of $n$ a suitable variant of the Hardy--Littlewood circle method can  be brought to bear on this problem. 
Let $X\subset \PP_K^{n-1}$ be  a non-singular cubic hypersurface defined over $K$. Assuming that $\ch(\FF_q)>3$ it follows from work of Lee (see \cite{lee} and his 2013 PhD thesis \cite{lee-thesis}) that weak approximation holds for $X$ over $K$ provided that $n\geq 14$. Note that the Hasse principle is trivial for $n$ in this range by Theorem 
\ref{t:10}.

By developing  an alternative version of the circle method, we shall establish the following improvement.

\begin{theorem}\label{thm:HP+WA}
Let $K=\FF_q(t)$ with $\ch(\FF_q)>3$. 
Let $X\subset \PP_K^{n-1}$ be a non-singular cubic hypersurface defined over $K$, with $n\geq 8$.
Then $X$ satisfies the Hasse principle and  weak approximation over $K$.
\end{theorem}

The restriction on the characteristic of $\FF_q$ in this result is unfortunate but intrinsic to the method.
The same restriction appears in Lee's work \cite{lee, lee-thesis},
where it stems from  the use of Weyl differencing in the analysis of certain cubic exponential sums, 
which produces factors of $3!$ within the argument of the resulting sums. 
In our case, the restriction on the characteristic 
comes from the need to find an auxiliary point on the hypersurface $X$ at which the associated 
Hessian does not vanish.
For diagonal forms over $\FF_q(t)$,   Liu and Wooley \cite{liu} have shown how to handle arbitrary characteristic. 
Their approach uses the large sieve 
to act as a substitute for  Weyl differencing and it would be interesting to see whether this innovation
could be adapted to general forms. 

It is natural to compare  Theorem 
 \ref{thm:HP+WA}  with the situation of non-singular cubic hypersurfaces over function fields $K=k(C)$ of a  curve $C$ over an algebraically closed field $k$ of characteristic $0$.
In this setting  Lang--Tsen theory confirms that $X(K)\neq \emptyset$ for $n\geq 4$. 
On the other hand,
Hassett and Tschinkel \cite[Thm.~1]{HT}
have shown that 
 $X$ satisfies weak approximation over $K$ provided that $n\geq 7$.

Theorem \ref{thm:HP+WA} is  the $\FF_q(t)$-analogue  of recent work by Hooley \cite{oct} about non-singular cubic hypersurfaces $X\subset \PP^{n-1}_\QQ$ over the rational numbers. Hooley's main result
 establishes the Hasse principle for $X$, provided  that $n\geq 8$,  conditionally under a certain unproved ``Hypothesis HW'' about the analytic properties of Hasse--Weil $L$-functions associated to a family of 
$5$-dimensional cubic hypersurfaces.
Over the last century 
the theory of the Hardy--Littlewood circle method 
has become  heavily industrialised in its application to cubic forms over $\QQ$, reaching a zenith in 
Hooley's work on octonary cubic forms. 
The igniting spark in his work  is the smooth $\delta$-function technology 
that was introduced by Duke, Friedlander and Iwaniec \cite{DFI}.  This paves the way 
to getting  non-trivial averaging over  the approximating fractions $a/q$
that appear in the associated cubic exponential sums. 
Note that Hooley requires non-trivial averaging over both  numerators and  denominators to handle cubic forms in $8$ variables. This is usually termed a ``double Kloosterman refinement'', with the usual ``Kloosterman refinement'' 
connoting non-trivial averaging over the numerators only. The ordinary Kloosterman refinement is  only capable of handling cubic forms in $n\geq 9$
variables (see pioneering work of Heath-Brown \cite{hb-10} and Hooley \cite{nonaryI}),  but when it works it  produces  completely unconditional results. 
The use of a double Kloosterman refinement over $\QQ$ leads to the analysis of  global $L$-functions associated to cubic hypersurfaces of dimension $5$. Since our knowledge about such $L$-functions 
is extremely scarce in dimension $>1$, any  progress is dependent on Hypothesis HW, which  describes the meromorphic  continuation and location of zeros of these $L$-functions.
The significance of Theorem \ref{thm:HP+WA} is that 
 working over $K=\FF_q(t)$ affords a completely unconditional result.

\medskip

The proof of Theorem \ref{thm:HP+WA} is long and complicated and we proceed to outline some of the key ingredients. 
Our approach is based on estimating the number $N(d)$ of suitably weighted vectors $(x_1,\dots,x_n)\in \FF_q[t]^n$, with $\max_{i}\deg x_i< d$, for which $[x_1,\dots,x_n]\in X(K)$.  The principal result of this paper is  Theorem~\ref{THM}, 
which provides an asymptotic formula for $N(d)$ when $n=8$,  as $d\to \infty$.
This will suffice to prove Theorem~\ref{thm:HP+WA} when $n=8$. For $n\geq 9$ we will deduce the result via a fibration 
$X\to \PP^1_K$
in \S \ref{s:beep}. 
Theorem \ref{THM} is  established using the circle method.

As one might expect, 
parts of the circle method machinery 
become greatly simplified when  transported to the  function field  $K$.
The first simplification comes in the analogue of the smooth $\delta$-function that lies at the heart of Hooley's work.
Indeed, the absolute values of $K$ satisfy the ultrametric inequality. This allows us to cover the analogue of the unit interval with non-overlapping arcs using nothing more sophisticated than a version of Dirichlet's theorem on Diophantine approximation 
over $K$
(see Remark \ref{rem:dissection}). 
Thus we are immediately placed in the position of being able to carry out a double Kloosterman refinement.  This appears to be the first attempt to extract non-trivial savings, \`a la Kloosterman,
over function fields. 

The process of non-trivial averaging leads us to consider
the global $L$-function
$L(H_\ell^m(Y),s)$   which is affiliated to the
middle  $\ell$-adic cohomology group 
  $H_\ell^m(Y)=\het^m(Y\otimes_K \bar K, \QQ_\ell)$ of a non-singular cubic hypersurface $Y\subset \PP_K^{m+1}$ of dimension $m$.
In \S \ref{s:red} we will relate these $L$-functions to a very general class of global $L$-functions 
that were associated  to arbitrary lisse $\ell$-adic sheaves 
by Grothendieck \cite{bourbaki}. The 
 second major advantage 
of working over $K$ is that, thanks to Grothendieck  and Deligne \cite{weil2, bourbaki}, 
we know that these  $L$-functions are actually rational functions of $q^{-s}$ 
that satisfy the Riemann hypothesis. 
Thus  
for $k\in \{0,1,2\}$ there are polynomials $P_k=P_{k,m}\in \ZZ[T]$, with 
inverse roots having  absolute value $q^{(k+m)/2}$,  such that 
$$
L(H_\ell^m(Y),s)=\frac{P_1(q^{-s})}{P_0(q^{-s})P_2(q^{-s})}.
$$
In \S
\ref{s:sqf}, for even  $n\geq 8$,
this information will allow us to execute an  unconditional double Kloosterman refinement  by getting savings in the treatment of 
the relevant
cubic exponential sums with square-free modulus. This is the most novel part of our investigation. 

In order to make use of the analytic properties of $L(H_\ell^m(Y),s)$, we shall also need to contend with an issue that 
represents a much greater challenge in the function field setting than in the classical one. Over $\QQ$,  partial summation   is widely used as a means of transforming the summation of products of sequences into easier  summations,
but this device is not  readily available over $K$. The underlying obstacle comes from the fact that 
there are  only two rational integers with a given absolute value,  
but  $q^{d+1}$ elements of $\FF_q[t]$ with given degree $d$ (and all of these will have equal absolute value).
We will circumvent this difficulty by introducing  Dirichlet characters on 
$(\FF_q[t^{-1}]/t^{-J}\FF_q[t^{-1}])^*$, for  a positive integer $J$,
and then
showing that the analytic properties enjoyed by $L(H_\ell^m(Y),s)$ continue to hold  when $H_\ell^m(Y)$ is twisted by the
 Galois representation induced by these
characters.

There remains the not insignificant task of handling cubic exponential sums with square-full modulus. 
Unfortunately, 
the passage to function fields doesn't offer any simplification of this task and the necessary arguments are 
mostly direct analogues of the corresponding treatment over $\QQ$ found in \cite{hb-10} and \cite{oct}. Given the length of the paper we will capitalise  
on the inherent similarities by not providing a complete treatment of all the estimates that are recorded in \S \ref{s:start-cubic} and \S \ref{s:averages}. Instead we shall 
content ourselves with proving the function field analogues of the key ideas that  underpin the arguments over $\QQ$. 
One ingredient that we require is a non-trivial bound for the number of $K$-rational points on a geometrically irreducible hypersurface $V\subset \PP_K^{n-1}$ which is not a hyperplane. Let $H:\PP_K^{n-1}(K)\to \RR$ be the standard exponential height function. Then, 
as a special case of  Lemma \ref{lem:cohen}, it follows that
\[
\#\left\{x\in V(K) : H(x)\leq q^B
\right\} 
=O_{\ve,V}\left( q^{B(n-3/2+\ve)}\right),
\]
for any $B\geq 1$ and any $\ve>0$.
There are very few  results of this sort in the literature 
over function fields
and it would be interesting to see whether the rapid recent  advances involving the ``determinant method'' over number fields could be adapted to improve this upper bound.

\medskip

Finally, suppose   that  $X\subset \PP^{n-1}_{\FF_q}$ is a non-singular cubic hypersurface defined over a finite field $\FF_q$, with $\mathrm{char}(\FF_q)>3$.
There is a correspondence between the counting function $N(d)$ for $\FF_q(t)$-points on $X$ of bounded height
and the cardinality of $\FF_q$-points on the moduli space $\mathrm{Mor}_d(\PP_{\FF_q}^1,X)$, which parameterises the rational maps of degree $d$ on $X$. Following an idea of Ellenberg and Venkatesh it is possible to 
exploit the Lang--Weil estimate to make deductions about the basic geometry of this moduli space via an asymptotic formula for $N(d)$, provided that  sufficient uniformity is achieved in the $q$-aspect. 
Using the present investigation as a base, we
have produced a short companion paper \cite{BV-short} which carries out this plan.

\begin{ack}
While working on this paper the first  author was
supported by ERC grant \texttt{306457}
and the second author by 
EPSRC programme grant \texttt{EP/J018260/1}.
This work has benefitted from useful conversations with 
Alexei Entin, 
Bruno Kahn, Emmanuel Kowalski, Daniel Loughran, Philippe Michel and Trevor Wooley. Their input is gratefully acknowledged.  Thanks are also due to the anonymous referee for several helpful comments  that have particularly helped to clarify the exposition in \S \ref{s:red}.
\end{ack}

\section{Auxiliary facts about function fields}

\subsection{Notation}
In this section we collect together some notation and basic facts concerning the function field $K=\FF_q(t)$. 
To begin with, for any real number $R$ we will always write $\hat R=q^R$.

Let 
$\mathcal{O}=\FF_q[t]$ be the ring of integers of $K$ and 
let $\Omega$ be the set of  places of $K$. These correspond to either monic irreducible polynomials $\varpi$ in $\mathcal{O}$, which we call the {\em finite primes},  or the {\em prime at infinity} $t^{-1}$ which we usually denote  by $\infty$. 
The associated absolute value  $|\cdot|_v$ is either $|\cdot|_\vp$ for some prime $\vp\in \mathcal{O}$ or $|\cdot|_{\infty}$, according to whether $v$ is a finite or infinite place, respectively. 
These  are given by 
$$
|a/b|_\vp=\left(\frac{1}{q^{\deg \vp}}\right)^{\ord_\vp(a/b)} \quad \text{ and }\quad
|a/b|_\infty=
q^{\deg a-\deg b},
$$
for any $a/b\in K^*$.
We extend these definitions to  $K$ by taking $|0|_\vp=|0|_\infty=0.$
We will usually just write $|\cdot|=|\cdot|_\infty$.

For $v\in \Omega$ we let $K_v$ denote the completion of $K$ at $v$ with respect to $|\cdot|_v$. We put
$
\mathcal{O}_v=\{a\in K_v: |a|_v\leq 1\}
$
for the maximal compact subring
and 
$
\mathcal{O}_v^*=\{a\in K_v: |a|_v= 1\}
$
for the unit group.
Furthermore,  we let $\FF_v$ denote its residue field. We have  $\FF_\infty=\FF_q$ and 
$\FF_{\vp}=\FF_{q^{\deg(\vp)}}$ for any finite prime $\vp$.
The elements of $\mathcal{O}_\infty$ are power series expansions in $t^{-1}$.

We may identify $K_\infty$ with the set 
$$
\FF_q((1/t))=\left\{\sum_{i\leq N}a_it^i: \mbox{for $a_i\in \FF_q$ and some $N\in\ZZ$} \right\}
$$
and put 
$$
\TT=\{\alpha\in K_\infty: |\alpha|<1\}=\left\{\sum_{i\leq -1}a_it^i: \mbox{for $a_i\in \FF_q$}
\right\}.
$$
Let  $\delta\in \TT$. Then $\TT/\delta\TT$ is the set of cosets $\alpha+\delta \TT$, of which there are  
$|\delta|.$

We can extend the absolute value at the infinite place to $K_\infty$ to get  a non-archimedean
absolute value 
$|\cdot|:K_\infty\rightarrow \RR_{\geq 0}$ 
given by $|\alpha|=q^{\ord \alpha}$, where $\ord \alpha$ is the largest $i\in
\ZZ$ such that $a_i\neq 0$ in the 
representation $\alpha=\sum_{i\leq N}a_it^i$.
In this context we adopt the convention $\ord 0=-\infty$ and $|0|=0$.
We extend this to vectors by setting 
$
|\x|=\max_{1\leq i\leq n}|x_i|, 
$
for any $\x\in K_\infty^n$. 

Since $\TT$ is a locally compact
additive subgroup of $K_\infty$ it possesses a unique Haar measure $\d
\alpha$, which is normalised so that 
$
\int_\TT \d\alpha=1.
$
We can extend $\d\alpha$ to a (unique) translation-invariant measure on $\Ki$ in
such a way that 
$$
\int_{\{\alpha\in\Ki: 
|\alpha|<\hat N
\}} \d \alpha=\hat N,
$$ 
for any $N\in \NN$.
These measures also extend to $\TT^n$ and $\Ki^n$, for any $n\in \NN$.

For given $\x,\b\in \cO^n$  and $M\in \cO$ we will sometimes write $\x\equiv \b
\bmod{M}$ to mean that 
$\x=\b+M\y$ for some $\y\in \cO^n$.

\subsection{Characters}\label{s:add-characters}
There is a non-trivial additive character $e_q:\FF_q\rightarrow \CC^*$ defined
for each $a\in \FF_q$ by taking 
$e_q(a)=\exp(2\pi i \tr(a)/p)$, where $\tr: \FF_q\rightarrow \FF_p$ denotes the
trace map.
This character induces a non-trivial (unitary) additive character $\psi:
K_\infty\rightarrow \CC^*$ by defining $\psi(\alpha)=e_q(a_{-1})$ for any 
$\alpha=\sum_{i\leq N}a_i t^i$ in $\Ki$. In particular it is clear that
$\psi|_\cO$ is trivial. 
More generally, given 
 any $\gamma\in \Ki$, the map $\alpha\mapsto \psi(\alpha\gamma)$ is an additive
character on $\Ki$.
We have the basic orthogonality property
$$
\sum_{\substack{b\in \cO\\ |b|<\hat N}}\psi(\gamma b)=\begin{cases}
\hat N, & \mbox{if $|\gamma|<\hat N^{-1}$,}\\
0, & \mbox{otherwise}.
\end{cases}
$$
for any $\gamma \in \ki$ and any integer $N\geq 0$ (see Lemma 7 of
\cite{kubota}).

We will also need standard characters at the finite places
 (we follow Ex.~7.5 of \cite{ram} for their construction).
Let $K_\vp$ be the completion of $K$ at the place corresponding to finite prime $\vp\in \cO$ of degree $d\geq 1$, with corresponding ring
of integers $\cO_\vp$.
According to 
 \cite[Ex.~7.5(c)]{ram},  any element $x\in K_\vp$ can be written as
$x=y/\vp^N+z$ for some integer $N\geq 0$ and $z\in \cO_\vp$,  where
$
y=a_1t^{dN-1}+a_2t^{dN-2}+\dots +a_{dN},
$
with all coefficients $a_i\in \FF_q$.
With this representation one defines 
the non-trivial additive character  $\psi_\vp: K_\vp\rightarrow\CC^*$
to be 
given by 
$$
\psi_\vp(x)=e_q(a_1).
$$

Letting $\A_K$ denote the ad\`eles over $K$, 
we may now define the {\em standard adelic character}
$\psi_K:\A_K\rightarrow\CC^*$
to be
$$
\psi_K(x)=\psi(x_\infty)\prod_{\vp}\psi_\vp(x_\vp),
$$
for any $x=(x_v)\in \A_K$.  It follows from 
\cite[Ex.~7.6]{ram} that $\psi_K$ is a non-trivial additive character of $\A_K$
which is trivial on $K$.

\subsection{Fourier analysis on non-archimedean local fields}

The material we summarise here is found in \cite[\S 7]{ram}, but has its genesis in work of 
Schmid and Teichm\"uller \cite{ST}. (The authors are grateful to Ivan Fesenko for this reference.)
We first fix a non-trivial 
 additive character 
$\varphi: F\rightarrow \CC^*$ 
on a  non-archimedean local field $F$.
A function $f:F \rightarrow \CC$ is said to be {\em smooth } if it is locally
constant (that is, $f(x)=f(x_0)$ for all $x$ sufficiently close to $x_0$). A
{\em Schwartz--Bruhat function} is a smooth function $f: F \rightarrow \CC$ with
compact support.
We denote by $S(F)$ the set of all such functions. Then for any $f\in S(F)$ we
may define the Fourier transform of $f$ by 
$$
\hat f(y)=\int_{F} f(x)\varphi(xy)\d x,
$$
where $\d x$ is  Haar measure. This function also belongs to $S(F)$.

Let  $K=\FF_q(t)$.  
We define $S(\A_K)$
to be the space of functions given by
$$
f(x)=\prod_{\nu} f_\nu (x_\nu),
$$
for $x=(x_\nu)\in \A_K$. Here, 
 $f_\nu\in S(K_\nu)$ for every place $\nu$ and $f_\vp|_{\cO_\vp}=1$ for almost
all primes  $\vp$.
The {\em adelic Fourier transform} of any $f\in S(\A_K)$ is given by 
$$
\hat f(y)=\int_{\A_K} f(x)\psi_K(xy)\d x,
$$
where $\psi_K$ is the standard adelic character on $\A_K$
and $\d x$ is Haar measure on $\A_K$ (normalised
to be the self-dual measure for $\psi_K$).
With this notation the Poisson summation formula (see \cite[Thm.~7.7]{ram}, for
example) 
states that 
$$
\sum_{x\in K} f(x)=\sum_{x\in K} \hat f(x),
$$
for any $f\in S(\A_K)$.  This extends to a summation over $\x\in K^n$ in the
obvious way.

We will need to introduce some weight functions on $K^n$. 
For a prime $\vp$ 
define $w_\vp: K_\vp \rightarrow \{0,1\}$ via
$$
w_\vp(x)=\begin{cases}
1, &\mbox{if $|x|_\vp\leq 1$,}\\
0, & \mbox{otherwise}.
\end{cases}
$$
This gives an  indicator function for the   ring of integers 
$\cO_\vp$.
It is easy to check that $\hat w_\vp=w_\vp$.
Next let $w_\infty: K_\infty \rightarrow \{0,1\}$ be  the indicator function for
$\TT$, defined via
$$
w_\infty(x)=
\begin{cases}
1, &\mbox{if $|x|<1$,}\\
0, & \mbox{otherwise}.
\end{cases}
$$
We proceed to define weight functions $w_f, w: K^n\rightarrow \{0,1\}$ via
\begin{equation}\label{eq:weights}
w_f(\x)=\prod_{1\leq i\leq n}\prod_{\vp} w_\vp(x_i), 
\quad 
w(\x)=
\prod_{1\leq i\leq n} w_\infty(x_i).
\end{equation}
Let $\z\in K^n$. 
Then  $|\z|<\hat P$ if and only if 
$w(\z/t^P)=1$ and  $\z\in \cO$ if and only if $w_f(\z)=1.$

We will use the $n$-dimensional Poisson summation formula to prove the following
result.

\begin{lemma}\label{lem:poisson}
Let $f\in \ki[x_1,\dots,x_n]$ be a polynomial and let 
$v\in S(K_\infty^n)$.  Then we have 
$$
\sum_{\substack{ \z\in \cO^n}} \psi (f(\z)) v(\z) = 
\sum_{\substack{ \c\in \cO^n}} \int_{K_\infty^n} v(\u)
\psi (f(\u)+\c.\u)\d \u. 
$$

\end{lemma}

\begin{proof}
Recalling the definitions of the weight functions $w_f$ and $w$,  we may write 
$$
\sum_{\substack{ \z\in \cO^n}} \psi (f(\z)) v(\z)= 
\sum_{\substack{ \z\in K^n}} g(\z),
$$
where
$
g(\z)=\psi (f(\z)) w_f(\z) v(\z).
$
It is clear that $g\in S(\A_K^n)$ and so we are free to apply the
$n$-dimensional version of Poisson summation to conclude that
\begin{align*}
\sum_{\substack{ \z\in \cO^n}} \psi (f(\z)) v(\z)
&= 
\sum_{\substack{ \c\in K^n}} 
\hat w_f(\c)
\int_{K_\infty^n} v(\u)
\psi (f(\u)+\c.\u)\d \u.
\end{align*}
Since
$\hat w_f(\c)=
 w_f(\c)$, the lemma follows.
\end{proof}

\subsection{Some integral formulae}\label{s:integrals}

In this section we collect  some basic facts and estimates concerning multi-dimensional integrals over $K_\infty^n$.  Recall the definition of the additive character $\psi:K_\infty\to \CC^*$ from 
\S \ref{s:add-characters}. We begin by recording the following fact (see Lemma 1(f) of \cite{kubota}).

\begin{lemma}\label{lem:orthog}
Let $Y\in \ZZ$ and $\gamma\in K_\infty$. Then 
$$
\int_{|\alpha|<\hat Y} \psi(\alpha \gamma) \d \alpha=\begin{cases}
\hat Y, &\mbox{if $|\gamma|<\hat Y^{-1}$},\\
0, &\mbox{otherwise.}
\end{cases}
$$
\end{lemma}

Taking $Y=0$,  it follows from this result that 
\begin{equation}\label{eq:orthog}
\int_{\TT} \psi(\alpha x) \d \alpha=\begin{cases}
1, &\mbox{if $x=0$,}\\
0, &\mbox{if $x\in \cO\setminus\{0\}$.}
\end{cases}
\end{equation}
We also have the following change of variables formula, which readily follows from Igusa \cite[Lemma 7.4.2]{igusa}.

\begin{lemma}\label{lem:change}
Let $\Gamma\subset \Ki^n$ be a box defined by the inequalities 
$|x_i|<\hat R_i$, for some real numbers $R_1,\dots,R_n$.
Let $f:\Gamma\rightarrow \CC$ be a continuous function. Then 
for any $M\in \mathrm{GL}_n(\Ki)$ we have 
$$
\int_{\Gamma}f(\bal)\d \bal=|\det M| \int_{M\bbe\in \Gamma} f(M\bbe)\d\bbe.
$$
\end{lemma}

It will be convenient to reserve some notation for the {\em height} of a
polynomial
$f\in K_\infty[x_1,\dots,x_n]$. 
Assuming that $f(\x)=\sum_{\mathbf{i}} a_{\mathbf{i}}\x^{\mathbf{i}}$, for
coefficients 
$a_{\mathbf{i}}\in K_\infty$, we define
$$
H_f=\max_{\mathbf{i}}|a_{\mathbf{i}}|.
$$
We proceed to establish the following result. 

\begin{lemma}\label{lem:1st}
Let $f\in \ki[x_1,\dots,x_n]$ be a polynomial
and let $\w\in \ki^n$. 
Assume that $|\w|\geq 1$ and $|\w|>H_f$.
Then 
$$
\int_{\TT^n} \psi\left( f(\x)+\w.\x\right)\d\x
=0.
$$
\end{lemma}

\begin{proof}
Suppose without loss of generality that $|\w|=|w_1|=\hat N$ for some integer $N\geq 0$. We 
concentrate on the one-dimensional integral
$$
I=\int_{\TT} \psi\left( f(x,x_2,\dots,x_n)+w_1x\right)\d x, 
$$
for fixed $x_2,\dots,x_n\in \TT$.
Write $g(x)=f(x,x_2,\dots,x_n)$. 
We may suppose that  $g(x)=c_0x^d+\dots+c_{d-1}x$ for appropriate coefficients
$c_i=c_i(x_2,\dots,x_n)\in K_\infty$.
Our hypothesis implies that 
\begin{equation}\label{eq:chew}
H_g \leq H_f < \hat N.
\end{equation} 
It is
clear that 
$\psi(g(x)+w_1x)=1$ if $|x|<q^{-N-1}.$
But then, using the definition of integration over $\TT$, we find that 
\begin{align*}
I
&= \lim_{m\rightarrow \infty} q^{-m}\sum_{a_{-m},\dots ,a_{-1}\in \FF_q} 
\psi\left( h(a_{-m}t^{-m}+\dots+a_{-1}t^{-1})
\right)\\
&= q^{-N-1}\sum_{a_{-N-1},\dots,a_{-1}\in \FF_q} 
\psi\left( h(a_{-N-1}t^{-N-1}+\dots+a_{-1}t^{-1})
\right),
\end{align*}
where $h(x)=g(x)+w_1x$. 
The coefficient of $t^{-1}$ in  
$w_1x$  is $a_{-N-1}$
Moreover, \eqref{eq:chew} implies that 
$
|g(a_{-N-1}t^{-N-1}+y)-g(y)|<|t^{-1}|
$
for any $y\in \TT$.
This  implies that 
the coefficient of $t^{-1}$
in $g(x)$ is a polynomial in $a_{-N}, \dots,a_{-1}$ alone, 
from which it follows
that
$I=0$, since
$$
\sum_{a_{-N-1}\in \FF_q} e_q(a_{-N-1})=0.
$$
This completes the proof of the lemma.
\end{proof}

As an easy  consequence of Lemma \ref{lem:1st} we get the following result. 

\begin{lemma}\label{lem:deriv}
Let $f\in \ki[x_1,\dots,x_n]$ be a polynomial.
Suppose that there exists $\u\in \TT^n$ 
and   $\lambda\geq 1$ 
such that 
$|\nabla f(\u)|\geq \lambda$ and
$|\partial^{\bbe} f(\u)|<\lambda$, for all $|\bbe|\geq 2$. Then
$$\int_{\TT^n}\psi(f(\x))\d\x=0.$$
\end{lemma}

\begin{proof}
Make the change of variables $\x=\u+\y$ and note that 
$$
f(\x)=f(\u)+\y.\nabla f(\u)+\tfrac{1}{2}\y^T\nabla^2f(\u)\y+\dots.
$$ 
The conclusion is now a direct consequence of Lemma \ref{lem:1st} with 
$\w=\nabla f(\u)$.
\end{proof}

Given a non-zero polynomial  $F\in K_\infty[x_1,\dots,x_n]$,
 integrals of the form 
\begin{equation}\label{eq:J}
J_F(\gamma;\w)=\int_{\TT^n} \psi\left( \gamma F(\x)+\w.\x\right)\d\x
\end{equation}
will feature prominently in our work, for given $\gamma\in \ki$ and $\w\in
\ki^n$. 
On noting that $H_{\gamma F}= |\gamma|H_F$, the following result is a trivial consequence of Lemma 
\ref{lem:1st}.

\begin{lemma}\label{lem:J-easy}
We have 
$
J_F(\gamma;\w)=0$ if $|\w|> \max\{1,|\gamma| H_F\}$.
\end{lemma}

The following result will be  useful  when 
$|\w|$ is not too large.

\begin{lemma}\label{lem:small}
We have 
 $$
 J_F(\gamma;\w)=
 \int_{\Omega} \psi\left(\gamma F(\x)
+\w.\x \right) \d\x,
$$
where
$
\Omega=\left\{\x\in \TT^n: |\gamma\nabla F(\x)+ \w|\leq H_F
\max\{1,
|\gamma|^{1/2}\}\right\}. 
$
\end{lemma}

\begin{proof}
Let $\Omega_0=\TT^n\setminus \Omega$. We  break  the integral over $\Omega_0$ into a sum of  integrals over 
smaller  regions. Let $\delta\in K_\infty$ be such that 
$|\delta|=\min \{1,|\gamma|^{-1/2}\}$. 
Introducing a dummy sum over $\y \in (\TT/\delta\TT)^n$  and then 
using Lemma \ref{lem:change} to make the
change of variables $\x=\y+\delta\z$,  we obtain
\begin{align*}
\int_{\Omega_0} \psi\left( \gamma F(\x)+\w.\x\right)\d\x
&=
|\delta|^{-n}
\sum_{\y \in (\TT/\delta\TT)^n}
\int_{\Omega_0} \psi\left( \gamma F(\x)+\w.\x\right)\d\x\\
&=
\sum_{\y \in (\TT/\delta\TT)^n}
\int_{\{\z\in \TT^n: \y+\delta\z\in \Omega_0\}}\psi\left( f(\z)
\right)\d\z,
\end{align*}
where  $f(\z)=\gamma F(\y+\delta\z)+\w.(\y+\delta \z) $. 
We want to show that the inner integral vanishes, to which end
we claim that 
$\y+\delta\z\in \Omega_0$ if and only if $\y$ satisfies 
\begin{equation}\label{eq:gut}
|\gamma \nabla F(\y)+\w|>H_F\max\{1,|\gamma|^{1/2}\}.
\end{equation}
Using Taylor expansion and observing that 
$|\delta\gamma|=\min\{1,|\gamma|^{1/2}\}$, 
we deduce that there is a vector $\u$
depending on $\z$, with $|\u|<H_F \min\{1,|\gamma|^{1/2}\}$, such that 
$\gamma \nabla F(\y+\delta\z)+\w=\gamma \nabla F(\y)+\w+\u$.
Put $A=H_F\max\{1,|\gamma|^{1/2}\}$.
If $\y+\delta\z\in \Omega_0$ then 
$$
A
<\max\{
|\gamma \nabla F(\y)+\w|
,|\u|\}
<\max\{
|\gamma \nabla F(\y)+\w|
,A\},
$$
which implies that \eqref{eq:gut} holds.
Conversely, if \eqref{eq:gut} holds then 
$$
A<|\gamma \nabla F(\y)+\w+\u|<
|\gamma \nabla F(\y+\delta\z)+\w|.
$$
This therefore establishes the claim. 
Hence we have
\begin{align*}
\int_{\Omega_0} \psi\left( \gamma F(\x)+\w.\x\right)\d\x
&=
\sum_{\substack{\y \in (\TT/\delta\TT)^n\\
|\gamma \nabla F(\y)+\w|>H_F\max\{1,|\gamma|^{1/2}\}
}}
\int_{\TT^n}\psi\left( f(\z)
\right)\d\z.
\end{align*}
Now all the partial
derivatives of $f(\z)$ of order $k \geq 2 $ are strictly less than    $H_F |\gamma| |\delta|^k
\leq H_F\min\{1,|\gamma|\}$.  Moreover, our preceding argument  shows that   $|\nabla f(\z)|> H_F\max\{1,|\gamma|\}$ 
for every  $\z\in \TT^n$.
An application of Lemma 
\ref{lem:deriv} therefore shows that the inner integral vanishes, 
as required to complete the proof. 
\end{proof}

\subsection{Density of integer points on affine hypersurfaces}

Let $V\subset \AA_K^n$ be an affine variety defined over $\cO$ of degree $d\geq 1$ and dimension $m\geq 1$.
Using a version of the large sieve inequality
over function fields  due to Hsu \cite{hsu}, 
our main goal in this section is to establish a pair of estimates for the number of $\cO$-points on $V$ with bounded absolute value.

\begin{lemma}\label{lem:trivial}
We have 
$
\#\{\x\in V(\cO):  |\x|\leq \hat N\}=O_{d,n}(
q^{(N+1)m}),
$
where the implied constant only depends on $d$ and $n$.
\end{lemma}

This result is optimal whenever $V$ contains a linear component of dimension $m$.
Alternatively, we will obtain the following improvement. 

\begin{lemma}
\label{lem:Cohen}
Assume that $V$ is absolutely irreducible  and $d\geq 2$. Then 
we have 
$$
\#\{\x\in V(\cO):  |\x|\leq \hat N\}=O_{d,n}(
q^{(N+1)(m-1/2)}N\log q),
$$ 
where the implied constant only depends on $d$ and $n$.
\end{lemma}

Now let $G\in \cO[X_1,\dots, X_n]$
be a   homogeneous polynomial, 
 which is absolutely irreducible over $K$ and has degree  $d\geq 2$. 
The following result  is now a trivial consequence of Lemma \ref{lem:Cohen}
applied to the absolutely irreducible hypersurface $g=0$, where 
$g(\mathbf{X})=G(\a+k\mathbf{X})$.

\begin{lemma}\label{lem:cohen}
Let $k\in \cO$ and let $\a\in \cO^n$. 
Then for any any $\ve>0$ we have 
\[
\#\left\{\x\in \cO^n:  
\begin{array}{l}
| \x|\leq \hat N, 
~G(\x)=0,\\  \x \equiv \a \bmod{k}
\end{array}
\right\}
\ll_{d,n,\ve}  \left(1+\frac{\hat N}{|k|}\right)^{n-3/2+\ve}.
\]
\end{lemma}

The implied constant in this estimate depends at most on $n$, the degree of $G$ and on the choice of $\ve$.
Lemma~\ref{lem:cohen} is an extension of 
\cite[Lemma 15]{hb-10} to function fields. 

We proceed with the proof of Lemmas \ref{lem:trivial} and \ref{lem:Cohen}, which are based on 
Serre's proof of the analogous result for number fields (see Serre \cite[Chapter~13]{serre-book}).
We select  coordinates on $\AA_K^n$ such that the projection  
$\pi:V\rightarrow \AA_K^{m}$ onto the first $m$ coordinates induces a finite morphism. 
Let $Z=\pi(V) $ be the corresponding (thin) subset of $\AA_K^{m}$
and let 
$Z(N)=\#\{\x\in Z\cap\cO^{m}:|\x|\leq \hat N\}$. 
Since the fibre of each point under $\pi$ has at  most $d$ points, it will be
enough to prove the bound
\begin{equation}
\label{eq:Omegabound}
Z(N)=
\begin{cases}
O_{d,n}(q^{(N+1)(m-1/2)}N\log q), &\mbox{if $V$ is abs.~irred.~and $d\geq 2$,}\\
O_{d,n}(q^{(N+1)m}), & \mbox{otherwise.}
\end{cases}
\end{equation}
Our key tool in proving these bounds will be the 
following large sieve inequality over $K$ due to Hsu \cite[Theorem 3.2]{hsu}.

\begin{lemma}
\label{lem:LargeSieve}
Let $M,N,m\in \NN$ and let $X$ be a subset of $\cO^m$.
For each prime $\vp$ suppose that there exists a real number $\alpha_\vp\in (0,1]$
such that $$\#X_\vp\leq \alpha_\vp|\vp|^{m},$$ 
where  $X_\vp$ denotes the canonical image of $X$ in 
$(\cO/\vp\cO)^m$.
Then 
$$\#\{\x\in X: |\x|\leq \hat N\}\leq q^{m(\max\{N,2M-1\}+1)}/L(M),$$
where
 $$
 L(M)=1+\sum_{\substack{b\in\cO \text{  monic }\\ |b|\leq \hat{M}}}
 \prod_{\substack{\vp\mid b}}\left(\frac{1-\alpha_\vp}{\alpha_\vp}\right).
 $$
\end{lemma}

Taking $\alpha_\vp=1$ for every $\vp$ and $L(M)\geq 1$ we easily arrive at the second part 
of \eqref{eq:Omegabound} by taking $M=(N+1)/2$.
For the first part, for any prime $\vp\in \cO$, we let $Z(N)_\vp$ denote the canonical image of $Z$ in $(\cO/\vp\cO)^m$. The following result is proved in  exactly the same way as the number field version  \cite[Thm.~5 in Chapter 13]{serre-book}.

\begin{lemma}\label{lem:finite}
Assume that $V$ is absolutely irreducible and has degree $d\geq 2$.
There is a finite Galois extension $K_\pi/K$ of degree at most $d!$ and a number $c_\pi\in (0,1-1/d!]$ such that if $\vp$ splits completely in $K_\pi$ then 
$$\#Z(N)_\vp\leq c_\pi |\vp|^{m}+O_{d,n}(|\vp|^{m-1/2}).$$
\end{lemma}

We may now use this result to deduce the first part of 
\eqref{eq:Omegabound}.
Let $K_\pi$ be as in Lemma \ref{lem:finite}. Then 
 if a prime $\vp$ splits completely in $K_\pi$ it follows that  $Z(N)_\vp\leq c_\vp|\vp|^{n-1} $
 for some constant $0< 
c_\vp\leq 1-1/d!$. 
We now apply Lemma \ref{lem:LargeSieve} with  $M=(N+1)/2$, invoking the prime number theorem to deduce that 
\begin{align*}
L(M)&\geq \sum_{\substack{\vp\in \cO \text{ monic 
and 
irreducible}\\ |\vp|\leq \hat{M}\\ \vp \text{ splits completely in }K_\pi}}(1-c_\vp)
\geq \frac1{d!}\frac{q^M}{[K_\pi:K]M}
+O(q^{M/2})
,\end{align*}
This completes the proof of  \eqref{eq:Omegabound}, and so the proof of Lemma
\ref{lem:Cohen}.

\section{Global
\texorpdfstring{$L$}{L}-functions and 
\texorpdfstring{$\ell$}{l}-adic sheaves}\label{s:red}

In \S \ref{s:l-adic} we review some facts about $\ell$-adic sheaves on affine curves and in \S \ref{s:correspondence} we recall the construction of their associated  $L$-functions.
In \S \ref{s:weil} we record the statement of the Weil conjectures as established by Deligne.  Next, in \S \ref{s:kahn}, we 
recall the construction of the  Hasse--Weil $L$-function of a smooth and projective variety over a global field of positive characteristic and state some of its fundamental properties.  Finally,  in \S \ref{s:twisting}, we discuss the analagous properties of a global $L$-function obtained through  twisting by a character of finite order.

\subsection{Review of $\ell$-adic sheaves and $\ell$-adic cohomology}\label{s:l-adic}

The main references for this section are Deligne \cite{weil2} and  Katz  (see \cite[Chapter~2]{katz} and \cite{lfunction}).
Let us assume that $j:U\hookrightarrow C$ is a non-empty affine open subset of  a smooth proper geometrically connected curve $C$ over the finite field $\FF_q$.
For any prime  $\ell\nmid q$,  suppose that we are given a lisse $\QQ_\ell$-sheaf $\cF$
on $U$ and let $V$ be the $\QQ_\ell$-vector space associated to  $\cF$ by the monodromy action. 
For any
integer $i\geq 0$ we have  both ordinary and compact cohomology groups
$$
H^i(\bar{U}, \cF) \quad \mbox{ and } \quad H_c^i(\bar{U}, \cF).
$$
These are finite dimensional $\bar \QQ_\ell$-vector spaces on which 
$\Gal(\bar \FF_q/\FF_q)$ acts continuously and which vanish for $i>2$.
There is a natural ``forget supports'' map 
$H_c^i(\bar{U}, \cF) \to H^i(\bar{U}, \cF)$, which need not be an isomorphism (since $U$ is not proper).
We have 
$$
H_c^0(\bar U, \cF)=H^2(\bar U, \cF)=0.
$$
Let  $\eta=\Spec (\FF_q(U))$
 be the generic point and $\bar \eta = \Spec(\bar{\FF_q(U)})$ 
 the geometric point above it.  We denote by $\pi_1^{\mathrm{geom}}=\pi_1(\bar U,\bar \eta)$  the ``geometric'' fundamental group. 
Then 
$$
H^0(\bar U, \cF)=V^{\pi_1^{\mathrm{geom}}}
$$ 
is the subspace of invariants of 
$\pi_1^{\mathrm{geom}}$ acting on $V$, and 
$$
H_c^2(\bar U, \cF)=V(-1)_{\pi_1^{\mathrm{geom}}}
$$
is the
space of coinvariants of 
$\pi_1^{\mathrm{geom}}$ acting on the Tate twist $V(-1)$ of $V$.

We will require information about the dimensions of these cohomology groups.
The Euler characteristic of $\bar U$ is 
\begin{equation}\label{eq:euler}
\chi(\bar U)=2-2g-\sum_{x\in |C\setminus U|}\deg(x),
\end{equation}
where $g$ is the genus of $C$ and the sum is taken  over the closed points  $x\in C\setminus U$, with $\deg(x)$ being the degree of its residue field over $\FF_q$.
Next, the Swan conductor of $\cF$ takes the shape
$$
\sw(\cF)=\sum_{x\in |C\setminus U|}\deg(x)\sw_x(\cF).
$$
It measures 
the wild ramification of the sheaf. 
With this notation to hand, the Euler--Poincar\'e formula (see \cite[\S 2.3.1]{katz}) states that 
$$
\sum_{i=0}^2(-1)^i \dim H_c^i(\bar U,\cF) = \rank(\cF) \chi(\bar U) - \sw(\cF),
$$
whence
$$
\dim H_c^1(\bar U,\cF) = 
\dim H_c^2(\bar U,\cF)
-\rank(\cF) \chi(\bar U) + \sw(\cF).
$$

We may also  form a constructible $\QQ_\ell$-sheaf $j_*\cF$ on $C$, where we recall that $j:U\hookrightarrow C$ is the inclusion map. Its cohomology groups are related to the above groups via the identities
$$
H^i(\bar C, j_*\cF)=\begin{cases}
H^0(\bar U, \cF), &\mbox{if $i=0$,}\\
\mathrm{Im}(H_c^1(\bar{U}, \cF) \to H^1(\bar{U}, \cF)), &\mbox{if $i=1$,}\\
H^2_c(\bar U, \cF), &\mbox{if $i=2$}.
\end{cases}
$$
Fix an embedding $\iota:\bar \QQ_\ell\hookrightarrow \CC$ and 
suppose that $\cF$ is $\iota$-pure of weight $w$.
Then, by the fundamental work of Deligne \cite[Thm.~3.2.3]{weil2}, it follows that
$H^i(\bar C,j_*\cF) $ is $\iota$-pure of weight $w+i$ for each integer $0\leq i\leq 2$.

It follows from the facts above that 
\begin{equation}\label{eq:degree1}
\dim H^i(\bar C, j_*\cF)\leq \rank(\cF) \quad \mbox{ for $i=0,2$}
\end{equation}
and 
\begin{equation}\label{eq:degree2}
\dim H^1(\bar C, j_*\cF)\leq \rank(\cF)(1-\chi(\bar U) )  +\sw(\cF).
\end{equation}
Now suppose that 
$\cF_1$ and $\cF_2$ are  lisse $\QQ_\ell$-sheaves on $U$, with 
 $r_i=\rank(\cF_i)$ for $i=1,2$. 
Arguing as in the
work of Fouvry, Kowalski and Michel (see the proof of \cite[Prop.~8.2(2)]{FKM}), one finds that
$$
\sw(\cF_1\otimes \cF_2)\leq r_1r_2\left(\sw(\cF_1)+\sw(\cF_2)\right).
$$
It therefore follows from \eqref{eq:degree2} that 
\begin{equation}\label{eq:degree2'}
\dim H^1(\bar C, j_*(\cF_1\otimes \cF_2))\leq r_1r_2\left\{1-\chi(\bar U)   +\sw(\cF_1)+\sw(\cF_2)\right\}.
\end{equation}

\subsection{Global $L$-functions}\label{s:correspondence}

Let $\cF$ be a $\iota$-pure lisse $\QQ_\ell$-sheaf of weight $w$ on an open subset $j:U\hookrightarrow C$.
In the 1960s, Grothendieck \cite{bourbaki} associated a global $L$-function 
$L(C,j_*\cF,T)$
to the constructible $\QQ_\ell$-sheaf $j_*\cF$. It follows from
the correspondence of 30/09/64 in 
 \cite{letters} (see also    \cite[Conj.~$C_9$]{serre}), that this $L$-function is a rational function, with 
  \begin{equation}\label{eq:grothendieck}
L(C,j_*\cF,T)=\frac{P_1(T)}{P_0(T)P_2(T)},
\end{equation}
where 
$P_0,P_1,P_2\in \ZZ[T]$ are polynomials given by 
$$
P_i(T)=\det\left(1-T \fr_q\mid H^i(\bar C,j_*\cF)\right)
$$
for $0\leq i\leq 2$.
Here 
$\fr_q$  is the  Frobenius endomorphism 
acting on $H^i(\bar C,j_*\cF)$. 
It follows from Deligne \cite{weil2} that the inverse roots of $P_i$ have modulus $q^{(w+i)/2}$.

\subsection{The Weil conjectures}\label{s:weil}

Let $V$ be a smooth and projective  variety 
of dimension $m$ which is 
defined over a finite field $\FF_q$.
Then $V$ is also defined over any extension $\FF_{q^r}$ of $\FF_q$ and 
 we may define the zeta function
$$
Z(V,T)=\exp \left(\sum_{r=1}^\infty \frac{\#V(\FF_{q^r})T^r}{r}\right).
$$
According to Deligne \cite{weil1} and his resolution of the Weil conjectures, the zeta function can be expressed as a rational function 
\begin{equation}\label{eq:zeta}
Z(V,T)=\frac{P_1(V,T)P_3(V,T)\dots P_{2m-1}(V,T)}{P_0(V,T)P_2(V,T)\dots P_{2m}(V,T)},
\end{equation}
where
\begin{equation}\label{eq:poly}
P_i(V,T)=\det \left(1-T \fr_q \mid 
\het^i(\bar V,\QQ_\ell)\right),
\end{equation}
for  $i\in \{0,\dots,2m\}$
 and any prime $\ell\nmid q$. Here $\fr_q$ is the  Frobenius endomorphism  acting on 
$\het^i(\bar V,\QQ_\ell)$ induced by the Frobenius map on $\bar V$. 
Note that if one takes $\cF=\QQ_\ell$ to be the trivial sheaf in \S \ref{s:correspondence}
then $Z(C,T)=L(C,j_*\cF,T)$.

There is a factorisation 
\begin{equation}\label{eq:factor}
P_i(V,T)=
\prod_{j=1}^{b_{i,\ell}} (1-\omega_{i,j}T),
\end{equation}
where $b_{i,\ell}=\dim_{\QQ_\ell} \het^i(\bar V,\QQ_\ell)$. Deligne
shows that each
$\omega_{i,j}$ is an  algebraic integer with the property that $|\omega_{i,j}|=q^{i/2}$, for $1\leq j\leq b_{i,\ell}$ and $0\leq i\leq 2m$.
A formal consequence of \eqref{eq:zeta} and  \eqref{eq:factor} is the identity
\begin{equation}\label{eq:counting}
\#V(\FF_{q^r})=\sum_{i=0}^{2m}(-1)^i \sum_{j=1}^{b_{i,\ell}} \omega_{i,j}^r,
\end{equation}
to which we will return in due course.

\subsection{Global $L$-functions once again}\label{s:kahn}

Let $X$ be a smooth and projective variety of dimension $m$ defined over $K=\FF_q(C)$
and let $\bar X=X\otimes_{K} \bar K$. We will need to work with models for $X$ over the ring of integers $\cO$ of $K$.
Let $S\subset \Omega$ denote the finite set of places outside of which $X$ has good reduction.
The smooth projective morphism $X\to \Spec(K)$ extends to a smooth projective morphism $p:\mathcal{X}\to U$,  for 
a suitable open subset $U$ of $C$.  (This corresponds to choosing a specific equation over $\cO$ which has good reduction at the primes outside of $S$.)
For any 
$v\in \Omega\setminus S$, we let $\mathcal{X}_v$ be the special fibre at $v$ of $\mathcal{X}$ over $\mathcal{O}_v$.
Then $\mathcal{X}_v$ is a smooth and projective $\mathcal{O}_v$-scheme 
such that $\mathcal{X}_v \otimes_{\mathcal{O}_v} K_v$ can be identified with $X\otimes_K K_v$. 
We denote by $X_v=\mathcal{X}_v \otimes_{\mathcal{O}_v} \FF_v$ the reduction  at $v$. This is a smooth and projective variety defined over the finite field $\FF_v$. 
For any prime  $\ell\nmid q$ it will be convenient to put
$H_\ell^i(X)=\het^i(\bar X,\QQ_\ell)$ for the geometric $\ell$-adic cohomology group.

In this section,  following Serre  \cite{serre}, we define some global $L$-functions 
associated to $X$ and discuss their analytic properties. 
Let  $i\in \{0,\dots,2m\}$.  For   $v\in \Omega$,
Serre defines the local factor 
$$
L_v(H_\ell^{i}(X),s)=\det \left(1-\#\FF_v^{-s} \fr_v \mid 
H_\ell^i(X)^{I_v}\right)^{-1},
$$
where $I_v$ is the inertia group of $v$
and $\fr_v$ is the geometric Frobenius endomorphism at $v$. 
Let $P_{i,v}(T)=
\det \left(1-T \fr_v \mid 
H_\ell^i(X)^{I_v}\right).$
When  $v\not\in S$ this coincides with the polynomial $P_i(X_v,T)$ that we met in 
\eqref{eq:poly}. 
 For arbitrary $v\in \Omega$, it 
follows from 
Terasoma \cite{tera}
that $P_{i,v}(T)$ 
is independent of the choice of $\ell$ and 
from Deligne \cite[Thm.~1.8.4]{weil2} that its inverse roots 
have absolute value at most $q^{i/2}$.
When $v\not \in S$  we then have 
$$
Z(X_v,\#\FF_v^{-s})=\prod_{i=0}^{2m}
L_v(H_\ell^{i}(X),s)^{(-1)^i}.
$$

For any $i\in \{0,\dots,2m\}$,
Serre \cite{serre} defines the global $L$-function 
$$
L(H_\ell^i(X),s)=
\prod_{v\in \Omega} L_v(H_\ell^i(X),s).
$$
This $L$-function satisfies a functional equation.  
Associated to $X$ is a smooth model  $p:\mathcal{X}\to U$, for a suitable open subset $U$ of $C$. If $j_U:U 
\hookrightarrow C$ is the corresponding immersion then we obtain a 
lisse  $\QQ_\ell$-sheaf
$$
j_*H_\ell^i(X)=(j_U)_*R^ip_*\QQ_\ell,
$$
where $j:\Spec(K) \to C$ is the inclusion of the generic point.
According to Grothendieck  \cite{letters} (see also \cite[\S 5.5]{kahn}) we then have 
$$
L(H_\ell^i(X),s)=L(C,j_*H_\ell^i(X), q^{-s}),
$$
 in the notation of 
\S \ref{s:correspondence}. Hence it follows from \eqref{eq:grothendieck} that 
\begin{equation}\label{eq:YEAH}
L(H_\ell^i(X),s)=\frac{P_{1,i}(q^{-s})}{P_{0,i}(q^{-s})P_{2,i}(q^{-s})},
\end{equation}
where 
for $k\in \{0,1,2\}$ one has
\begin{equation}\label{eq:def-Pj}
P_{k,i}(T)=
 \det \left(1-T \fr_q \mid 
H^k(\bar{ C},j_* H_\ell^i(X))\right)\in \ZZ[T],
\end{equation}
with inverse roots  having absolute value $q^{(i+k)/2}$.
In particular, any poles or zeros of $L(H_\ell^i(X),s)$  must have 
$\Re(s)=(i+k)/2$ for $k\in \{0,1,2\}$.

We now specialise the previous discussion to the case of a smooth
hypersurface $X\subset \PP_K^{m+1}$  of degree $d$.
As before, let $S$ be the finite set of places outside of which $X$ has good reduction and choose a smooth model  
$p:\mathcal{X}\to U$, for a suitable open subset  $U$ in $C$, which  we consider fixed once and for all. 
We define the 
discriminant $\Delta_X$ of $X$ to be the classical discriminant 
of the degree $d$ form $F\in \cO[x_0,\dots,x_m]$ that defines $\mathcal{X}$. (See
Example 4.15 of  \cite[Chap.~1]{GKZ} for its construction.) Thus 
 $\Delta_X$ is a (non-zero) polynomial of degree $(m+1)3^m$ in the coefficients of $F$. In particular 
$\Delta_X \in \cO$ and its prime divisors correspond to the finite places in $S$.

Let $X_v$ be the reduction of $X$ at any  $v\in \Omega\setminus S$. 
The middle cohomology group $H_\ell^m(X)$ is the only one of interest to us, since  
$$
\het^i(\bar X_v,\QQ_\ell)=
\begin{cases}
\QQ_\ell(-i/2), & \mbox{if $i$ is even and $i\neq m$,}\\
0, & \mbox{if $i$ is odd and $i\neq m$},
\end{cases}
$$
for any  $v\in \Omega\setminus S $ and $i\in \{0,\dots,2m\}$ (see Ghorpade and Lachaud \cite[\S 3]{G-L}, for example).
It then follows from 
 \eqref{eq:counting} that 
\begin{equation}\label{eq:counting'}
\#X_v	(\FF_{v})=
\#\FF_v^m+\#\FF_v^{m-1}+\dots+1 
+(-1)^m \sum_{j=1}^{b_{m}} \omega_{m,j},
\end{equation}
where $b_{m}=\dim_{\QQ_\ell} H_\ell^m(X)$
is a positive integer that depends only on $d$ and $m$ (it does not depend on the choice of $\ell$), 
and
 $\omega_{m,j}$ are the eigenvalues of the Frobenius endomorphism at $v$ on  $H_\ell^m(X)$, satisfying 
 $|\omega_{m,j}|=\#\FF_v^{m/2}$ for $1\leq j\leq b_{m}$.

Taking $i=m$ we will need to control the degrees of the polynomials $P_{k,m}(T)$ appearing in \eqref{eq:def-Pj}.
The closed points $x\in |C\setminus U|$ correspond to the prime divisors $\varpi$ of the discriminant $\Delta_X$ that was defined above. Hence \eqref{eq:euler} yields
$$
-\chi(\bar U)=2g-2+O_{d,m}(\log |\Delta_X|),
$$
where $g$ is the genus of $C$.
Moreover, as is implicit
in work of Hooley \cite[\S 6]{hooley-waring}, we have 
$\sw(H_\ell^m(X))=O_{d,m}(\log |\Delta_X|)$,
since $\sw_x(H_\ell^m(X))$ can be bounded uniformly in terms of $d$ and $m$ for any closed point $x\in |C\setminus U|.$
Combining \eqref{eq:degree1} and \eqref{eq:degree2}, we deduce that 
$$
\deg P_{0,m}\leq b_m, \quad 
\deg P_{1,m}=O_{d,m,g}\left( 1+\log |\Delta_X|\right),
\quad
\deg P_{2,m}\leq b_m. 
$$

\subsection{Twisting by a character}\label{s:twisting}

For the sake of simplicity, in this section we shall restrict attention to the case $K=\FF_q(t)$  (so that 
$C=\PP^1$ and $g=0$).
We continue to assume that  $X\subset \PP_K^{m+1}$ is a smooth
hypersurface  of degree $d$ with associated set $S\subset \Omega$  of places outside of which $X$ has bad 
reduction. We let $p:\mathcal{X}\to U$ be a smooth model, which is fixed once and for all, and we let $\Delta_X\in \cO$ denote the corresponding discriminant. 
We need to  consider the effect of twisting the middle cohomology group $H_\ell^{m}(X)$ by a fixed character of finite order.

Let  $N\in \ZZ_{>0}$ and  let 
$$
\chi_{\mathrm{Dir}}: (\mathcal{O}_\infty/ t^{-N} \mathcal{O}_\infty)^* \to \CC^*
$$ 
be a Dirichlet character.
Putting $x=t^{-1}$ and 
$A=\FF_q[x]$, we note that 
$(\mathcal{O}_\infty/ t^{-N} \mathcal{O}_\infty)^*\cong (A/ x^{N} A)^*$.
This lifts to a character $\chi_{\mathrm{Dir}}:\mathcal{O}_\infty^*\to\CC^*$.
Given any id\`ele $y=(y_\varpi)\in I_K$, we may suppose that $y_\varpi=u_\varpi \varpi^{e_\varpi}$ for $u_\varpi\in \mathcal{O}_\varpi^*$ and $e_\varpi\in \ZZ$ such that 
$e_\varpi=0$ for almost all $\varpi$.
Putting  $a=\prod_\varpi \varpi^{-e_\varpi}\in K^*$, we then have a unique representation 
$y=au$ where $u=(u_\varpi)\in \prod_\varpi  \mathcal{O}_\varpi^*$ for every prime $\varpi$.
We may now define a Hecke character
$\chi_{\mathrm{Hecke}}: I_K\to \CC^*$ via
$$
\chi_{\mathrm{Hecke}}(au)=\chi_{\mathrm{Dir}}(u_\infty).
$$
It is  constant on $K^*$ and gives a character on the id\`ele class group $I_K/K^*$.

There are two  relevant multiplicative characters in our investigation. The first is  $\eta:\mathcal{O}\rightarrow \CC^*$, given by  
$$
\eta(r)=\chi_{\mathrm{Dir}}(r/t^{\deg r} )
$$
for any $r\in \mathcal{O}$.
Note that 
$r/t^{\deg r}\in \mathcal{O}_\infty^*$ for any $r\in \mathcal{O}$.
The second is a 
Dirichlet character $\eta':(\cO/M\cO)^*\to \CC^*$ modulo $M$, for given $M\in \cO$ which in our application will have  bounded absolute value. 
By class field theory one can view $\eta$ and $\eta'$ as 
lisse   $\QQ_\ell$-sheafs
on $U$ of rank $1$, both of which are $\iota$-pure of weight $0$.
The character $\eta$ is ramified only at infinity and 
$\eta'$ is ramified only at the primes dividing $M$. One has 
$\sw(\eta)=O(N)$ and $\sw(\eta')=O_{|M|}(1)$.

We may now define the  global 
$L$-function 
$L(\eta\otimes \eta' \otimes H_\ell^m(X),s)$,
with local factors
$$
L_v(\eta \otimes\eta'\otimes H_\ell^{m}(X),s)=\det \left(1-\#\FF_v^{-s} \fr_v \mid 
\eta\otimes\eta' \otimes H_\ell^m(X)^{I_v}\right)^{-1},
$$
for $v\in \Omega$.
As before, we  have 
$$
L(\eta\otimes\eta'\otimes H_\ell^m(X),s)=L(\PP^1,j_*(\eta \otimes\eta'\otimes H_\ell^m(X)), q^{-s}),
$$
where 
if  $j:\Spec(K) \to \PP^1$ is the inclusion of the generic point then 
$$
j_*(\eta \otimes\eta'\otimes H_\ell^m(X))=(j_U)_*(\eta \otimes \eta' \otimes R^ip_*\QQ_\ell).
$$
Moreover, the analogue of \eqref{eq:YEAH} holds true. 
Thus 
\begin{equation}\label{eq:YEAH-twist}
L(\eta \otimes\eta'\otimes H_\ell^m(X),s)=\frac{P_{1,m}(q^{-s})}{P_{0,m}(q^{-s})P_{2,m}(q^{-s})},
\end{equation}
where $P_{k,m}\in \ZZ[T]$ 
for $k\in \{0,1,2\}$, 
with inverse roots having
 absolute value $q^{(m+k)/2}$.
Finally, using 
\eqref{eq:degree2'}
and noting that $\rank (\eta \otimes\eta'\otimes H_\ell^m(X))\leq b_m$, we have 
\begin{equation}\label{eq:degree-Hm-twist}
\deg P_{0,m}\leq b_m, \quad 
\hspace{-0.2cm}
\deg P_{1,m}=O_{d,m,|M|}\left(\log |\Delta_X|+N\right),
\quad
\hspace{-0.2cm}
\deg P_{2,m}\leq b_m. 
\end{equation}
In our work it is the reciprocal of 
$L(\eta \otimes\eta'\otimes H_\ell^m(X),s)$
that features and so the location of its poles is dictated by the 
zeros of $P_{1,m}(q^{-s})$ in \eqref{eq:YEAH-twist}.

\section{Activation of the circle method over  function fields}\label{s:active}

We suppose that we are given a form $F\in \cO[x_1,\dots,x_n]$ of degree $d\geq
2$ together with a vector $\b\in \cO^n$ and an element $M\in \cO$ such that $M\mid F(\b)$.
Let  $\omega\in S(K_\infty^n)$ be  a weight function.
Then, for $P\in \cO$ we consider the counting function
$$
N(P)=\sum_{\substack{\x\in \cO^n\\ F(\x)=0\\
\x\equiv \b\bmod{M}}} \omega(\x/P).
$$
We are interested in the behaviour of this as $|P|\rightarrow \infty$, for fixed
$M$ and $\b$.
According to \eqref{eq:orthog} we may write
$$
N(P)=\int_{\TT} S(\alpha) \d \alpha,
$$
where
$$
S(\alpha)=\sum_{\substack{\x\in \cO^n\\
\x\equiv \b\bmod{M}}
} \psi(\alpha F(\x))\omega(\x/P).
$$
We would like to dissect $\TT$ into a disjoint union of intervals in order to
try and use non-trivial averaging in our estimation of $S(\alpha).$
The starting point for this is the following analogue of Dirichlet's approximation
theorem (as proved in
\cite[Lemma 3]{kubota} or  \cite[Lemma 5.1]{lee}, for example).

\begin{lemma}
Let $\alpha\in
K_\infty$ and let $Q>1$. 
Then there exists coprime $a,r\in \mathcal{O}$, with $r$ monic,  such that
$|a|<|r|\leq \hat Q$ and 
$$
|r\alpha-a|< \hat Q^{-1}.
$$
\end{lemma}

For any  $Q>1$ this result allows one to partition $\TT$ into 
a union of intervals centred at rationals $a/r$. The non-archimedean nature of $K$ ensures that the intervals are actually non-overlapping, as follows.

\begin{lemma}\label{lem:dissection}
For any  $Q>1$  we have a disjoint union
$$
\TT=\bigsqcup_{\substack{r\in \cO\\ |r|\leq \hat Q\\
\text{$r$ monic}}}
\bigsqcup_{\substack{a\in \cO\\ |a|<|r|\\ (a,r)=1} }
\left\{\alpha\in \TT: |r\alpha-a|<\hat Q^{-1}\right\}.
$$
\end{lemma}

\begin{proof}
Suppose that there exists $\alpha\in \TT$ belonging to two distinct intervals
associated to  $a/r\neq a'/r'$, say.
Then by the ultrametric inequality we have 
$$
\left|
\frac{a}{r}-\frac{a'}{r'}
\right|
\leq \max\left\{
\left|
\frac{a}{r}-\alpha
\right|, 
\left|
\frac{a'}{r'}-\alpha
\right|\right\}<\frac{1}{\hQ\min\{|r|,|r'|\}}.
$$
On the other hand, since $ar'-a'r$ is a non-zero element of $\cO$, we have 
$$
\left|
\frac{a}{r}-\frac{a'}{r'}
\right|\geq \frac{1}{|rr'|}\geq \frac{1}{\hQ\min\{|r|,|r'|\}}.
$$
This is a contradiction, which thereby establishes the lemma. 
\end{proof}

It follows from Lemma \ref{lem:dissection} that
\begin{equation}\label{eq:blue}
N(P)=
\sum_{\substack{
r\in \cO\\
|r|\leq \hat Q\\
\text{$r$ monic}
} }
\sumstar_{\substack{
|a|<|r|} }
\int_{|\theta|<|r|^{-1}\hQ^{-1}} S\left(\frac{a}{r}+\theta\right) \d \theta,
\end{equation}
where we henceforth put
$$
\sumstar_{\substack{
|a|<|r| }}
=\sum_{\substack{
a\in \cO\\ 
|a|<|r|\\ (a,r)=1} }.
$$

\begin{remark}\label{rem:dissection}
The reader will note that there is no division into major and
minor arcs in our expression for $N(P)$. 
In the classical setting over $\QQ$ this would correspond to the opening steps of a  Kloosterman
refinement, a device which is rendered essentially trivial over function fields.
\end{remark}

We may write
$$
S\left(\frac{a}{r}+\theta\right)=\sum_{
\substack{\y\in \cO^n\\|\y|<|r_M|\\
\y\equiv \b \bmod{M}
}} \psi\left(\frac{aF(\y)}{r}\right)
\sum_{\substack{\z\in \cO^n}}\psi(\theta F(\y+r_M\z))
\omega\left(\frac{\y+r_M\z}{P}\right),
$$
where $r_M=rM/(r,M)$
 is the least common multiple of $r$ and $M$.
We evaluate the inner sum over $\z$ using Poisson summation. Thus Lemma~\ref{lem:poisson} implies that 
\begin{align*}
\sum_{\substack{\z\in \cO^n}}
&\psi(\theta F(\y+r_M\z))
\omega\left(\frac{\y+r_M\z}{P}\right)\\
&=
\sum_{\substack{ \c\in \cO^n}} \int_{K_\infty^n}
\omega\left(\frac{\y+r_M\u}{P}\right)
\psi (\theta F(\y+r_M\u)+\c.\u)\d \u.
\end{align*}
Making the change of variables $\x=(\y+r_M\u)P$, it follows from Lemma
\ref{lem:change}  (together with the fact that the measure on $K_\infty^n$ is translation
invariant)
that the right hand side is
\begin{align*}
\left|\frac{P}{r_M}\right|^n
&\sum_{\substack{ \c\in \cO^n}} 
\psi\left(\frac{-\c.\y}{r_M}\right)
\int_{K_\infty^n} \omega(\x)
\psi \left(\theta P^dF(\x)+\frac{P \c.\x}{r_M}\right)\d \x.
\end{align*}
Putting everything together in \eqref{eq:blue}, we may now establish the following result. 

\begin{lemma}\label{lem:summary}
We have 
$$
N(P)=|P|^n
\sum_{\substack{
r\in \cO\\
|r|\leq \hat Q\\
\text{$r$ monic}} }
|r_M|^{-n}
\int_{|\theta|<|r|^{-1}\hQ^{-1}} 
\sum_{\substack{ \c\in \cO^n
}} 
S_{r,M,\b}(\c)
I_{r_M}(\theta;\c)
 \d \theta,
$$
where 
$r_M=rM/(r,M)$ and 
\begin{align*}
S_{r,M,\b}(\c)&=
\sumstar_{\substack{
|a|<|r|} }
\sum_{\substack{\y\in \cO^n\\|\y|<|r_M|\\
\y\equiv \b\bmod{M}
}} 
\psi\left(\frac{aF(\y)}{r}\right)
\psi\left(\frac{-\c.\y}{r_M}\right),\\
I_{s}(\theta;\c)&=
\int_{K_\infty^n} \omega(\x)
\psi \left(\theta P^dF(\x)+\frac{P \c.\x}{s}\right)\d \x.
\end{align*}
\end{lemma}

The exponential integrals can be estimated using the results in 
\S \ref{s:integrals} provided that the weight function $\omega$ is chosen suitably. The 
exponential sums $S_{r,M,\b}(\c)$ satisfy the following  basic multiplicativity property. 

\begin{lemma}\label{lem:multi2}
Let $r=r_1r_2$ for coprime $r_1,r_2\in \cO$. Let $M=M_1M_2M_3 $ for $M_1,M_2,M_3\in \cO$ such that 
$M_1\mid r_1^\infty$, $M_2\mid r_2^\infty$ and $(M_3,r)=1$.
Then there exists $\b_1,\b_2,\b_3\in (\cO/M\cO)^n$,
depending on $\b, M$ and the residue of $r_1,r_2$ modulo $M$, 
 such that 
$$
S_{r,M,\b}(\c)
=
S_{r_1,M_1,\b_1}(\c)
S_{r_2,M_2,\b_2}(\c)
\psi\left(\frac{-\c.\b_3}{M_3}\right).
$$
\end{lemma}

\begin{proof}
Let us put $s_i=r_iM_i/(r_i,M_i)$ for $i=1,2$. Then $s_1,s_2,M_3$ are pairwise coprime and we have a factorisation 
$r_M=s_1s_2M_3$.
As $\y_1$ ranges over vectors  modulo $s_1$,  
$\y_2$ ranges   modulo $s_2$, 
and 
$\y_3$ ranges modulo $M_3$, 
so  the vector $$
\y=s_2M_3 \y_1+s_1M_3 \y_2+s_1s_2 \y_3
$$ 
ranges over a complete set of 
residues modulo $r_M$. Likewise, 
as $a_1$ (resp.~ $a_2$) ranges over elements of $\cO$ modulo $r_1$ (resp.~ modulo $r_2$),  which are coprime to $r_1$ (resp.~ $r_2$), so 
 $a=r_2 a_1+r_1a_2$ ranges over a complete set of residues modulo $r$, which are coprime to $r$.
It is now clear that 
$$
\psi\left(\frac{aF(\y)}{r}\right)=
\psi\left(\frac{a_1 s_2^3M_3^3 F(\y_1)}{r_1}\right)
\psi\left(\frac{a_2 s_1^3M_3^3 F(\y_2)}{r_2}\right)
$$
and
$$
\psi\left(\frac{-\c.\y}{r_M}\right)=
\psi\left(\frac{-\c.\y_1}{s_1}\right)
\psi\left(\frac{-\c.\y_2}{s_2}\right)
\psi\left(\frac{-\c.\y_3}{M_3}\right).
$$
Choose $t_1,t_2,t_3\in \cO$ such that 
$t_1M_3s_2\equiv 1\bmod{M_1}$, 
$t_2M_3s_1\equiv 1\bmod{M_2}$
and 
$t_3s_1s_2\equiv 1\bmod{M_3}$.
Then it is clear that the statement of the lemma holds with 
$
\b_i=t_i\b\bmod{M}$, for $1\leq i\leq 3$.
\end{proof}

The importance of 
Lemma \ref{lem:multi2} is that it allows us to factorise the exponential sum in which we are interested, so that it suffices to examine the sum at the prime power moduli. When piecing these together it will be important to bear in mind the following convention that will henceforth be adopted.

\begin{definition}\label{def:r}
Associated to any $r\in \cO$ and  $i\in \NN$ are the elements 
$$
b_i
=\prod_{\substack{\vp^i \| r} }\vp^i, \quad
k_i=\prod_{\substack{\vp^i \| r} }\vp, \quad r_{i}=\prod_{\substack{\vp^e\| r\\ e\geq i}}\vp^{e},
$$
in $\cO$.
In particular, 
for any $j\in \NN$ we have the factorisation 
$$
r=r_{j+1}\prod_{i=1}^{j}b_i=r_{j+1}\prod_{i=1}^{j}k_i^{i}, \quad (\text{with  $(j+1)$-full $r_{j+1}$}).
$$
\end{definition}

\section{Cubic exponential sums: basic estimates}\label{s:start-cubic}

We now specialise to the case of non-singular cubic forms
 $F\in \cO[x_1,\dots,x_n]$ under the hypothesis that 
$\ch(\FF_q)>3$.
We define the associated Hessian matrix 
\begin{equation}\label{eq:hessian}
\mathbf{H}(\x)=\left(\frac{\partial^2 F}{\partial x_i\partial x_j}\right)_{1\leq i,j\leq n}.
\end{equation}
Our assumption on the characteristic of $\FF_q$ ensures that this matrix doesn't vanish identically.
Of special importance to us will be the dual form 
$$F^*\in\cO[x_1,\dots,x_n],
$$ 
whose zero locus parameterises the set of
hyperplanes whose intersection with the cubic hypersurface $F=0$ produce a singular variety. It is
well-known that $F^*$ is absolutely irreducible and has degree $3\cdot 2^{n-2}$. 

This section is devoted to a suite of estimates for the 
complete cubic exponential sum
$$
S_{r,M,\b}(\c)=
\sumstar_{\substack{
|a|<|r|} }
\sum_{\substack{\y\in \cO^n\\|\y|<|r_M|\\
\y\equiv \b\bmod{M}
}} 
\psi\left(\frac{aF(\y)}{r}\right)
\psi\left(\frac{-\c.\y}{r_M}\right),
$$
both pointwise and on average over $\c$.
We recall  that $r,M\in \cO$ and   $\b,\c\in \cO^n$, with  $M\mid F(\b)$
and  $r_M=rM/(r,M)$. 

We begin by  focusing our attention on  the exponential sum 
$S_{\vp^\alpha,M,\b}(\c)$ for a prime $\vp$ and an integer $\alpha\geq 1$.
We will typically do so for large primes.  
In particular we will have $\vp\nmid M$ for all of the primes considered in this section, so that 
$$
S_{\vp^\alpha,M,\b}(\c)=S_{\vp^\alpha,1,\0}(\c)=S_{\vp^\alpha}(\c),
$$
say. 

\subsection*{The cases $\alpha\in \{1,2\}$}

Suppose that $\vp\nmid M$ and $\alpha\in \{1,2\}$. 
Then 
\begin{align*}
 S_{\vp^\alpha}(\c)=\sumstar_{\substack{
|a|<|\vp^\alpha|} }
\sum_{\substack{|\y|< |\vp^\alpha|
}} 
\psi\left(\frac{aF(\y)-\c.\y}{\vp^\alpha}\right).
\end{align*}
When  $\alpha=1$ it follows from 
 Heath-Brown \cite[Lemma~12]{hb-10} and 
 Hooley
\cite[Lemma~60]{nonaryIII}
that there is a constant $A(n,|\Delta_F|)>0$ depending only on $n$ and $|\Delta_F|$
such that 
\begin{equation}\label{eq:alpha=1}
S_{\vp}(\c)\leq A(n,|\Delta_F|)  |\vp|^{(n+1)/2}|(\vp,\nabla F^*(\c))|^{1/2}.
\end{equation}
These estimates are founded  on the work of Deligne \cite{weil1}.

Suppose next that $\alpha=2$. We write
$a=a_1+\vp a_2 $ and $\y=\y_1+\vp\y_2$, for $a_i,\y_i$ running modulo $\vp$.
Then 
\begin{align*}
 S_{\vp^2}(\c) 
=\sumstar_{|a_1|<|\vp|}&\sum_{\substack{
|a_2|<|\vp|} }
\sum_{\substack{|\y_1|< |\vp|}} 
\psi\left(\frac{
a_1F(\y_1)-\c.\y_1}{\vp^2}\right)  \\
&\times \sum_{\substack{|\y_2|< |\vp|}} 
\psi\left(\frac{(a_1\nabla F(\y_1)-\c).\y_2+a_2 F(\y_1)}{\vp}\right).
\end{align*}
The inner sum over $\y_2$ vanishes unless 
$
a_1\nabla F(\y_1)\equiv \c\bmod \vp.
$
Likewise, the sum over $a_2$ vanishes unless 
$ \vp\mid F(\y_1)$. It follows that 
$$
|S_{\vp^2,M,\b}(\c)|\leq |\vp|^{n+1}N
$$
where
$N$ is the 
number of $a_1, \y_1 \bmod \vp$ such that 
$a_1\nabla F(\y_1)\equiv \c\bmod
\vp$ and 
$(a_1,\vp)=1$ 
 and $\vp\mid F(\y_1)$.
But this is now a problem about point counting over finite fields and the argument used by Hooley \cite[Lemma 11]{nonaryIII} 
yields
$N=0$ if $\vp\nmid F^*(\c)$ and 
$N=O(|\vp|)$ otherwise.
This therefore shows that there is a constant $A(n,|\Delta_F|)>0$ such that 
\begin{equation}\label{eq:alpha=2}
S_{\vp^2}(\c)\leq A(n,|\Delta_F|)  |\vp|^{n+1}|(\vp, F^*(\c))|.
\end{equation}

Recalling the notation in Definition \ref{def:r}, we may now combine   Lemma \ref{lem:multi2} 
with \eqref{eq:alpha=1} and \eqref{eq:alpha=2} to deduce the following result.

\begin{lemma}\label{lem:62}
There is a constant $A=A(n,|\Delta_F|)>0$ such that
$$
|S_{b_1b_2,M,\b}(\c)|\leq  A^{\omega(b_1b_2)} |b_1b_2|^{(n+1)/2}|(b_1,\nabla F^*(\c))|^{1/2}|(k_2, F^*(\c))|,
$$
uniformly in $\b\in \cO^n$.
\end{lemma}

\subsection*{The case $\alpha>2$}

Suppose that  $\vp\nmid M$ and that $\alpha>2$ is an integer. 
Evaluating the sum over $a$, we begin by noting that
\begin{align*}
 S_{\vp^\alpha}(\c)
 &=\sum_{\substack{\y\in \cO^n\\|\y|<|\vp|^\alpha
}} 
\hspace{-0.2cm}
\psi\left(\frac{-\c.\y}{\vp^\alpha}\right)
\left(\sum_{\substack{
|a_1|<|\vp|^\alpha} }
\hspace{-0.2cm}
\psi\left(\frac{a_1F(\y)}{\vp^\alpha}\right)-\sum_{\substack{
|a_2|<|\vp|^{\alpha-1}} }
\hspace{-0.2cm}
\psi\left(\frac{a_2F(\y)}{\vp^{\alpha-1}}\right)\right)\\
&=|\vp|^\alpha\sum_{\substack{\y\in \cO^n\\|\y|<|\vp|^\alpha
\\ F(\y)\equiv 0\bmod{\vp^\alpha}}} 
\psi\left(\frac{\c.\y}{\vp^\alpha}\right)-|\vp|^{\alpha-1}\sum_{\substack{\y\in \cO^n\\|\y|<|\vp|^\alpha
\\ F(\y)\equiv 0\bmod{\vp^{\alpha-1}}}}
\psi\left(\frac{\c.\y}{\vp^\alpha}\right)\\
&=|\vp|^\alpha S_1(\vp^\alpha,\c)-|\vp|^{\alpha-1}S_2(\vp^{\alpha},\c),
\end{align*}
say.
Substituting $\y=\y_1+\vp^{\alpha-1}\y_2 $, we get 
\begin{align*}
 S_2(\vp^{\alpha},\c)=\sum_{\substack{\y_1\in \cO^n\\|\y_1|<|\vp|^{\alpha-1}
\\ F(\y_1)\equiv 0\bmod{\vp^{\alpha-1}}}}
\psi\left(\frac{\c.\y_1}{\vp^\alpha}\right)\sum_{\substack{\y_2\in
\cO^n\\|\y_2|<|\vp|}}\psi\left(\frac{\c.\y_2}{\vp}\right).
\end{align*}
This term clearly vanishes if $\vp\nmid \c$. Therefore 
\begin{align*}
 S_{\vp^\alpha}(\c)=|\vp|^{\alpha}S_1(\vp^{\alpha},\c), \quad  \text{ if $\alpha> 1$ and 
 $\vp\nmid \c $}.
\end{align*}
For $\alpha>1$ and $\vp\nmid\c$, the argument in \cite[\S 6]{oct} goes through to give
\begin{align*}
 S_{\vp^\alpha}(\c)=\frac{|\vp|}{|\vp|-1}\left\{|\vp|^\alpha\nu_1(\vp^\alpha,\c)-|\vp|^{\alpha-1}
\nu_2(\vp^\alpha,\c)\right\},
\end{align*}
where $\nu_1(\vp^\alpha,\c)$  denotes the number of incongruent solutions modulo $\vp^\alpha$ of
the  conditions 
\begin{align*}
 F(\y)\equiv 0\bmod{\vp^\alpha}, \quad  \c.\y\equiv 0\bmod{\vp^\alpha}, \quad  \y\not\equiv \0\bmod{\vp},
\end{align*}
whereas
 $\nu_2(\vp^\alpha,\c)$ is the number 
of  solutions modulo ${\vp^\alpha}$ of
\begin{align*}
 F(\y)\equiv 0\bmod{\vp^\alpha},\quad  \c.\y\equiv 0\bmod{\vp^{\alpha-1}},\quad  \y\not\equiv \0\bmod{\vp}.
\end{align*}
We may now conclude as  follows.

\begin{lemma}\label{lem:yellow}
We have $S_{\vp^\alpha}(\c)=0$ if  $\alpha>1$ and $\vp\nmid MF^*(\c)$.
\end{lemma}

The following result is the function field analogue of the union of Lemmas~12--15 in  \cite{oct} . The desired estimates are established in exactly the same manner and 
the necessary arguments will not be repeated here.

\begin{lemma}\label{lem:7.3}
Let $\vp\nmid \c$ be a prime such that  $|\vp|\gg 1 $ and $\vp\mid F^*(\c)$.
Let $r$ denote the minimal value of the rank modulo $\vp$  of the Hessian  $\mathbf{H}(\y)$, where $\y$ runs over the vectors 
which contribute to $\nu_1(\vp^2,\c) $. Then we have 
$$
S_{\vp^\alpha}(\c)
\ll 
\begin{cases}
|\vp|^{2n+3-r/2}, & \text{if $\alpha=3$},\\
|\vp|^{(\alpha-1)n+4-r}, & \text{if $\alpha\geq 4$},\\
|\vp|^{(\alpha-1)n+6-2r}, & \text{if $\alpha\geq 6$},
\end{cases}
$$
with $r\geq 2$.
If  $\vp\mid \c $, then for $\alpha=3$ or $4$ we have
$$
S_{\vp^\alpha}(\c)\ll |\vp|^{(\alpha-1)n+3}. 
$$
\end{lemma}

The estimates in this result are true for a given value of $r\geq 2$ which depends 
on the value of $\c$.
According to Hooley (see \cite[Eq.~(56)]{oct}), associated to each prime $\vp$ is an affine algebraic variety $V_\vp\subset \AA_{\FF_\vp}^n$, with dimension 
\begin{equation}\label{eq:dad}
D(\vp) \leq \begin{cases}
r-1, &\text{if $r=n-1$ or $n$},\\
r, &\text{if $2\leq r\leq n-2$},
\end{cases}
\end{equation}
such that the estimates in Lemma \ref{lem:7.3} are true for a given value of $r\geq 2$ when 
the 
reduction of $\c$ modulo $\vp$ is constrained to lie in $V_\vp$.

\section{Cubic exponential sums: averages}\label{s:averages}

Recall Definition \ref{def:r} and the attendant notation $b_i,k_i,r_i$ associated to an element  $r\in \cO$.
Throughout this section  
$M\in \cO$ will denote a generic  fixed integer and $\b\bmod{M}$ such that $M\mid F(\b)$
will also be regarded as fixed. In particular, the implied constant in any estimate is allowed to depend on $|\b|$ and $|M|$.
The purpose of this section
is to estimate 
$|S_{r_3,M,\b}(\c)|$ on average over $\c$.
We shall  follow the strategy in \cite{hb-10} and \cite{oct}, although several of our arguments are closer in spirit to those found in \cite[\S 5]{41}.

We begin by recording the trio of estimates that we shall require, before moving onto a discussion of their proofs. 
The first result we need is the analogue of \cite[Lemma 16]{oct}

\begin{lemma}
\label{lem:square-full}
For any $C\geq 1$ and any $\ve>0$ we have
$$
\sum_{\substack{ \c\in \cO^n\\
|\c|< \hat C
}} 
|S_{ r_3,M,\b}(\c)| \ll  | r_3|^{n/2+1+\ve}\left(| r_3|^{n/3}+\hat C^n\right)
$$
and
$$\sum_{\substack{ \c\in \cO^n\\
|\c|< \hat C
}} 
|S_{b_3,M,\b}(\c)| \ll  |b_3|^{n/2+2/3+\ve}\left(|b_3|^{n/3}+\hat C^n\right).$$
\end{lemma}

Our remaining results concern averages of $|S_{r_3,M,\b}(\c)|$ over  sparser sets of $\c$. 
The following result is a  slight sharpening of the  analogous results in 
\cite[Lemma 16]{hb-10} and \cite[Lemma 12]{nonaryI}.

\begin{lemma}\label{lem:square-full 3}
For any $C\geq 1$ and any $\ve>0$ we have
 $$
\sum_{\substack{ \c\in \cO^n\setminus\{\0\}\\
|\c|< \hat C\\ F^*(\c)=0
}} 
|S_{ r_3,M,\b}(\c)| 
\ll |r_3|^\ve\hat C^{\ve}\left(
|b_3|^{5n/6+2/3}|r_4|^{n+1/2}+{\hat C}^{n-3/2}|r_3|^{n/2+4/3}\right).
$$
\end{lemma}

\begin{table}
\centering
\begin{tabular}{|c|c|c| }
\hline
$\gamma_{j,1}$  & $\gamma_{j,2}$ &  Conditions on $j$    \\
\hline
\hline
$2n+\tfrac{13}{2}$ & $3n+4$ & $j=4$ \\
$4n+\tfrac{7}{2}$ & $6n$ & $j=6$ \\
$\tfrac{j}{2}(n+\tfrac{7}{2})-\tfrac{3}{4}$ & 
$\tfrac{j}{2}(2n+1)-\tfrac{n}{2}+\tfrac{1}{2}$
 & $j$ odd \\
$\tfrac{j}{2}(n+\tfrac{7}{2})-1$ & 
$\tfrac{j}{2}(2n+1)-1$
 & $j\not\in\{4,6\}$ even \\
\hline
\end{tabular}
\newline\newline
\caption{Value of the exponents $\gamma_{j,i}$ in \eqref{eq:Gi}}
\label{table}
\end{table}

The final bound involves a summation over an even sparser set of vectors $\c$.
In order to proceed we recall the definition of the functions $G_1(r)$ and $G_2(r)$ that appear in Hooley's work. 
For any $r\in \cO$
and $i=1,2$, let
\begin{equation}\label{eq:Gi}
G_i(r)=\prod_{\vp^j\| r } \vp^{\gamma_{j,i}},
\end{equation}
where the values of $\gamma_{j,i}$ are given in Table \ref{table} and are 
  extracted from \cite[Eqs. (83), (84)]{oct}.
We are now ready to record the following result, which is  the analogue of \cite[Lemma 21]{oct}.

\begin{lemma}
\label{lem:square-full 5}
Let $n=8$.
For any $C\geq 1$ and any $\ve>0$ we have
\begin{align*}
\sum_{\substack{ \c\in \cO^n\setminus \0\\|\c|< \hat C\\ \nabla F^*(\c)=\0}} |S_{r_3,M,\b}(\c)|\ll|r_3|^\ve\hat C^{\ve} \left(|G_1(r_3)|\hat C^{n-5/2}+|G_2(r_3)|\right).
\end{align*}
\end{lemma}

With reference to Table \ref{table}, when $n=8$ we easily deduce that
\begin{equation}\label{eq:G1}
\frac{|G_1(r_3)|}{|r_3|^{n/2+2}}\leq \frac{1}{|b_3|^{1/2}|b_4|^{3/8}|b_5|^{2/5}|b_6|^{1/12}|b_7|^{5/14}|r_8|^{1/4}}
\end{equation}
and 
\begin{equation}\label{eq:G2}
\frac{|G_2(r_3)|}{|r_3|^{n/2+2}}\leq |b_3|^{4/3}|b_4| |b_5|^{9/5}|b_6|^{2}|b_7|^{2}|b_8|^{19/8} |r_9|^{5/2}.
\end{equation}
In particular, it follows from these bounds 
and Lemma \ref{lem:square-full 5} that 
\begin{equation}\label{eq:cor-square}
\sum_{\substack{ \c\in \cO^n\setminus \0\\|\c|< \hat C\\ \nabla F^*(\c)=\0}} \frac{|S_{r_3,M,\b}(\c)|}{|r_3|^{n/2+2}}
\ll |r_3|^\ve\hat C^{\ve} \left(\frac{\hat C^{n-5/2}}{|b_3|^{1/2}}+|b_3|^{4/3}|r_4|^{5/2}\right),
\end{equation}
when $n=8$.

We will provide reasonably detailed  proofs of Lemma 
\ref{lem:square-full} and  Lemma~\ref{lem:square-full 3}, but the proof 
of Lemma \ref{lem:square-full 5} will not be given  here. The latter is closely based on ideas already present in the proofs of the preceding lemmas, with the added 
information about the behaviour at small prime powers that is provided by 
Lemma~\ref{lem:7.3}. The changes required for the function field analogue of \cite[Lemma~21]{oct} are tedious,  routine and do not merit repetition here.

\subsection{Proof of 
Lemma \ref{lem:square-full}}

We begin by establishing the second part of the lemma. 
It follows from
multiplicativity and 
 Lemmas \ref{lem:yellow} and  \ref{lem:7.3} that 
 \begin{equation}\label{eq:duff}
 |S_{b_3,M,\b}(\c)|\ll |b_3|^{\ve}\prod_{\vp\mid b_3} |\vp|^{2n+3-r(\vp,\c)/2+R(\vp,\c)},
 \end{equation}
 where $R(\vp,\c)=0$ if $r(\vp,\c)>1$ and $R(\vp,\c)=1/2$ if $r(\vp,\c)=1$.
Here we stress that the value of $r(\vp,\c)$ depends only on the value of $\c$ modulo $\vp^2$.
Recall from Definition \ref{def:r}
the notation  $b_3=k_3^3$ and note that
there are at most $(\hat C/|k_3|+1)^n$ choices of 
$ \c\in \cO^n$ for which $ |\c|< \hat C$ and $ \c\=\a\bmod {k_3}$.
But then, on invoking \eqref{eq:dad} and the remark after Lemma \ref{lem:7.3}, we easily deduce that 
\begin{align*}
\sum_{\substack{ \c\in \cO^n\\
|\c|< \hat C
}} 
|S_{b_3,M,\b}(\c)| 
&\ll |b_3|^\ve \left(\frac{\hat C}{|k_3|}+1\right)^n
\sum_{\a\bmod k_3}
\prod_{\vp\mid b_3} |\vp|^{2n+3-r(\vp,\a)/2+R(\vp,\a)}\\
&\ll |k_3|^{2n+3}|b_3|^\ve \left(\frac{\hat C}{|k_3|}+1\right)^n
\sum_{2\leq r\leq n}
\prod_{\vp\mid b_3} |\vp|^{D(\vp)-r/2}\\
&\ll |k_3|^{5n/2+2}|b_3|^{\ve} \left(\frac{\hat C}{|k_3|}+1\right)^n
\end{align*}
This completes the proof of the second  part of
Lemma 
 \ref{lem:square-full}.

 \medskip

We now turn to the first part of the lemma. In fact, future work will deem it convenient to establish the following more general version, in which the implied constant is made more explicit.

\begin{lemma}
\label{lem:triangle}
Let $\r\in K_\infty^n$, let   $C\geq 1$ and let   $\ve>0$. Then there exists a constant $c_{n,\ve}>0$, depending only on $n$ and $\ve$, such that 
$$
\sum_{\substack{ \c\in \cO^n\\
|\c-\r|< \hat C
}} 
|S_{ r_3,M,\b}(\c)| 
 \leq c_{n,\ve} |M|^{n} |\Delta_F|^{2n}H_F^{n/2} 
 | r_3|^{n/2+1+\ve}\left(| r_3|^{n/3}+\hat C^n\right).
$$
\end{lemma}

The statement of Lemma 
\ref{lem:square-full} easily follows on  taking $\r=\mathbf{0}$ in this result.
During the proof  of Lemma \ref{lem:triangle}  we will reserve $c_n$ (resp.~$c_{n,\ve}$)
for a generic positive constant that depends only on $n$ (resp.~$n$ and $\ve$). Recall the
definition \eqref{eq:hessian} of the Hessian matrix
$\mathbf{H}(\x)$ associated to the cubic form $F$. 
For any $m\in \cO$  and any $\k\in \cO^n$ let 
\begin{equation}\label{eq:goat}
N_m(\k)=\#\{\y\bmod m: \mathbf{H}(\k)\y\equiv \0\bmod{m} \}.
\end{equation}
We will need the following  result, which is an analogue of \cite[Lemma~13]{41}.

\begin{lemma}
\label{lem:Nmbound}
For any $m\in \cO$ and $R\geq 1$ there exists a constant $c_n>0$ such that 
$$
\sum_{\substack{\k\in \cO^n\\ |\k|<\hat R}} N_m(\k)^{1/2}\leq c_n H_F^{n/2} |m|^{n/2}\left(1+\frac{\hat R^3}{|m|}\right)^{n/2}. 
$$
\end{lemma}

\begin{proof}
Let $D$ denote the degree of
$m$. Given $K\geq 1$, let
$$
S_K=\{\y\in \cO^n:|\y|<\hat{D-K}\} \quad \mbox{ and } \quad
S^1_K=\{t^{D-K}\y:|\y|<\hat K \}. 
$$ 
For any  $\y\in \cO^n$  such that $|\y|<|m|$,  we write $\y=\y_1+\y_2$, where $\y_1\in S_K $ and $\y_2\in S^1_K$. Thus
\begin{align*}
N_m(\k)=\sum_{\y_2\in S^1_K}\sum_{\substack{\y_1\in S_K\\
\bH(\k)(\y_1+\y_2)\equiv \0\bmod m}}
\hspace{-0.4cm}
1
\leq \hat K^n\#\{\y\in S_K: \bH(\k)\y\equiv
\0\bmod m\},
\end{align*}
since if $\y_1+\y_2 $ and $\y_1'+\y_2$ are both
counted by the inner sum then  we have 
$\y_3=\y_1-\y_1'\in S_K$ and $\bH(\k)\y_3\equiv \0\bmod{m}$.

Choosing $K$ such that  $\hat K=H_F\hat R$, we find that 
$$
|\bH(\k)\y|<\frac{H_F\hat R|m|}{\hat
K} = |m|
$$
for any $\y\in S_K$.
Thus, for $\y\in S_K$ we have $\bH(\k)\y\equiv
\0\bmod m $ if and only if $\bH(\k)\y=\0$. It follows that 
\begin{align*}
 N_m(\k)&\leq \hat K^n\#\{\y\in S_K: \bH(\k)\y=\0\}
 = (H_F\hat R)^n\left(\frac{|m|}{H_F\hat R}\right)^{n-\rho(\k)},
\end{align*}
where $\rho(\k)=\rank \bH(\k)$. Hence
\begin{align*}
 \sum_{|\k|<\hat R}N_m(\k)^{1/2}\leq H_F^{n/2} \hat R^{n/2}\sum_{r=0}^n
\left(\frac{|m|}{\hat R}\right)^{(n-r)/2}
\hspace{-0.2cm}
\#\{|\k|<\hat R: \rho(\k)=r\}.
\end{align*}
According to \cite[Lemma~2]{41}, the condition  $\rho(\k)\leq r$ forces $\k$ to lie in 
an affine variety $T_r\subset \AA_K^n$ of dimension at most $r$ and degree $O_n(1)$.
Hence Lemma \ref{lem:trivial} implies that there is a positive constant $c_n>0$ such that 
$$
\#\{|\k|<\hat R: \rho(\k)=r\}\leq \#\{\k\in T_r(\cO): |\k|<\hat R\} \leq c_n \hat R^{r}.
$$
It follows that 
\begin{align*}
 \sum_{|\k|<\hat R}N_m(\k)^{1/2}
 &\leq c_n H_F^{n/2}\hat R^{n/2}\sum_{r=0}^n
\left(\frac{|m|}{\hat R}\right)^{(n-r)/2} \hat R^r\\
& \leq (n+1)c_n H_F^{n/2}
 |m|^{n/2}\left(1+\frac{\hat R^3}{|m|}\right)^{n/2}.
\end{align*}
The statement of the lemma is now clear.
\end{proof}

It will be convenient to relate 
 $S_{r_3,M,\b}(\c)$ to the exponential sum
$$
T(a,s;\c)=\sum_{\substack{\z\in \cO^n\\|\z|<|s|}} 
\psi\left(\frac{ag(\z)-\c.\z}{s}\right),
$$
for appropriate $g\in \cO[x_1,\dots,x_n]$,   $a,s\in \cO$ with  $(a,s)=1$ and 
  $\c\in \cO^n$.  These sums satisfy the  following multiplicativity property. 

\begin{lemma}\label{lem:multi1}
Suppose that $s_1,s_2\in \cO$ are coprime and let $\overline s_1, \overline s_2\in \cO$ be chosen so that
$s_1\overline s_1+s_2\overline s_2=1$. Then 
$T(a,s_1s_2;\c)=T(a\bar s_2,s_1;\bar s_2\c)T(a\bar s_1,s_2;\bar s_1\c).$
\end{lemma}

\begin{proof}
As $\z_1$ ranges over vectors in $\cO^n$ modulo $s_1$ and 
$\z_2$ ranges over vectors  modulo $s_2$, so  $\z=s_2\bar s_2 \z_1+s_1\bar s_1 \z_2$ ranges over a complete set of residues modulo $s_1s_2$. Moreover, we clearly have 
$$
ag(\z)-\c.\z\equiv  s_2\bar s_2\left\{ag(\z_1)-\c.\z_1\right\}+
s_1\bar s_1\left\{ag(\z_2)-\c.\z_2\right\} \bmod{s_1s_2},
$$
since $(s_i\bar s_i)^j\equiv s_i \bar s_i \bmod{s_1s_2}$ for $i\in \{1,2\}$ and all $j\geq 1$.
The desired result now follows easily. 
\end{proof}

Making the change of variables $\y=\b+M\z$ we  
obtain 
\begin{equation}\label{eq:turtle}
S_{r_3,M,\b}(\c)
=
\psi\left(\frac{-\c.\b}{r_3M/(r_3,M)}\right)
\sumstar_{\substack{
|a|<|r_3|} }
T\left(a,\frac{r_3}{(r_3,M)};\c\right),
\end{equation}
with 
 underlying polynomial 
\begin{equation}\label{eq:horse}
g(\z)=\frac{1}{(r_3,M)}F(M\z+\b).
\end{equation}
This is a cubic polynomial with coefficients in $\cO$ since $M\mid F(\b)$. Moreover it has non-singular homogeneous cubic part 
$
 g_0(\z)=(r_3,M)^{-1}M^3F(\z).
$
We now factorise
$ r_3/(r_3, M)$ into a cube-free part and a cube-full part. Since $r_3$ is cube-full it follows that the cube-free part has absolute value at most  $|M|$. Applying  Lemma~\ref{lem:multi1} and estimating the contribution from the cube-free part trivially 
it follows from \eqref{eq:turtle}  that 
$$
|S_{ r_3,M,\b}(\c)| \leq |M|^n  
\sumstar_{\substack{
|a|<|r_3|} }
|T(\bar ba,s;\bar b\c)|
$$
for some 
 cube-full $s\in \cO$ with  $s\mid  r_3$, 
together with some element $\bar b\in \cO$ with $|\bar b|\leq |M|$ and $(\bar b,s)=1$.
To prove 
Lemma \ref{lem:triangle}, 
it will therefore  suffice to show that there is a constant $c_{n,\ve}>0$ depending only on $n$ and $\ve$ such that
\begin{equation}\label{eq:suffice1}
\sum_{\substack{ \c\in \cO^n\\
|\c-\r|<\hat C
}} 
|T(a,s;\c)|
 \leq c_{n,\ve} |\Delta_F|^{2n}H_F^{n/2}  |s|^{n/2+\ve}\left(\hat C^n+|s|^{n/3}\right),
\end{equation}
for any cube-full $s\in \cO$, any $a\in \cO$ which is coprime to $s$ and any $C\geq 1$.

We henceforth write $s=c^2d$, where $d\mid c$ and 
\begin{equation}\label{eq:heron}
d=\prod_{\substack{\vp^e \|s  \\ e\geq 3, ~2\nmid e}}\vp.
\end{equation}
Following the opening argument in  \cite[Lemma 11]{41} more or less verbatim, we easily conclude that 
\begin{align*}
|T(a,s;\c)|\leq |c^2d|^{
n/2}\sum_{\substack{|\u|<|c|\\ a\nabla g(\u)-\c\equiv \0\bmod{c}}}
M_d(\u)^{1/2},
\end{align*}
where 
\begin{equation}\label{eq:donkey}
M_m(\u)=\#\{\y\bmod m:\nabla^2 g(\u)\y\equiv \0\bmod{m} \}.
\end{equation}
Let us denote the left hand side of \eqref{eq:suffice1} by $\mathcal{M}(C)$. Then our work so far shows that 
\begin{align*}
\mathcal{M}(C)
\leq
|c^2d|^{n/2}\sum_{|\c-\r|< \hat C}\sum_{\substack{|\u|<|c|\\ a\nabla
g(\u)-\c\equiv \0 \bmod{c}}} M_d(\u)^{1/2}.
\end{align*}
Let $\ve>0$. Then it follows from 
\cite[Lemma~14]{41} that there is a constant $c_{n,\ve}>0$ depending only on $n$ and $\ve$ such that
\begin{equation}\label{eq:bat}
\begin{split}
\sum_{|\u|<|d|}M_d(\u)
&=\#\{\u,\y\bmod{d}: \nabla^2g(\u)\y\equiv
0\bmod{d}\}\\
&\leq c_{n,\ve} |\Delta_F|^{2n} |d|^{n+\ve}.
\end{split}
\end{equation}

Our argument now differs according to whether $|c|<\hat C$ or $|c|\geq \hat C$. 
Beginning with the former case, we have 
\begin{align*}
 \mathcal{M}(C)&\leq |c^2d|^{ n/2}\sum_{|\u|<|c|}M_d(\u)^{1/2}\#\{|\c-\r|< \hat C: a\nabla
g(\u)-\c\equiv \0\bmod{c} \}\\
&= |c^2d|^{n/2}
\left(\frac{\hat C}{|c|}\right)^n\sum_{|\u|<|c|}M_d(\u)^{1/2}\\
&\leq  |c^2d|^{n/2}\left(\frac{\hat C}{|c|}\right)^n\left(\frac{|c|}{|d|}\right)^n\sum_{|\u|<|d|}
M_d(\u).
\end{align*}
This is at most 
$c_{n,\ve}|\Delta_F|^{2n} |c^2d|^{n/2+\ve}\hat C^n$,
by \eqref{eq:bat}.

Next, suppose that $|c|\geq \hat C$. Starting as above we note that 
\begin{align*}
\#\{|\c-\r|< \hat C: a\nabla
g(\u)-\c\equiv \0\bmod{c} \}
&=
\sum_{\h\in \cO^n}
w\left(\frac{a\nabla g(\u)-\r-c\h}{t^C} \right),
\end{align*}
where $w$ is given by \eqref{eq:weights}. 
Now it follows from Lemma \ref{lem:poisson} that 
\begin{align*}
\sum_{\h\in \cO^n}
w\left(\frac{a\nabla g(\u)-\r-c\h}{t^C} \right)
&=
\sum_{\k\in
\cO^n}\int_{K_\infty^n}w\left(\frac{a\nabla g(\u)-\r-c\x}{t^C}
\right)\psi(\k.\x) \d\x.
\end{align*}
But this is equal to 
$$
\left(\frac{\hat C}{|c|}\right)^n
\sum_{\k\in
\cO^n}
\psi\left(\frac{a\k.\nabla
g(\u)-\r.\k}{c}\right)
\int_{\TT^n}
\psi\left(\frac{t^C\k.\y}{c} \right)
\d\y,
$$
whence an application of  Lemma \ref{lem:orthog} yields
\begin{align*}
\mathcal{M}(C)
&\leq \frac{|c^2d|^{ n/2}\hat C^n}{|c|^n}\sum_{\substack{\k\in
\cO^n\\|\k|<{|c|/\hat C}}}\sigma_\k,
\end{align*}
where
\begin{align*}
 \sigma_\k=\sum_{|\y|<|d|}M_d(\y)^{1/2}\sum_{\substack{|\u|<|c|\\ \u\equiv
\y\bmod{d}}}\psi\left(\frac{a\k.\nabla
g(\u)-\r.\k}{c} \right).
\end{align*}
We proceed with an application of Cauchy's inequality and \eqref{eq:bat}, to obtain
\begin{align*}
|\sigma_\k|^2
&\leq c_{n,\ve} |\Delta_F|^{2n}
|d|^{n+\ve}
 \sum_{|\y|<|d|}\left|\sum_{\substack{
|\u|<|c|\\\u\equiv
\y\bmod{d}}}\psi\left(\frac{a\k.\nabla
g(\u)}{c}\right)\right|^2\\
&\leq c_{n,\ve} |\Delta_F|^{2n}|d|^{n+\ve}\sum_{\substack{|\u_1|,|\u_2|<|c|\\\u_1\equiv
\u_2\bmod{d} }}\psi\left( \frac{a\k.(\nabla
g(\u_1)-\nabla g(\u_2))}{c}\right).
\end{align*}
Writing $\u_1=\u_2+d\z$
and recalling \eqref{eq:horse}, 
we see that  
$$
\nabla g(\u_1)-\nabla g(\u_2)=d (r_3,M)^{-1}M^3\bH(\z)\u_2
$$
plus a term which in independent of $\u_2$. Hence there exists $m\in \cO$, with 
$|m|\leq |c/d|$, such that 
$|\sigma_\k|^2\leq c_{n,\ve} |\Delta_F|^{2n}|d|^{n+\ve}|c|^nN_{m}(\k)$,
in the notation of \eqref{eq:goat}. 
It now follows from 
 Lemma \ref{lem:Nmbound} that
 \begin{align*}
\mathcal{M}(C)
&\leq c_{n,\ve}^{1/2} |\Delta_F|^{n}
\frac{|c^2d|^{ n/2}\hat C^n}{|c|^n}\sum_{\substack{\k\in
\cO^n\\|\k|<{|c|/\hat C}}}
|d|^{n/2+\ve}|c|^{n/2}N_{m}(\k)^{1/2}\\
&\leq c_nc_{n,\ve}^{1/2}|\Delta_F|^{n} H_F^{n/2}
\frac{|c^2d|^{ n/2}\hat C^n  |d|^{n/2+\ve}|c|^{n/2}}{|c|^n}
\left(|m|+
\frac{|c|^3}{\hat C^3}\right)^{n/2}\\
&\leq
c_nc_{n,\ve}^{1/2}|\Delta_F|^{2n} H_F^{n/2}
|c^2d|^{ n/2+\ve}\hat C^n 
\left(1+ \frac{|c^2d|}{\hat C^3}\right)^{n/2}.
\end{align*}
In view of our earlier work this bound is also valid when $|c|<\hat C$.

Let $D=\deg(c^2d)$.
We therefore arrive at the desired bound \eqref{eq:suffice1} on noting that
$\mathcal{M}(C)\leq \mathcal{M}(\max\{C,\tfrac{1}{3}D\}).$

\subsection{Proof of Lemma \ref{lem:square-full 3}}

In addition to taking into account the sparsity of vectors $\c$ for which $F^*(\c)=0$, in the proof of Lemma 
\ref{lem:square-full 3}
we will also need to sum non-trivially over $a$ in the definition of 
$S_{r_3,M,\b}(\c)$. 

To begin with we factorise $r_3=b_3r_4$ and use Lemma \ref{lem:multi2}
to factorise the sum 
$S_{r_3,M,\b}(\c)$.  The sum corresponding to $b_3$ we estimate using \eqref{eq:duff}. 
For the sum involving $r_4$ we return to 
 \eqref{eq:turtle} and relate the exponential sum to $T(a,s;\c)$ for a quartic-full $s\in \cO$.
Abusing notation slightly, this leads to the preliminary estimate
$$ 
\sum_{\substack{ 
|\c|< \hat C\\ F^*(\c)=0
}} 
|S_{ r_3,M,\b}(\c)| 
\ll |r_3|^\ve
\hspace{-0.3cm}
\sum_{\substack{|\c|<\hat C\\  F^*(\c)=0
}} 
\hspace{-0.2cm}
\prod_{\vp\mid b_3} |\vp|^{2n+3-r(\vp,\c)/2+R(\vp,\c)}
\left|~
\sumstar_{\substack{|a|<|r_4|} }
T(a,r_4;\c)\right|.
$$
The term involving $b_3$ only depends on $\c$ modulo $k_3$. Thus, arguing as in the proof of the second part of Lemma \ref{lem:square-full}, we break the $\c$-sum into residue classes modulo $k_3$ and deduce that
\begin{align*}
\sum_{\substack{
|\c|< \hat C\\ F^*(\c)=0
}} 
|S_{ r_3,M,\b}(\c)| 
\ll~& 
|k_3|^{2n+3}
|r_3|^\ve 
\sum_{\a\bmod{k_3}}\Sigma(\a)
\prod_{\vp\mid b_3} |\vp|^{-r(\vp,\a)/2+R(\vp,\a)},
\end{align*}
where
$$
\Sigma(\a)=
\sum_{\substack{|\c|<\hat C\\ F^*(\c)=0\\\c \=\a\bmod{k_3}
}} 
\left|~
\sumstar_{\substack{|a|<|r_4|} }
T(a,r_4;\c)\right|.
$$
We will show that 
\begin{equation}\label{eq:suffice2}
\Sigma(\a)\ll |r_4|^{n+1/2+\ve} +\frac{|r_4|^{n/2+4/3+\ve}\hat C^{n-3/2+\ve}}{|b_3|^{n/3-1/2}}.
\end{equation}
Recollecting \eqref{eq:dad}, we can insert this into the above 
estimate in order to conclude the proof of 
Lemma \ref{lem:square-full 3}.

In order to prove \eqref{eq:suffice2}, 
 we write $r_4=c^2d$ as before, with $d$ given by \eqref{eq:heron}.
The  argument in \cite[\S 7]{hb-10} now 
goes through more or less verbatim, leading to the bound
\begin{align*}
\sumstar_{\substack{|a|<|r_4|} }
T(a,r_4;\c)
&\ll |c|^{n+1}|d|^{n/2+1}
\sumstar_{|a_1|<|c|} 
\sum_{\substack{|\u|<|c| 
\\ 
a_1\nabla g(\u)-\c \equiv \0\bmod{c}\\
g(\u)\equiv 0 \bmod{c}} }
\hspace{-0.3cm}
M_{d}(\u)^{ 1/2},
\end{align*}
in the notation of \eqref{eq:donkey}.
Making the change of variables $\h=M\u+\b$, we deduce that there are elements $c',d'$ with $d'\mid c'$ and 
$|c'|$ (resp.~ $|d'|$) of order $|c|$ (resp.~ $|d|$), such that 
$$
\sum_{\substack{|\u|<|c| 
\\ 
a_1\nabla g(\u)-\c \equiv \0\bmod{c}\\
g(\u)\equiv 0 \bmod{c}} }
\hspace{-0.3cm}
M_{d}(\u)^{ 1/2}
=\sum_{\substack{|\h|<|c'| 
\\ 
a_1\nabla F(\h)-\c \equiv \0\bmod{c'}\\
F(\h)\equiv 0 \bmod{c'}} }
\hspace{-0.3cm}
N_{d'}(\h)^{ 1/2}.
$$

Summing trivially over $a_1$, we now find that
\begin{equation}\label{eq:flea}
\sum_{\substack{|\c|<\hat C\\ F^*(\c)=0\\\c \=\a\bmod{k_3}
}} 
\left|~
\sumstar_{\substack{|a|<|r_4|} }
T(a,r_4;\c)\right|
\ll 
 |c|^{n+2}|d|^{n/2+1} \mathcal{N}
\sum_{\substack{|\h|<|c'|\\  
F(\h)\equiv 0\bmod{c'}
}} N_{d'}(\h)^{1/2},
\end{equation}
where 
\begin{align*}
 \mathcal N=\max_{|\r|<|k_3c'|}\#\left\{\c\in\cO^n:  |\c|<\hat C, ~F^*(\c)=0, ~\c\equiv \r\bmod{k_3c'} \right\}.
\end{align*}
The equation $F^*(\c)=0$ cuts out an absolutely irreducible hypersurface in $\AA^n$ of dimension $n-1$. Hence it follows from Lemma \ref{lem:cohen} that 
\begin{equation}\label{eq:gnat}
\mathcal{N}\ll \left(\frac{\hat C}{|k_3c|}+1\right)^{n-3/2}.
\end{equation}

It remains to analyse the sum 
$$
S(c,d)
=\sum_{\substack{|\h|<|c|\\  
F(\h)\equiv 0\bmod{c}
}} N_{d}(\h)^{1/2},
$$
for given $c,d\in \cO$ such that 
$d$ is square-free and $d\mid c$.
We will show that 
$$
S(c,d)\ll |c|^{n-1+\ve}|d|^{1/2}.
$$
Once combined with \eqref{eq:gnat} in \eqref{eq:flea},  this gives the desired bound \eqref{eq:suffice2} on noting that $|d|\leq |c^2d|^{1/3}=|r_4|^{1/3}$.
The sum in question  satisfies 
$S(c_1c_2,d_1d_2)=S(c_1,d_1)S(c_2,d_2)$ for any $c_i,d_i\in \cO$ such that $(c_1d_1,c_2d_2)=1$ and $d_i\mid c_i$.
Hence it will suffice to show that 
$$
S_1=S(\vp^e,1)\ll |\vp|^{e(n-1)} \quad \mbox{ and } \quad S_2=S(\vp^e,\vp)\ll |\vp|^{e(n-1)+1/2},
$$
for any $e\in \NN$ and any prime $\vp$.
This is achieved by closely  following the argument of Heath-Brown
\cite[page~245]{hb-10}. The estimation of $S_1$ uses exponential sums and an application of  Lemma 
\ref{lem:square-full} with $C=1$. The main ingredient in the  estimation of $S_2$ is
\eqref{eq:bat}.  Given that the arguments of 
\cite[page~245]{hb-10} carry over verbatim to the function field setting, they will not be repeated here. 

\section{Return to the main counting function}

Recall our standing  assumption that $\ch(\FF_q)>3$, together with the definition \eqref{eq:hessian} of the Hessian matrix associated to our non-singular cubic form
$F\in \cO[x_1,\dots,x_n]$.
The proof of  \cite[Lemma~1]{nonaryI} shows that
there exists a point $\x_0\in K_\infty^n$ satisfying 
\begin{equation}\label{eq:nice}
F(\x_0)=0,\quad \det \mathbf{H}(\x_0)\neq 0, \quad |\x_0|<1/H_F.
\end{equation}
An inspection of the proof reveals that the result is false in characteristic $2$
or $3$ when $F$ is cubic.
Such a point will automatically satisfy $\nabla F(\x_0)\neq \mathbf{0}$, since
$F $ is non-singular.

Next, let $L\geq 1$ be an integer.
We  define the weight function $\omega:K_\infty^n\rightarrow \RR_{\geq 0}$
via
\begin{equation}\label{eq:omega}
\omega(\x)=w\left(t^L(\x-\x_0)\right),
\end{equation}
where $w$ is given by \eqref{eq:weights}. 
Ultimately, $L$ will be taken to be a large but fixed integer.
For $L$ large enough, it is clear that 
\begin{equation}\label{eq:heat}
|\x|<1/H_F\quad  \text{ and } \quad |\det \mathbf{H}(\x)|=|\det \mathbf{H}(\x_0)|,
\end{equation}
for any $\x\in K_\infty^n$ such that 
$\omega(\x)\neq 0$.

Let  $\b\in \cO^n$ and let  $M\in \cO$ such that $M\mid F(\b)$.
It is clear that $\omega\in S(K_\infty^n)$ and we are interested in the
asymptotic behaviour of the counting function
\begin{equation}\label{eq:N'}
N(P)=\sum_{\substack{\x\in \cO^n\\ F(\x)=0\\
\x\equiv \b\bmod{M}}} \omega(\x/P),
\end{equation}
as $|P|\rightarrow \infty$.
The quantities $\x_0, \b, M,L$ are to be considered  fixed once and for all. 
Consequently, all our implied constants  are allowed to depend
on these  quantities as well as on the height $H_F$ of $F$.

Our main result concerning the behaviour of $N(P)$ is as follows. 

\begin{theorem}\label{THM}
Suppose that $n=8$. Then there exists constants $c\geq 0$ and $\delta>0$ such that 
$$
N(P)=c|P|^{n-3}+O(|P|^{n-3-\delta}).
$$
The constant $c$ is a Hardy--Littlewood product of local densities, with $c>0$ if 
for every finite prime $\vp$
there exists 
$\x\in\cO_\vp^n$ such that $F(\x)=0$ and 
$|\x-\b|_\vp <|M|_\vp$. 
\end{theorem}

In \S \ref{s:beep} we  show how this result implies the statement of 
Theorem \ref{thm:HP+WA}. Next, in \S \ref{s:honk}, we initiate our analysis of $N(P)$ along the lines of \S  \ref{s:active}.
The outcome of this first phase of the argument  is recorded in Lemma~\ref{lem:N2}. 
The main contribution to $N(P)$ comes from the trivial characters, which is what we analyse in \S \ref{8.1}. 
It is here that the explicit value of the leading constant $c$ is recorded. 
Finally,  \S \ref{8.2} is devoted to a preliminary analysis of  the contribution from the non-trivial characters.

\subsection{Deduction of Theorem \ref{thm:HP+WA}}\label{s:beep}

This section shows  how Theorem \ref{THM} implies 
Theorem \ref{thm:HP+WA}. Let $X\subset \PP_K^{n-1}$ be a non-singular cubic hypersurface defined by a cubic form $F$ over $K$ with $n\geq 8$ variables. Assume that  $X(K_v)\neq \emptyset$ for every place $v\in \Omega$.
In order to establish the Hasse principle and weak approximation, we need to show that $X(K)\neq \emptyset$
and 
$X(K)$ is dense in $X(\mathbf{A}_K)$ under the product topology. 

Using a familiar fibration argument, we use  induction on the number of variables $n\geq 8$, supposing for the moment  that it is has been verified when $n=8$. Thus let $n\geq 9$ and let $H_1,H_2$ be generic hyperplanes in $\PP_K^{n-1}$ defined over $K$. 
We consider the fibration $\pi: X\to \PP_K^1$ with fibres $X_{\lambda,\mu}=X\cap H_{\lambda,\mu}$, where $H_{\lambda,\mu}=\lambda H_1+\mu H_2$. By the Lefschetz hyperplane theorem $\Pic(X)$ is a free abelian group of rank 1 generated by the class of a hyperplane section $Y$.  All fibres of $\pi$ are therefore  geometrically integral. Indeed if a fibre  $X_{\lambda,\mu}$ were reducible, say  $X_{\lambda,\mu}=Y_1+Y_2$, then $Y_1,Y_2$ would give independent elements of the Picard group of $X$ which are not multiples of $Y$, which is impossible. Moreover $\PP_K^1$ satisfies the Hasse principle and weak approximation, as do the  smooth fibres by the inductive hypothesis. A standard argument (see Skorobogatov \cite{skoro}, for example) therefore yields the desired conclusion subject to a satisfactory treatment of the case $n=8$.

Henceforth suppose that $n=8$.
Let $S$ be a finite set of primes of $K$.
Suppose that we are given 
 points $x_\infty\in X(K_\infty)$ and $x_\vp\in X(K_{\vp})$ for each $\vp\in S$.
We wish to prove that there exists a rational point $x\in X(K)$ which is simultaneously close to these local points in their respective topologies. Since the Hessian does not vanish identically on $X$, there is no loss of generality in assuming that $x_\infty$ doesn't lie on the Hessian variety.

Let $N_\infty, N$ be positive
integers. 
 We choose representative coordinates so that $x_\infty=[\x_\infty]$ for
$\x_\infty\in \TT^n$ such that $|\x_\infty|<1/H_F$
and $x_\vp=[\x_\vp]$ for $\x_\vp\in \cO_{\vp}^n$, for each
$\vp\in S$. 
We need to show that there exists a
non-zero vector
$\z\in K^n$ such that $F(\z)=0$, with 
\begin{equation}\label{eq:local-weak}
|\z-\x_\infty|<\hat N_\infty^{-1} \quad \mbox{ and } \quad |\z-\x_\vp|_\vp <
|\vp|_\vp^{-N}, \mbox{ for all $\vp\in S$.}
\end{equation}
Combining weak approximation for $\cO^n$ with the Chinese remainder theorem, we
can find a vector $\b\in \cO^n$ such that $\b\equiv \x_\vp \bmod{\vp^N}$ for
every $\vp\in S$.
Let $M=\prod_{\vp\in S}\vp^N$ and
let $B$ run through elements of $\cO$ for which $B\equiv 1\bmod{M}$. 
For $|B|$ suitably large we will  show  that there is a vector $\x\in \cO^n$
such that 
$F(\x)=0$, with 
$$
|\x-B\x_\infty|<\hat N_\infty^{-1} |B|
 \quad \mbox{ and } \quad 
\x\equiv \b \bmod{M}.
$$
We claim that the vector $\z=\x/B\in K^n$ will satisfy the conditions required
to draw the desired conclusion. Now it is clear that $F(\z)=0$
and that the  restriction at the infinite place in \eqref{eq:local-weak} is satisfied.  Moreover,
for any $\vp\in S$ we will have 
$|\z-\x_\vp|_\vp < |\vp|_\vp^{-N}$ if and only if 
$|\x-\x_\vp|_\vp < |\vp|_\vp^{-N}$, since $B\equiv 1 \bmod{\vp^N}$. But this
follows from the fact that 
$$
|\x-\x_\vp|_\vp \leq \max\left\{|\x-\b|_\vp,
|\b-\x_\vp|_\vp\right\}<|\vp|_\vp^{-N}.
$$

It will therefore suffice to study the counting function $N(P)$ in
\eqref{eq:N'},  with $\x_0=\x_\infty$ and 
$L=N_\infty$.  Indeed, our 
 arguments so far show that the Hasse principle and weak approximation hold when $n=8$,
 if we are able to 
show that 
$$
N(P)>0,
$$ 
for $P\in \cO$ such that $|P|\rightarrow \infty$.
But this follows directly from the statement of  Theorem \ref{THM}.

\subsection{Preliminary analysis of  $N(P)$}\label{s:honk}

Our starting point is Lemma
\ref{lem:summary}, which gives
\begin{equation}\label{eq:N1}
N(P)=|P|^n
\sum_{\substack{
r\in \cO\\
|r|\leq \hat Q\\
\text{$r$ monic}}
} 
|r_M|^{-n}
\int_{|\theta|<|r|^{-1}\hQ^{-1}} 
\sum_{\substack{ \c\in \cO^n
}} 
S_{r,M,\b}(\c)
I_{r_M}(\theta;\c)
 \d \theta,
\end{equation}
where $r_M=rM/(r,M)$ and 
$S_{r,M,\b}(\c),
I_{r_M}(\theta;\c)$ are as in the statement of lemma. 
We proceed to use the results of \S \ref{s:integrals} to study 
$I_r(\theta;\c)$ for given  $r\in \cO$.
In view of \eqref{eq:omega}  and Lemma 
\ref{lem:change},	we have 
\begin{align}\notag
I_r(\theta;\c)
&=
\int_{K_\infty^n} 
w\left(t^L(\x-\x_0)\right)
\psi \left(\theta P^3F(\x)+\frac{P \c.\x}{r}\right)\d \x\\
\notag&=\frac{1}{{\hat L}^{n}}
\psi\left(
\frac{P \c.\x_0}{r}\right)
\int_{K_\infty^n} 
w\left(\y\right)
\psi \left(\theta P^3F(\x_0+t^{-L}\y)+\frac{Pt^{-L} \c.\y}{r}\right)\d \y\\
&=\frac{1}{{\hat L}^{n}}
\psi\left(
\frac{P \c.\x_0}{r}\right)
J_G\left(\theta P^3;
\frac{Pt^{-L} \c}{r}\right),\label{eq:small}
\end{align}
in the notation of \eqref{eq:J},
where $G(\y)=
F(\x_0+t^{-L}\y)$. It is clear that $G$ is a polynomial with coefficients in
$K_\infty$ and height 
$H_G\leq H_F$.

According to Lemma 
\ref{lem:J-easy} 
we have 
$
J_G(\theta P^3;Pt^{-L}\c/r)=0$ if 
$$
\frac{|P||\c|}{|r|}> \hat L \max\{1,| P|^3 |\theta | H_F\}.
$$
Hence we may truncate the sum over $\c$ in \eqref{eq:N1} to arrive at the following result.

\begin{lemma}\label{lem:N2}
We have 
$$
N(P)=|P|^n
\sum_{\substack{
r\in \cO\\
|r|\leq \hat Q\\
\text{$r$ monic}}
} 
|r_M|^{-n}
\int_{|\theta|<|r|^{-1}\hQ^{-1}} 
\sum_{\substack{ \c\in \cO^n\\
|\c|\leq \hat C
}} 
S_{r,M,\b}(\c)
I_{r_M}(\theta;\c)
 \d \theta,$$
where $\hat C=\hat L H_F |r_M|| P|^{-1}\max\{1, |\theta|| P|^3\}$.
\end{lemma}

We will need a good upper bound for 
$I_{r_M}(\theta;\c)$, for $r, \theta, \c$ appearing in the expression for $N(P)$ in this lemma.
 This need is met by the following result. 

\begin{lemma}\label{lem:I-hard}
We have 
$$
|I_{r_M}(\theta;\c)|\ll 
H_F^n \max\{1, |\theta|| P|^3\}^{-n/2}.
$$
\end{lemma}

\begin{proof}
When $|\c|\leq \hat C$
we put 
$\gamma=\theta P^3$ and $\w=Pt^{-L}\c/r_M$, for convenience. 
In particular we have 
$$
 |\w|\leq H_F\max\{1,|\gamma|\}.
$$
It then follows from Lemma \ref{lem:small} that 
\begin{align*}
|I_{r_M}(\theta;\c)| &\leq \frac{1}{{\hat L}^{n}}
\left|J_G(\gamma;\w)\right| \\
&
\leq \frac{1}{{\hat L}^{n}}\meas\left\{\y\in \TT^n: |\gamma\nabla G(\y)+ \w|\leq
H_G \max\{1,
|\gamma|^{1/2}\}\right\}\\
&\leq
\meas(\mathcal{R}),
\end{align*}
where
$$
\mathcal{R}=
\left\{\x\in \TT^n: |\x-\x_0|<\hat L^{-1}, ~ |\gamma\nabla F(\x)+ \w|\leq H_F
\max\{1,
|\gamma|^{1/2}\}\right\}.
$$
We would like to  estimate the measure of this region. 
Recall that $\x_0$ satisfies \eqref{eq:nice}.
The parameter $L\geq 0$ is  chosen  large enough that 
\eqref{eq:heat} holds for  all $\x\in K_\infty^n$ such that 
$ |\x-\x_0|<\hat L^{-1}$.

To begin with, 
if $|\gamma|\leq 1$ then we take the trivial bound $\meas (\mathcal{R})\leq 1$.
Let us suppose instead that $|\gamma|> 1.$
If $\x$ and $\x+\x'$ are both in $\mathcal{R}$ then 
$$
|\nabla F(\x+\x')-\nabla F(\x)| \leq H_F|\gamma|^{-1/2}.
$$
But 
$$
\left|\nabla F(\x+\x')-\nabla F(\x)-\mathbf{H}(\x)\x'\right|\leq H_F |\x'|^2.
$$
Using the inverse of $\mathbf{H}(\x)$, whose entries each have absolute value $O(1)$,
we find that 
$$
|\x'| \ll H_F\max\left\{|\gamma|^{-1/2} , |\x'|^2\right\}.
$$
This implies that $|\x'|\ll H_F|\gamma|^{-1/2}$, since $|\x'|<1/H_F$.
We have therefore shown that 
\begin{equation}\label{eq:brownies}
\meas(\mathcal{R})\ll H_F^n \min\{1,|\gamma|^{-n/2}\},
\end{equation}
which concludes the proof of the lemma.
\end{proof}

\subsection{The main term}\label{8.1}

In this section we investigate the contribution to $N(P)$ in Lemma \ref{lem:N2}
coming from $\c=\0$. Let us denote this term by $M(P)$. We will always assume that $n\geq 8$. Recalling the 
definition of 
$I_{r_M}(\theta;\0)$ from Lemma~\ref{lem:summary}, 
 we find that
$$
M(P)=|P|^n
\sum_{\substack{
r\in \cO\\
|r|\leq \hat Q\\
\text{$r$ monic}}
} 
|r_M|^{-n}
S_{r,M,\b}(\0)
K_r,
$$
where
$$
K_r=
\int_{|\theta|<|r|^{-1}\hQ^{-1}} 
\int_{K_\infty^n} \omega(\x)
\psi \left(\theta P^3F(\x)\right)\d \x~
 \d \theta.
$$
It follows from Lemma \ref{lem:I-hard} that 
$
K_r=O(|P|^{-3})$
for any $r$.
Moreover, 
we recall from \eqref{eq:omega} that 
$\omega(\x)=w(t^L(\x-\x_0))$ in $K_r$,
where $w$ is given by \eqref{eq:weights}, $L$ is a large  fixed integer and 
$\x_0$ satisfies \eqref{eq:nice}.
In particular $\nabla F(\x_0)\neq \0$ and we
  let $\xi\in \ZZ$ be such that 
$$
\hat \xi=|\nabla F(\x_0)|.
$$
In particular $|\xi|<1$. 

We begin with the following basic result.

\begin{lemma}
\label{lem:bday}
For any $Y\geq 1$ and any $\ve>0$, we 
have 
$$
\sum_{\substack{
r\in \cO\\
 |r|= \hat Y\\
\text{$r$ monic}}
} 
|r_M|^{-n}
|S_{r,M,\b}(\0)|\ll \hat Y^{5/4-n/6+\ve}.
$$
\end{lemma}

\begin{proof}
We factorise any $r$ in the summation as $r=b_1b_2b_3r_4$ and use the multiplicativity property  
Lemma \ref{lem:multi2} that is enjoyed by $S_{r,M,\b}(\0)$.  For the modulus $b_1b_2$ we apply Lemma \ref{lem:62}. For the modulus  $b_3$ (resp. ~$r_4$) we use the second (resp.~ first) part of Lemma \ref{lem:square-full} with $C=1$. This leads to the conclusion that 
\begin{align*}
S_{r,M,\b}(\0) 
&\ll |r|^\ve|b_1b_2|^{n/2+1} |b_3|^{5n/6+2/3} |r_4|^{5n/6+1}\\
&\ll |r|^{n/2+1+\ve} |b_3|^{n/3-1/3} |r_4|^{n/3}.
\end{align*}
Hence
\begin{align*}
\sum_{\substack{
r\in \cO\\
 |r|= \hat Y\\
\text{$r$ monic}}
} 
|r_M|^{-n}
|S_{r,M,\b}(\0) |
&\ll \hat Y^{1-n/6+\ve}\sum_{|b_3r_4|\leq \hat Y} 
|b_3|^{-1/3}\sum_{|b_1b_2|=\hat Y/|b_3r_4|} \frac{1}{|b_1b_2|^{n/3}}.
\end{align*}
The inner sum is absolutely convergent and there are $O(\hat R^{1/j})$ elements $r_j\in \cO$ such that $|r_j|\leq \hat R$.
Summing first over $r_4$ we see that the resulting sum over $b_3$ is absolutely convergent, which therefore completes the proof of the lemma.
\end{proof}

Let us put  $C=\hat{L-\xi}$. Since
$K_r=O(|P|^{-3})$ and 
 $5/4-n/6<0$ for $n\geq 8$,
Lemma \ref{lem:bday} implies that there exists $\delta>0$ such that 
the overall contribution to $M(P)$ from $r$ satisfying $C^{-1}\hat Q\leq |r|\leq \hat Q$ is $O(|P|^{n-3-\delta})$.
On the remaining range for $r$ we will actually  show that $K_r$ is independent of $r$.
Let 
$$
J_\theta
=
\int_{K_\infty^n} \omega(\x)
\psi \left(\theta P^3 F(\x)\right)\d \x.
$$
We then have 
$$
K_r=\int_{|\theta|<C|P|^{-3}} J_\theta ~\d\theta +\int_{C|P|^{-3}\leq |\theta|<|r|^{-1}\hat Q^{-1}} J_\theta ~\d\theta.
$$
The first integral is independent of $r$ and the second integral is over a non-empty interval if and only if $|r|<C^{-1}\hat Q^{-1}|P|^3=C^{-1}\hat Q$.

Recalling \eqref{eq:omega}, we obtain
\begin{align*}
J_\theta
&=
\int_{K_\infty^n} 
w\left(t^L(\x-\x_0)\right)
\psi \left(\theta P^3 F(\x)\right)\d \x\\
&=\frac{1}{\hat L^n}
\int_{\TT^n} 
\psi \left(\theta P^3 F(\x_0+t^{-L}\y)\right)\d \y.
\end{align*}
Let $f(\y)=\theta P^3 F(\x_0+t^{-L}\y)$.  Then, provided $L$ is sufficiently large, we will
have 
$$
|\nabla f(\y)|=
|\theta P^3t^{-L} \nabla F(\x_0+t^{-L}\y)|=\frac{|\theta| |P|^3 \hat \xi}{\hat
L}=\lambda,
$$
say, for all $\y\in \TT^n$.
Likewise, we have $|\partial^{\bbe} f(\y)|< \lambda$, 
for all $|\bbe|\geq 2$ and all $\y\in \TT^n$.  Hence Lemma 
\ref{lem:deriv} implies that  $J_\theta=0$
if $\lambda\geq 1$.
It therefore follows that 
$$
K_r=
\int_{|\theta|< C|P|^{-3}} 
\int_{K_\infty^n} \omega(\x)
\psi \left(\theta P^3F(\x)\right)\d \x
 \d \theta,
$$
when $|r|<C^{-1}\hat Q$, 
which is now independent of $r$.

Making the change of variables $\phi=\theta P^3$, we conclude that 
$$
M(P)=|P|^{n-3}
\mathfrak{S}(Q)
\mathfrak{I} +O(|P|^{n-3-\delta}),
$$
for $n\geq 8$, 
where
$$
\mathfrak{S}(Q)
=
\sum_{\substack{
r\in \cO\\
|r|\leq \hat Q\\
\text{$r$ monic}}
} 
|r_M|^{-n}
S_{r,M,\b}(\0)
$$
and 
$$
\mathfrak{I}=
\int_{|\phi|<\hat{L-\xi}} 
\int_{K_\infty^n} \omega(\x)
\psi \left(\phi F(\x)\right)\d \x
 \d \phi.
$$
The latter quantity is (essentially) the ``singular integral'' for the problem  and can be evaluated 
explicitly as follows.


\begin{lemma}\label{lem:sing-integral}
We have 
$$\mathfrak{I}=
\frac{1}{|\nabla F(\x_0)| \hat L^{n-1}}
\gg 1.
$$
\end{lemma}

\begin{proof}
Opening up $\omega$  and making  a change of variables as before, we see that 
\begin{align*}
\mathfrak{I}
&=
\frac{1}{ \hat{L}^n}
\int_{|\phi|<\hat{L-\xi}} 
\int_{\TT^n} 
\psi \left(\phi F(\x_0+t^{-L}\y)\right)\d \y
 \d \phi\\
&=
\frac{\hat{L-\xi}}{\hat L^n}
\meas \left\{ \y\in \TT^n: 
|F(\x_0+t^{-L}\y)| <\frac{\hat \xi}{\hat L}
\right\},
\end{align*}
by Lemma \ref{lem:orthog}.
Put $f(\y)=F(\x_0+t^{-L}\y)$. Then Taylor's theorem yields
$$
f(\y)=t^{-L}\y.\nabla F(\x_0)+\tfrac{1}{2}t^{-2L}\y^T \nabla^2F(\x_0)\y +t^{-3L}F(\y),
$$
since $F(\x_0)=0$ by \eqref{eq:nice}.
Now the second and third terms here have absolute value $O(|\y|^2/\hat L^2)$, whereas the first term has absolute value at most 
$|\y|\hat \xi/\hat L$.
Assuming that $L$ is large enough, it therefore follows from the ultrametric inequality that 
$|f(\y)|<\hat \xi/\hat L$ for any $\y\in \TT^n$. Hence the region in which we are interested has measure $1$, which finally leads to the desired conclusion.
\end{proof}

In view of Lemma \ref{lem:bday} we can extend the summation over $r$ in $\mathfrak{S}(Q)$ to infinity with acceptable error. Thus, for $n\geq 8$,  there exists $\delta>0$ such that 
$$
M(P)=|P|^{n-3}
\mathfrak{S}
\mathfrak{I} +O(|P|^{n-3-\delta}),
$$
where $\mathfrak{I}$ is given by Lemma \ref{lem:sing-integral} and 
$$
\mathfrak{S}
=
\sum_{\substack{
r\in \cO\\
\text{$r$ monic}}
} 
|r_M|^{-n}
S_{r,M,\b}(\0)
$$
is the (absolutely convergent) ``singular series''.
The analysis of $\mathfrak{S}$ is standard and 
will not be repeated here. 
It  runs exactly as in Lee  \cite{lee, lee-thesis}, with the outcome that $\mathfrak{S}>0$ if 
for every finite prime $\vp$ 
there exists 
$\x\in\cO_\vp^n$ such that $F(\x)=0$ and 
$|\x-\b|_\vp <|M|_\vp$.

\subsection{Preparations for the error term}\label{8.2}

It remains to show that overall contribution to $N(P)$ in Lemma \ref{lem:N2} from $\c\neq \0$ is $O(|P|^{n-3-\delta})$ for some $\delta>0$ if $n=8$. The purpose of this section is to lay some groundwork furthering this aim.
Now it is clear from \eqref{eq:blue} that there is a satisfactory overall contribution to $N(P)$ from values of $\theta$
such that $|\theta|<\hat Q^{-5}$. This allows us to henceforth focus on the contribution from $|\theta|\geq \hat Q^{-5}$.

Let $Y, \Theta\in \ZZ$ be such that 
\begin{equation}\label{eq:conditions on Y}
0\leq Y\leq Q, \quad 
-5Q\leq \Theta<-(Y+Q).
\end{equation}
The last inequality is equivalent to $\hat Q^{-5}\leq \hat \Theta<(\hat Y \hat Q)^{-1}$ and one sees that there are 
at most $4Q=O(\log |P|)$ choices for $Y,\Theta$.
We will content ourselves with focusing on the overall contribution to $N(P)$ from $\c\neq \0$ and $r, \theta $ such that 
$|r|=\hat Y$ and $|\theta|=\hat \Theta$. Let us denote this contribution by $E(P)=E(P;Y,\Theta)$.
Suppose that  we are able to prove the existence of a positive constant $\eta>0$ such that 
\begin{equation}\label{eq:EP}
E(P)=O(|P|^{n-3-\eta})
\end{equation}
for any 
 $Y, \Theta\in \ZZ$ satisfying \eqref{eq:conditions on Y}. Then this will lead to an asymptotic formula for $N(P)$, as $|P|\to \infty$,  for the range of $n$ that \eqref{eq:EP} is valid for.

In what follows it will be convenient to introduce the notation 
\begin{equation}\label{eq:J-notation}
J(\Theta)=\max\{1, \hat \Theta| P|^3\}.
\end{equation}
The constraint on $\c$ imposed in Lemma \ref{lem:N2} now becomes
$
|\c|\ll \hat Y |P|^{-1}J(\Theta),
$
for a suitable implied constant. 
In particular, since $\c\neq \0$, we must have
\begin{equation}\label{eq:hatY-size}
\hat Y\gg \frac{|P|}{J(\Theta)}.
\end{equation}
Switching the order of summation we obtain
\begin{equation}\label{eq:E(P)}
E(P)=|P|^n
\sum_{\substack{ \c\in \cO^n\\
\c\neq \0\\
|\c|\ll \hat Y |P|^{-1}J(\Theta)
}} 
\sum_{\substack{
r\in \cO\\ |r|=\hat Y\\
\text{$r$ monic}}
} 
|r_M|^{-n}
\int_{|\theta|=\hat \Theta} 
S_{r,M,\b}(\c)
I_{r_M}(\theta;\c)
 \d \theta,
\end{equation}
where $I_{r_M}(\theta;\c)\ll J(\Theta)^{-n/2}$.
Let $S$ be a set of finite primes to be decided upon in due course, but which contains all prime divisors of $M$.
Any $r\in \cO$ can be written $r=b_1'b_1r_2$ where 
$b_1'$ is square-free such that $\vp\mid b_1' \Rightarrow \vp\in S$
 and $b_1$ is square-free and coprime to $S$.
According to Lemma~\ref{lem:multi2} there is a factorisation 
 $M=M_1M_2M_3 $ for $M_1,M_2,M_3\in \cO$ such that 
$M_1\mid (b_1')^\infty$, $M_2\mid r_2^\infty$ and $(M_3,r)=1$, 
together with  $\b_1,\b_2,\b_3\in (\cO/M\cO)^n$ such that 
\begin{equation}\label{eq:factorise}
S_{r,M,\b}(\c)
=
S_{b_1,1,\0
}(\c)
S_{b_1',M_1,\b_1
}(\c)
S_{r_2,M_2,\b_2
}(\c)
\psi\left(\frac{-\c.\b_3}{M_3}\right).
\end{equation}
The vectors $ \b_1,\b_2$ and $\b_3$ depend only on the value of $b_1\bmod{M}$.
In \S \ref{s:sqf} we will consider the effect of $S_{b_1,1,\0
}(\c)
$ on $E(P)$. Later, in \S \ref{s:sqfull}, we will consider the contribution 
from $S_{r_2,M_2,\b_2
}(\c)
$.

\section{Contribution from square-free moduli}\label{s:sqf}

It will be convenient to define
$$
\mathcal{O}^\sharp= \left\{
b\in \mathcal{O}: \mbox{$b$ is monic and square-free}\right\}.
$$
Recalling the expression \eqref{eq:E(P)} 
and the subsequent factorisation \eqref{eq:factorise} of the exponential sum involved, it follows that there exists
$\b_1,\b_2\in (\cO/M\cO)^n$ and $b_0\in (\cO/M\cO)^*$ such that 
\begin{equation}\label{eq:E(P)-1}
\begin{split}
E(P)
\ll~& \frac{|P|^n}{\hat Y^{(n-1)/2}}
\sum_{\substack{ \c\in \cO^n\\
\c\neq \0\\
|\c|\ll \hat Y |P|^{-1}J(\Theta)
}} 
\sum_{\substack{b_1'\in \cO^\sharp\\ \vp\mid b_1'\Rightarrow \vp\in S}}\\
&\times
\sum_{\substack{
r_2\in \cO\\ |b_1'r_2|\leq \hat Y}} 
\frac{
|S_{b_1',M_1,\b_1
}(\c) S_{r_2,M_2,\b_2
}(\c)|}
{|b_1'r_2|^{(n+1)/2}}
\int_{|\theta|=\hat \Theta}
|\Sigma(Y, \theta)| \d\theta,
\end{split}
\end{equation}
where
\begin{equation}\label{eq:access}
\Sigma(Y,\theta)
=
\sum_{\substack{
b_1\in \cO^\sharp
\\
(b_1,S)=1\\
|b_1'b_1r_2|=\hat Y 
\\ b_1\equiv b_0\bmod{M}} }
\frac{S_{b_1,1,\0}(\c)
I_{b_1(b_1'r_2)_M}(\theta;\c)}{|b_1|^{(n+1)/2}
}.
\end{equation}
Here we have observed that 
$(b_1'b_1r_2)_M=b_1(b_1'r_2)_M$ since $(b_1,M)=1$.
Our main job in this section is to estimate 
$\Sigma(Y,\theta)$ whenever $\c$ is suitably generic.

In what follows we will 
put $b=b_1$ and redefine $b_1'r_2$ to be $d$, for simplicity. 
Putting 
$S_{b}(\c)=S_{b,1,\0}(\c)$, we  have 
\begin{align*}
S_{b}(\c)
&=
\sumstar_{\substack{
|a|<|b|} }
\sum_{\substack{\y\in \cO^n\\|\y|<|b|
}} 
\psi\left(\frac{aF(\y)-\c.\y}{r}\right).
\end{align*}
Let $\Delta_F\in \cO$ denote the non-zero discriminant of $F$. Assuming that $F^*(\c)\neq 0$
we shall take $S$ to be the set of primes dividing $\Delta_FMF^*(\c)$.
Alternatively, if $F^*(\c)=0$ but $\nabla F^*(\c)\neq \0$, then we will take 
 $S$ to be the set of primes dividing $\Delta_FM\nabla F^*(\c)$.
Lemma \ref{lem:multi2} shows that the sum $S_b(\c)$ is a multiplicative function of $b$. 
When $b=\vp$ for a prime $\vp$,  the sum is a complete exponential sum over the finite field $\FF_\vp$. 
It then follows from \eqref{eq:alpha=1}
that  $S_{\vp}(\c)
\ll  |\vp|^{(n+1)/2}|(\vp,\nabla F^*(\c))|^{1/2}$.
Hence
\begin{align*}
\sum_{\substack{
b\in \cO^\sharp
\\
(b,S)=1\\
|bd|=\hat Y 
\\ b\equiv b_0\bmod{M}} }
\frac{|S_{b}(\c)|
}{|b|^{(n+1)/2}
}
&\ll  \hat Y^\ve
\sum_{\substack{
b\in \cO^\sharp
\\
(b,S)=1\\
|bd|=\hat Y}}
|(b,\nabla F^*(\c))|^{1/2}
\end{align*}
in \eqref{eq:access}.
According to our definition of $S$ this is 
$O(|d|^{-1}\hat Y^{1+\ve})$
if  $\nabla F^*(\c)\neq \0$
and 
$O(|d|^{-3/2}\hat Y^{3/2+\ve})$
if 
$\nabla F^*(\c)=\0$.
Recalling the bound in Lemma \ref{lem:I-hard}
for 
$I_{bd_M}(\theta;\c)$
and the definition \eqref{eq:J-notation} of $J(\Theta)$, 
this leads to the following ``easy'' estimate for $\Sigma(Y,\theta)$.

\begin{lemma}\label{lem:goal-easy}
Let $\ve>0$. 
If   $\nabla F^*(\c)\neq \0$ and $S$ is the set of primes dividing $\Delta_F M\nabla F^*(\c)$, then
$$
\Sigma(Y,\theta)\ll  \frac{\hat Y^{1+\ve}}{|d|J(\Theta)^{n/2}}.
$$
If $\nabla F^*(\c)=\0$ and $S$ is the set of primes dividing 
$\Delta_F M$, then 
$$
\Sigma(Y,\theta)\ll  \frac{\hat Y^{3/2+\ve}}{|d|^{3/2}J(\Theta)^{n/2}}.
$$
\end{lemma}

The implied constants in these estimates are allowed to depend on the choice of $\ve$, a convention that we shall henceforth adhere to.
Lemma \ref{lem:goal-easy} does not take advantage of any cancellation in the sum over $b\in \cO^\sharp$ coming from sign changes 
in the exponential sum $S_{b}(\c)$. 
The following ``hard'' estimate does so under suitable hypotheses. Its proof will occupy the rest of this section.

\begin{lemma}\label{lem:goal}
Assume that $n$ is even and that $F^*(\c)\neq 0$. Let $S$ be the set of primes dividing $\Delta_F MF^*(\c)$. 
Assume, furthermore, that $
\hat Y > \sqrt{q|d||P||\c|}$.
Then for any $\ve>0$ we have 
$$
\Sigma(Y,\theta)\ll \frac{(|\c| \hat Y)^\ve}{|d|^{1/2}} \frac{\hat Y^{1/2}}{J(\Theta)^{n/2-1/2}}.
$$
\end{lemma}

At this stage it might be useful to compare  Lemmas \ref{lem:goal-easy} and \ref{lem:goal} for typical values of 
$Y,\Theta, \c$ satisfying \eqref{eq:conditions on Y},
by which we mean that  $Y\sim Q$, 
$\hat \Theta\sim (\hat Y \hat Q)^{-1}$
and 
$|\c|\sim |P|^{1/2}$.
But then  $J(\Theta)\sim 1$ and the bound in Lemma \ref{lem:goal} is roughly 
of order $(\hat Q/|d|)^{1/2}$, while that in 
 Lemma \ref{lem:goal-easy} is 
of order $\hat Q/|d|$.

We begin the proof of Lemma \ref{lem:goal}  by writing 
$$
\Sigma(Y,\theta)
=
\int_{K_\infty^n} 
w\left(t^L(\x-\x_0)\right)
\psi \left(\theta P^3F(\x)\right)
\Sigma(Y;\x)
\d \x,
$$
where
$$
\Sigma(Y;\x)=
\sum_{\substack{
b\in \cO^\sharp
\\
(b,S)=1\\
|b|=\hat Y/|d|\\b\equiv b_0\bmod{M}
} }
\frac{S_{b}(\c)
}{|b|^{(n+1)/2}}
\psi\left(\frac{P \c.\x/d_M}{b }\right).
$$
Here we recall  that $|\x|<1$ for $\x\in K_\infty^n$ such that $w(t^L(\x-\x_0))\neq 0$. 
We  detect the condition 
$b\equiv b_0\bmod{M}$ by summing over  Dirichlet characters $\eta_1 \bmod{M}$.
Letting $D_1=(\cO/M\cO)^*$, this gives
\begin{align*}
\Sigma(Y;\x)=
\frac{1}{\#D_1}\sum_{\eta_1\bmod{M}}\overline{\eta_1(b_0)}\sum_{\substack{
b\in \cO^\sharp
\\
(b,S)=1\\
|b|=\hat Y/|d|
} }
\frac{\eta_1(b)S_{b}(\c)
}{|b|^{(n+1)/2}}
\psi\left(\frac{P \c.\x/d_M}{b }\right).
\end{align*}

Next, 
we let  $J\in \ZZ$ be such that
$$
\hat J=q^J=\max\left\{1, \frac{q|P||\c|}{\hat Y}\right\}.
$$
In particular $J\geq 0$. Typically we 
expect  $\hat J$ to be rather  small.
For any $b$ arising in $\Sigma(Y;\x)$ let us put $K=\deg(b)$, so that 
$$
\hat K=q^K=\frac{\hat Y}{|d|}.
$$
Lemma \ref{lem:goal} is stated under the assumption that $\hat Y \geq  \sqrt{q|d||P||\c|}$, 
which is equivalent to    $J\leq  K$. 
Let us put  $x=t^{-1}$ for the prime at infinity and $A=\FF_q[x] \subset \mathcal{O}_\infty$.
Then, since $b$ is  monic, there exist
$c_1,\dots,c_K\in \FF_q$ such that 
$$
b=
\underbrace{t^K+c_1t^{K-1}+\dots+c_{J-1}t^{K-J+1}}_{=t^Ka}
+\underbrace{c_Jt^{K-J}+\dots+c_K}_{=t^{K-J}b'},
$$
where  $a\in (A/x^J A)^*$ and 
 $b'\in A$. 
 Thus $b=t^K(a+x^J b')$ with $|b'|\leq 1$ and $|a|=1$.
  But then it follows that 
 \begin{align*}
 \psi\left(\frac{P \c.\x/d_M}{b }\right)&=
 \psi\left(\frac{P \c.\x}{d_M}\left\{\frac{1}{t^K(a+x^J b')}-\frac{1}{t^Ka}
 \right\}\right) 
 \psi\left(\frac{P \c.\x/d_M}{t^Ka}\right) \\
 &= \psi\left(\frac{P \c.\x/d_M}{t^Ka}\right),
 \end{align*}
since 
\begin{align*}
\left|\frac{P \c.\x}{d_M}\left\{\frac{1}{t^K(a+x^J b')}-\frac{1}{t^Ka}
 \right\}\right|
&<
\frac{|P||\c|}{|t^Kad_M|}
 \left|-\frac{x^J b'}{a}+\dots\right| \\
&\leq
\frac{|P||\c|}{\hat J \hat K  |d|}\\
&\leq q^{-1}.
\end{align*}
The conclusion of this is that the character $\psi$ 
in $\Sigma(Y;\x)$ 
only depends on the value of $b/t^K \bmod{x^J}$.

Putting $D_2=(A/x^JA)^*$, 
it follows that 
$$
\Sigma(Y;\x)
=
\frac{1}{\#D_1}\sum_{\eta_1\bmod{M}}\overline{\eta_1(b_0)}\sum_{a\in D_2}
\psi\left(\frac{P \c.\x/d_M}{t^Ka}\right)
\hspace{-0.2cm}
\sum_{\substack{
b\in \cO^\sharp
\\
(b,S)=1\\
|b|=\hat Y/|d|\\
t^{-K}b\equiv a \bmod{x^J}
} }
\frac{\eta_1(b)S_{b}(\c)
}{|b|^{(n+1)/2}}.
$$
Introducing Dirichlet characters $\chi:D_2\rightarrow \CC^*$ 
to detect the congruence condition in the inner sum, we deduce that 
\begin{align*}
\Sigma(Y;\x)
=~&
\frac{1}{\#D_1\#D_2}\sum_{\eta_1\bmod{M}}
\sum_{\chi \bmod{x^J}}\\
&\times\sum_{a\in D_2}
\psi\left(\frac{P \c.\x/d_M}{t^Ka}\right)
\bar{\eta_1(b_0)
\chi(a)}
\Sigma_0(\eta_1,\chi;Y),
\end{align*}
where
$$
\Sigma_0(\eta_1,\chi;Y)
=\sum_{\substack{
b\in \cO^\sharp
\\
(b,S)=1\\
|b|=\hat Y/|d|\\
} }
\frac{\eta_1(b)\chi(t^{-K}b)S_{b}(\c) 
}{|b|^{(n+1)/2}}.
$$
In conclusion,  we have therefore established the identity
\begin{equation}\label{eq:cross}
\begin{split}
\Sigma(Y,\theta)
=~&
\frac{1}{\#D_1\#D_2}\sum_{\eta_1\bmod{M}}\bar{\eta_1(b_0)}
\sum_{\chi \bmod{x^J}}
\Sigma_0(\eta_1,\chi;Y)\\
&\times
\sum_{a\in 
D_2}\bar{\chi(a)}I_{t^Kad_M}(\theta;\c).
\end{split}
\end{equation}

Our first concern is an estimate for the inner sum over $a$. It is easy to see that 
\begin{equation}\label{eq:bath}
\begin{split}
\frac{1}{\#D_2}\sum_{\chi \bmod{x^J}}\left|
\sum_{a\in D_2}\bar{\chi(a)}I_{t^Kad_M}(\theta;\c)\right|
&\ll \hat J J(\Theta)^{-n/2}\ll J(\Theta)^{1-n/2} ,
\end{split}
\end{equation}
since the size constraint on $\c$ in 
\eqref{eq:E(P)-1} gives
$\hat J=\max\{1,|P||\c|/\hat Y\}\ll J(\Theta)$.
It turns out that this bound does not suffice for \eqref{eq:EP} when $n=8$
and it is necessary to produce a bound which  takes advantage of non-trivial averaging over $a$. 
This is achieved in the following result. 

\begin{lemma}\label{lem:red}
We have 
$$ 
\frac{1}{\#D_2}\sum_{\chi \bmod{x^J}}\left|
\sum_{a\in D_2}\bar{\chi(a)}I_{t^Kad_M}(\theta;\c)\right|
\ll J(\Theta)^{1/2-n/2}.
$$
\end{lemma}

\begin{proof}
Let $\chi \bmod{x^J}$ be a Dirichlet character.  Opening up 
$I_{t^Kad_M}(\theta;\c)$, we
deduce from \eqref{eq:small} that
\begin{align*}
 \sum_{a\in D_2}\chi(a)I_{t^Kad_M}(\theta;\c)=
 ~& \frac{1}{{\hat L}^{n}}\sum_{a\in D_2}\chi(a)
\psi\left(
\frac{P \c.\x_0}{t^Kad_M}\right)
J_G\left(\theta P^3,\frac{Pt^{-L}\c}{t^Kad_M}\right),
\end{align*}
in the notation of \eqref{eq:J}, 
where $G(\x)=F(\x_0+t^{-L}\x)$. Lemma \ref{lem:small} implies that 
\begin{align*}
J_G\left(\theta P^3,\frac{Pt^{-L}\c}{t^Kad_M}\right)
=\int_{\Omega_a}\psi\left(\theta P^3G(\x)+
\frac{Pt^{-L} \c.\x}{t^Kad_M}\right)\d\x,
\end{align*}
where 
$$\Omega_a=
\left\{\x\in \TT^n:  \left|\theta P^3\nabla G(\x)+\frac{Pt^{-L}\c}{t^Kad_M}\right|\ll 
J(\Theta)^{1/2} \right\}.
$$ 
It follows from  \eqref{eq:brownies} that $\meas(\Omega_a)\ll J(\Theta)^{-n/2}.$

Let $\ve>0$ and choose $J_0\in \ZZ$ such that $\hat J_0$ has order of magnitude $J(\Theta)^{1/2+\ve}.$
If $J_0>J$ then Lemma \ref{lem:red} follows from \eqref{eq:bath}.
Alternatively, we may proceed under the assumption that $J/2\leq J_0\leq J$.
Recall that $x=t^{-1}$ and suppose that  $a\equiv a'\bmod{x^{J_0}}$, for $a,a'\in D_2$. Then 
\begin{align*}
\left| \frac{Pt^{-L} \c}{t^Kad_M}-\frac{Pt^{-L} \c}{t^Ka'd_M}\right|\leq 
\hJ\left|\frac{a-a'}{aa'}\right|\leq \frac{\hat J}{\hat J_0}\ll J(\Theta)^{1/2-\ve}.
\end{align*}
Hence  the set $\Omega_a$ only depends on the value of $a\bmod{x^{J_0}}$.

Let us write $a=a_0+x^{J_0}a_1$, where 
$a_0\in (A/x^{J_0}A)^*$ and $a_1\in A/x^{J-J_0}A $. 
Then 
\begin{align*}
 \sum_{a\in D_2}\chi(a)I_{t^Kad_M}(\theta;\c)
 =~&
 \sum_{a_0\in (A/x^{J_0}A)^*}\sum_{a_1\in A/x^{J-J_0}A}\chi(a_0+x^{J_0}a_1)
\\
&\times 
\frac{1}{{\hat L}^{n}}
\int_{\Omega_{a_0}}\psi\left(\theta 
P^3G(\x)\right)
\psi\left(\frac{P \c.(\x_0+t^{-L}\x)}{t^K(a_0+x^{J_0}a_1)d_M}\right)\d\x.
\end{align*}
For fixed   $a_0\in (A/x^{J_0} A)^*$ and $\x\in \Omega_{a_0}$ we proceed to examine the sum
\begin{align*}
S(\x)= \sum_{a_1\in A/x^{J-J_0}A}\psi\left(\frac{P \c.\y}{t^K(a_0+x^{J_0}a_1)d_M}\right)\chi(1+x^{J_0}a_1\bar{a_0}),
\end{align*}
where 
$\y=\x_0+t^{-L}\x$ and 
$\bar{a_0}$ denotes the multiplicative inverse of $a_0\bmod{x^{J-J_0}}$.

Let $\phi_\chi$ be the additive character defined on $A/x^{J-J_0}A$ via
$$
\phi_\chi(a)=\chi(1+x^{J_0}a).
$$
This must  be a twist of the standard additive character. Thus there exists
an element $a_\chi\in A/x^{J-J_0}A$ such that 
$$
\phi_\chi(a) =\psi\left(\frac{a_\chi a}{x^{J-J_0}}\right),
$$
for any $a\in A/x^{J-J_0}A$.
This gives a surjective homomorphism 
$$
\phi:\Hom\left((A/x^{J}A)^*,\CC^*\right)\rightarrow A/x^{J-J_0}A, 
$$ 
defined by $\phi(\chi)=a_\chi$,
with kernel isomorphic to $\Hom((A/x^{J_0}A)^*,\CC^*)$. 
We conclude that 
\begin{align}
\label{eq:hope}
S(\x)=  \sum_{a_1\in A/x^{J-J_0}A}\psi\left(\frac{P \c.\y}{t^K(a_0+x^{J_0}a_1)d_M}\right)
  \psi\left(\frac{a_\chi a_1\bar{a_0}}{x^{J-J_0}}\right).
\end{align}
Observe that $|{P \c.\y}/({t^Kd_M})| \leq \hJ$.
Hence
\begin{align*} 
\psi\left(\frac{P \c.\y}{t^K(a_0+x^{J_0}a_1)d_M}\right)&=\psi\left(\frac{P 
\c.\y}{t^Ka_0(1+x^{J_0}a_1\overline{a_0})d_M}\right)\\
&=\psi\left(\frac{P \c.\y(1-a_1\bar{a_0}x^{J_0})}{t^Ka_0d_M}\right)
\\
&=\psi\left(\frac{P 
\c.\y}{t^Ka_0d_M}\right)\psi\left(\frac{-P \c.\y a_1\bar{a_0}x^{J_0}}{t^Ka_0d_M}\right)
\end{align*}
and 
\begin{align*}
 \psi\left(\frac{-P \c.\y a_1\bar{a_0}x^{J_0}}{t^Ka_0d_M}\right)=\psi\left(\frac{a''a_1\bar{a_0}^2}{x^{J-J_0}}\right),
\end{align*}
for some $a''\in A$. Applying this reasoning in \eqref{eq:hope}, we are led to the identity
\begin{align*}
S(\x)= 
\psi\left(\frac{P 
\c.\y}{t^Ka_0d_M}\right)
\sum_{a_1\in A/x^{J-J_0}A}\psi\left(\frac{a_1\bar{a_0}(a_\chi+a''\bar{a_0})}{x^{J-J_0}}\right),
\end{align*}
where  $a_\chi$ and $a''$ are independent of the choices of $a_0$ and $a_1$. 

For  fixed $a_0$ we deduce that $S(\x)=0$ unless  
$a_\chi\equiv a'''\bmod{x^{J-J_0}} 
$, where $a'''=-a''\bar{a_0}\bmod{x^{J-J_0}}$,  in which case $|S(\x)|\leq \hat J/\hat J_0$.  
 However, for  fixed $a'''\in A/x^{J-J_0}A$ we have 
$
\#\{\chi\in\phi^{-1}(a''')\}\leq\#\{\chi\in\phi^{-1}(0)\}\leq \hat{J_0},
$ 
since $\phi$ is a homomorphism. 
Thus
\begin{align*}
\frac{1}{\#D_2}\sum_{\chi \bmod{x^J}}\left|
\sum_{a\in D_2}\chi(a)I_{t^Kad_M}(\theta;\c)\right|
\ll~& \frac{1}{\hat J} 
\sum_{\chi \bmod{x^J}}
\sum_{a_0\in (A/x^{J_0}A)^*}\\
&\times
\left| \int_{\Omega_{a_0}} \psi\left( \theta P^3 G(\x)\right) S(\x) \d\x\right|
\\
\ll~& 
\sum_{a_0\in (A/x^{J_0}A)^*} \meas(\Omega_{a_0})\\
\ll~& \hat J_0J(\Theta)^{-n/2}.
\end{align*}
This completes the proof of the lemma, since $\hat J_0$ has order $J(\Theta)^{1/2+\ve}.$
\end{proof}

It is now time to  start analysing the sum $\Sigma_0(\eta_1,\chi;Y) $ for fixed 
Dirichlet  characters $\eta_1:D_1\rightarrow \CC^*$ and $\chi: D_2\rightarrow\CC^*$. 
Let us define a further character $\eta_2:\mathcal{O}\rightarrow \CC^*$, given by  
$
\eta_2(r)=\chi(r/t^{\deg r} )
$
for any $r\in \mathcal{O}$.
This a multiplicative character of order at most $\hat J$. We proceed to bound the sum
$$\sum_{\substack{b\in \cO^\sharp\\ 
(b,S)=1\\
|b|\leq \hat Z 
}} 
\frac{\eta_1(b)\eta_2(b)S_{b}(\c)
}{|b|^{(n+1)/2}}$$
for any $Z\geq 1$, 
where $S$ is  the set of primes dividing $\Delta_F MF^*(\c)$.

Let $X\subset \PP_K^{n-1}$ denote the smooth and  projective  hypersurface $F=0$ defined over $K$ and 
let $X_\c\subset \PP_K^{n-2}$ denote the projective hypersurface cut out from $X$ by the hyperplane
 $\c.\x=0$.   Since $F^*(\c)\neq 0$ it follows that $X_\c$ is smooth.
Moreover, we have  $\dim(X)=n-2$ and $\dim (X_\c)=n-3$.
We begin our analysis of $S_\vp(\c)$ with an application of   Hooley \cite[Lemma 7 and Eq.~(86)]{oct}. This shows that 
\begin{equation}\label{eq:bonus}
S_\vp(\c)=|\vp|\left\{
|\vp| \#{X}_{\c,\vp}(\FF_\vp) - \#{X}_{\vp}(\FF_\vp)+1
\right\},
\end{equation}
for any prime $\vp$.  It now follows from \eqref{eq:counting'} that 
$$
S_\vp(\c)=
(-1)^{n-3} |\vp|^2\sum_{j=1}^{b_{n-3}}\omega_{n-3,j} +O(|\vp|^{n/2}),
$$
for any finite prime $\vp\not \in S$, 
where for any prime $\ell\nmid q$ the number 
$b_{n-3}$ is the dimension of the 
 middle cohomology group $H_\ell^{n-3}(X_\c)=\het^{n-3}(\bar X_\c,\QQ_\ell)$ (as a vector space over $\QQ_\ell$)
and 
 $\omega_{n-3,j}$ are the eigenvalues of the Frobenius endomorphism acting on it.
The dimension  
$b_{n-3}$  is independent of the choice of $\ell$ and is bounded in terms of $n$.
Moreover,  $|\omega_{n-3,j}|=|\vp|^{(n-3)/2}$ for each index  $1\leq j\leq b_{n-3}$. 

We proceed to study the 
Dirichlet series
\begin{equation}\label{eq:dirichlet}
F(s)=\sum_{\substack{b\in \cO^\sharp\\ (b,S)=1}}
\frac{\eta_1(b)\eta_2(b)S_{b}(\c)}{|b|^s}
=\prod_{\vp\not \in S} \left(
1+\frac{\eta_1(\vp)\eta_2(\vp)S_\vp(\c)}{|\vp|^s}
\right),
\end{equation}
which is defined 
for $\sigma=\Re(s)>(n+3)/2$.
Let  $\vp \not\in S$
and $\sigma>n/2+1$. Then, 
with \S\S   \ref{s:weil}--\ref{s:twisting} to hand,
\eqref{eq:bonus} implies that 
\begin{equation}\label{eq:stone}
\begin{split}
1+\frac{\eta_1(\vp)\eta_2(\vp)S_\vp(\c)}{|\vp|^s}
=~&
\left(1
+\eta_1(\vp)\eta_2(\vp)\frac{(-1)^{n-3}}{ |\vp|^{s-2}}\sum_{j=1}^{b_{n-3}}\omega_{n-3,j}\right)\\
&\times\left(1 +O(|\vp|^{n/2-\sigma}+|\vp|^{
n+1-2\sigma})\right)\\
=~&L_\vp(\eta\otimes H_\ell^{n-3}(X_\c),s-2)^{(-1)^{n-3}}\\
&\times
\left(1 +O(|\vp|^{n/2-\sigma})\right),
\end{split}
\end{equation}
where we view  $\eta=\eta_1\otimes\eta_2$ as a Galois representation by class field theory.

We may now appeal to the contents of \S \ref{s:twisting} where some of the analytic properties of 
the global $L$-function $L(\eta \otimes H_\ell^{n-3}(X_\c),s-2)$ are recorded. 
When $\vp\in S$ it follows from our discussion in \S \ref{s:kahn}
that 
$$
L_\vp(\eta\otimes H_\ell^{n-3}(X_\c),s-2)=1+O(|\vp|^{(n+1)/2-\sigma}),
$$
since the inverse roots have modulus at most
$|\vp|^{(n-3)/2}$.
Hence, on recalling the definition of the associated global $L$-function,
we finally obtain
\begin{equation}\label{eq:continue}
F(s)= 
L(\eta\otimes H_\ell^{n-3}(X_\c),s-2)^{(-1)^{n-3}}E(s), \quad (\sigma>n/2+1),
\end{equation}
where
\begin{equation}\label{eq:E}
E(s)=\prod_{\vp\not \in S}
\left(1 +O(|\vp|^{n/2-\sigma})\right)
\prod_{\vp\in S}
\left(1 +O(|\vp|^{(n+1)/2-\sigma})\right).
\end{equation}
Note that $E(s)$ is holomorphic and bounded for 
$\sigma>n/2+1$.

We will need a decent  bound for the absolute value of the function $F(s)$ well inside its domain of analytic 
continuation. 
This  is achieved in the following result. 

\begin{lemma}\label{lem:size of F}
Assume that $n$ is even and let $\ve>0$.  Then 
for $\sigma\geq 1/2+\ve$ 
we have 
$
\left|F\left(s+(n+1)/{2}\right)\right|\ll_\ve |\c|^\ve.
$
\end{lemma}

\begin{proof}
Recalling \eqref{eq:continue},  the fact that   $n$ is even
implies that 
$F(s+(n+1)/2)=G(s)E(s+(n+1)/2)$ for
$\sigma>1/2$, with 
$$
G(s)=
L\left(\eta\otimes H_\ell^{n-3}(X_\c),s+\tfrac{n-3}{2}\right)^{-1},
$$
where
 $\eta=\eta_1\otimes\eta_2$.
It follows from   \eqref{eq:YEAH-twist} that 
$$
G(s)=
\frac{P_0(q^{-s-(n-3)/2})P_2(q^{-s-(n-3)/2})}{P_1(q^{-s-(n-3)/2})},
$$
with $P_{k}=P_{k,n-3}\in \ZZ[T]$ as in \S \ref{s:twisting} for $k\in \{0,1,2\}$.
Furthermore, if  we put $e_k=\deg P_k$ then it follows that $e_0,e_2=O(1)$ and 
\begin{equation}\label{eq:e1-bound}
e_1\ll 1+\log |F^*(\c)| \ll 1+\log |\c|,
\end{equation}
by \eqref{eq:degree-Hm-twist}.
Moreover, the inverse roots of $P_k$ have 
absolute value $q^{(n-3+k)/2}$.
It is now clear that $G(s)$ is holomorphic in the half-plane $\sigma>1/2$ and that in this region its only zeros come from the zeros of 
$P_2(q^{-s-(n-3)/2})$, which  are located on the line $\sigma=1$. We have 
$$
F\left(s+\tfrac{n+1}{2}\right)= P_2(q^{-s-(n-3)/2}) H(s),
$$
with 
$$
H(s)=
E\left(s+\tfrac{n+1}{2}\right) \frac{P_0(q^{-s-(n-3)/2})}{P_1(q^{-s-(n-3)/2})}.
$$
Now it is obvious that
$
|P_2(q^{-s-(n-3)/2})|\leq (1+q^{-\sigma+1})^{e_2} \ll 1, 
$
for $\sigma>1/2$. Hence it suffices to establish the bound in the lemma for $H(s)$.

We will produce a good bound when $\sigma>1$ together with a weaker bound which is valid for $\sigma>1/2$. In the 
familiar way (cf.~Titchmarsh \cite[Chapter~XIV]{tit}), we will then use the Hadamard three circle theorem to establish 
the final bound recorded  in the statement of the lemma. 
Our trivial bound is based on  \eqref{eq:dirichlet}.  Thus it follows from  \eqref{eq:stone} that 
there is a constant $c>0$ such that 
\begin{align*}
\left|\log F\left(s+\tfrac{n+1}{2}\right)\right|
&\leq 
\sum_{j=1}^{b_{n-3}}\sum_{\vp} \sum_{\alpha\geq 1} \frac{1}{\alpha|\vp|^{\alpha \sigma}}
+\sum_{\vp}\sum_{\alpha\geq 1} 
\frac{1}{\alpha} \left(\frac{c}{|\vp|^{\sigma+1/2}}\right)^\alpha\\
&\ll \log Z(\sigma),
\end{align*}
for $\sigma>1$, where
$Z(s)$ is the ordinary zeta function of $K=\FF_q(t)$. 
It easily follows that 
\begin{equation}\label{eq:H-easy}
|\log H(s)|\ll \log Z(\sigma), \quad (\sigma>1).
\end{equation}

Next, 
for $\sigma>1/2$, 
it follows from \eqref{eq:E} that  $E(s+(n+1)/2)\ll Z(\sigma+1/2)^c$ for some absolute constant $c>0$. 
Hence we obtain
$$
|H(s)|\ll Z\left(\sigma+\tfrac{1}{2}\right)^c (1-q^{1/2-\sigma})^{-e_1}
$$
for $\sigma>1/2$, whence
$$
\Re \log H(s)=\log |H(s)|\ll 
\log Z\left(\sigma+\tfrac{1}{2}\right) +e_1
$$
in this region.  Note that $\log H(s)$ is analytic in the half-plane $\sigma>1/2$.
We apply the Borel--Carath\'eodory theorem to $\log H(s)$ with circles of centre
$3/2+it_0$ and radii $1-\ve/2$ and $1-\ve$.  This leads to the conclusion that
\begin{equation}\label{eq:H-hard}
|\log H(s)|\ll \frac{1}{\ve}\left\{\log Z\left(\sigma+\tfrac{1}{2}\right)+e_1\right\}, \quad (\sigma\geq 
\tfrac{1}{2}+\ve).
\end{equation}

We now refine this bound by applying  the Hadamard three circle theorem to $\log H(s)$.
Let $\sigma_0=\sigma_0(\ve)$ and let $s=\sigma+it$ with $1/2+\ve\leq \sigma\leq 1+\ve/2$.
We take circles with centre $\sigma_0+it$ and radii $r_1=\sigma_0-1-\ve/2$, $r_2=\sigma_0-\sigma$ and 
$r_3=\sigma_0-1/2-\ve/2$.
Combining \eqref{eq:H-easy} and \eqref{eq:H-hard}, we deduce the existence of constants $c_1(\ve),c_2(\ve)>0$ such that 
$$
|\log H(s)| \leq c_1(\ve)^{1-\beta} \left(c_2(\ve)e_1\right)^\beta,
$$
where
$$
\beta=\frac{\log r_2/r_1}{\log r_3/r_1}= 2-2\sigma+\ve+O\left(\frac{1}{\sigma_0}\right)\leq 
1-\ve+O\left(\frac{1}{\sigma_0}\right).
$$
We take  $\sigma_0$ sufficiently large to ensure that $\beta\leq 1-\ve/2<1$.
Recalling the bound \eqref{eq:e1-bound} for $e_1$, 
all of this is now seen to give
$$
|H(s)| \leq c(\ve)^{1+(\log |\c|)^\beta}, \quad (\sigma\geq \tfrac{1}{2}+\ve),
$$
for an appropriate constant $c(\ve)>0$.
The statement of the lemma easily follows.
\end{proof}

We are now ready to establish the following estimate, which once combined with \eqref{eq:cross} and Lemma \ref{lem:red},
clearly completes the proof of Lemma 
\ref{lem:goal}.

\begin{lemma}\label{lem:home}
Assume that $n$ is even  and  $F^*(\c)\neq 0$.  Then 
for any  $\ve>0$ we have
$$
\sum_{\substack{b\in \cO^\sharp\\ 
(b,\Delta_F MF^*(\c))=1\\
|b|\leq \hat Z 
}} 
\frac{\eta_1(b)\eta_2(b)S_{b}(\c)
}{|b|^{(n+1)/2}}
\ll (|\c| \hat J)^\ve\hat Z^{1/2+\ve}.
$$
\end{lemma}

\begin{proof}
 It follows from Perron's formula 
that the sum to be estimated is equal to
 $$
\sum_{k\leq  \hat Z}
\frac{a_k}{k^{(n+1)/2}}=
\frac{1}{2\pi i}\int_{2-i\infty}^{2+i\infty} F\left(s+\frac{n+1}{2}\right) \frac{\hat Z^{s}\d s}{s},
$$
where 
$$
a_k=
\sum_{\substack{b\in \cO^\sharp, ~|b|=k\\ 
(b,\Delta_F MF^*(\c))=1}} 
\eta_1(b)\eta_2(b)S_{b}(\c)
$$
and $F(s)$ is the Dirichlet series \eqref{eq:dirichlet}. The latter 
is  absolutely convergent and bounded   for $\sigma>(n+3)/2$.
Noting that 
$$
\frac{1}{2\pi i} \int_{2\pm iT}^{2\pm i \infty} \frac{u^s \d s}{s}=O\left(\frac{u^2}{T |\log u|}\right),
$$
this may clearly be rewritten as
$$
\frac{1}{2\pi i}\int_{2-iT}^{2+iT} F\left(s+\frac{n+1}{2}\right) \frac{\hat Z^{s}\d s}{s} +
O\left(\frac{\hat Z^3}{T}\right).
$$

Let $\ve>0$. 
According to \eqref{eq:continue}, the function  $F(s+(n+1)/2)$ has an analytic continuation to 
the half-plane $\sigma\geq 1/2+\ve$ on which it is holomorphic. 
We   change the contour of integration so that it consists of the remaining three sides of the rectangle $R$
with vertices $2-iT, 1/2+\ve-iT, 1/2+\ve+iT$  and $2+iT$. We will  use Lemma 
\ref{lem:size of F} to estimate the contributions from the various contours. 
Thus, to begin with,  the horizontal contours are seen to  contribute 
\begin{align*}
&\ll
\frac{(|\c|\hat J)^\ve}{T}
\int_{\frac{1}{2}+\ve}^{2}  
 \hat Z^{\sigma}\d \sigma 
\ll
\frac{(|\c|\hat J)^\ve \hat Z^{2}}{T}.
\end{align*}
The remaining contour makes the overall contribution
\begin{align*}
&\ll
|\c|^\ve \hat Z^{1/2+\ve}
\int_{\frac{1}{2}+\ve-iT}^{\frac{1}{2}+\ve+iT}  \frac{\d t }{1+|t|}
\ll |\c|^\ve \hat Z^{1/2+\ve}T^{\ve}.
\end{align*}
Combining our estimates and taking $T=\hat Z^{3}$, we therefore arrive at the statement of the lemma.
\end{proof}

\begin{remark}\label{r:odd}
Let us put $m=n-3=\dim X_\c$. Our discussion so far has focussed on the case of  even $n$ (i.e.\ $m$ odd).
The purpose of this remark is to highlight the difficulty of dealing with odd $n$ (i.e.\ $m$ even).
Returning to  the proof of Lemma \ref{lem:size of F} and applying 
\eqref{eq:YEAH-twist}, when $m$ is even  we instead have 
$F(s+(n+1)/2)=G(s)H(s)$ for
$\sigma>1/2$, with 
 $H(s)$  holomorphic and bounded in this half-plane and where
$$
G(s)=
L\left(\eta_1\otimes \eta_2\otimes H_\ell^{m}(X_\c),s+\tfrac{m}{2}\right)=
\frac{P_1(q^{-s-m/2})}{P_0(q^{-s-m/2})P_2(q^{-s-m/2})},
$$
for suitable polynomials $P_0, P_1, P_2\in \ZZ[T]$. 
(Recall that for odd $m$ it was the reciprocal of this function that we needed to analyse.)
In order to have an analogue of Lemma 
\ref{lem:home} for  even $m$ we need a holomorphic continuation of $G(s)$ to the left of the line $\sigma=1$.
However, any inverse root of $P_2$ has absolute value $q^{\frac{m}{2}+1}$
and it is therefore  possible  that 
$P_2(q^{-s-m/2})$ has a zero at $s=1$ (which would  imply that 
$G(s)$ has a pole there). Since we have been unsuccessful in our attempts to  analyse this situation precisely, 
this  prevents us from establishing a version of Theorem \ref{THM} when $n=9$ using the methods of this paper.
As pointed out to the authors by the anonymous referee, 
the location  of the   poles of 
$L\left(\eta_1\otimes \eta_2\otimes H_\ell^{m}(X_\c),s\right)$ is related to  the Tate conjectures and it 
would be interesting to see what they have to say in this setting.
\end{remark}

\section{Contribution from square-full moduli}\label{s:sqfull}

In what follows we will adhere to the notation introduced in Definition \ref{def:r} regarding $j$-full
numbers. Thus any $r\in \cO$  admits a unique factorisation 
$$
r=r_{j+1}\prod_{i=1}^{j}b_i=r_{j+1}\prod_{i=1}^{j}k_i^{i},
$$
for any integer $j\geq 1$, 
with $r_{j}$ being  $j$-full.
In particular it is easy to prove that 
$$
\sum_{|r_j|<\hat X}1=O(\hat X^{1/j}) 
\quad \text{ and }\quad 
\sum_{|r_j|>\hat X}|r_j|^{-\ell}=O(\hat X^{1/j-\ell})
$$
for any $X>1$ and  $\ell>1/j$.
We will make frequent use of these bounds without further comment.

In this section we complete our estimation of $E(P)$, which was initiated in 
\eqref{eq:E(P)-1}, by using the bounds for  
$\Sigma(Y,\theta)$ derived in the preceding section together with the 
estimates for averages of complete exponential sums in \S 
\ref{s:averages}. 
We begin by recalling 
\eqref{eq:E(P)-1}, in which it follows from 
\eqref{eq:alpha=1} that  
$$
S_{b_1',M_1,\b_1
}(\c)\ll |b_1'|^{(n+1)/2+\ve} |(b_1',F^*(\c))|^{1/2}.
$$
Hence there exists
$\b_1,\b_2\in (\cO/M\cO)^n$ and $b_0\in (\cO/M\cO)^*$ such that 
\begin{equation}\label{eq:E(P)-2}
\begin{split}
E(P)
\ll~& \frac{|P|^{n+\ve}}{\hat Y^{(n-1)/2}}
\sum_{\substack{ \c\in \cO^n\\
\c\neq \0\\
|\c|\ll \hat Y |P|^{-1}J(\Theta)
}} 
\sum_{\substack{b_1'\in \cO^\sharp\\ \vp\mid b_1'\Rightarrow \vp\in S}}
|(b_1',F^*(\c))|^{1/2}
\\
&\times
\sum_{\substack{
r_2\in \cO\\ |b_1'r_2|\leq \hat Y}}
\frac{
|S_{r_2,M_2,\b_2
}(\c)|}
{|r_2|^{(n+1)/2}}
\int_{|\theta|=\hat \Theta}
|\Sigma(Y, \theta)| \d\theta,
\end{split}
\end{equation}
for suitable $M_2\mid M$,  where 
 $\Sigma(Y,\theta)$ is given by \eqref{eq:access}.
Our treatment of this sum  differs according to the value of $\c$ in the outer sum.  
It will be convenient to differentiate these contributions by writing 
\begin{itemize}
\item
$E_1(P)$ for  the part coming from  $\c$ such that  $F^*(\c)\neq 0$,
\item
$E_2(P)$ for the part coming from  $\c$  such that $\nabla F^*(\c)\neq  \0$ but $F^*(\c)=0$,
\item
$E_3(P)$ for the  part coming from $\c\neq \0$  such that $\nabla F^*(\c)= \0$.
\end{itemize}
In our estimation of these quantities we will follow common convention and take $\ve>0$ to be a positive quantity whose value may change from one appearance to the next.
Finally, it will be convenient to set
\begin{equation}\label{eq:C-hat}
\hat C=\hat Y |P|^{-1} J(\Theta),
\end{equation}
to ease notation. In particular we must have $\hat C\gg 1$ in \eqref{eq:E(P)-2}, which recovers the bound
$\hat Y\gg |P|/J(\Theta)$
that we recorded in 
\eqref{eq:hatY-size}.
Throughout this section we will make frequent use of the  inequalities 
\eqref{eq:conditions on Y} satisfied by $Y$ and $\Theta$.

\subsection{Treatment of $E_1(P)$}

In  this section we will assume that $n\geq 8$ is even and we will
 take $S$ to be the set of primes dividing $\Delta_F M F^*(\c)$. In particular, it is worth emphasising that $|b_1'|$ can potentially be rather large. 

Let $Y_1\geq 0$ be such that 
\begin{equation}\label{eq:Y1}
\hat Y_1= \frac{\hat Y}{J(\Theta)}.
\end{equation}
Our argument will differ according to the size of  $|b_1'r_2|$.  
Let $E_{1,a}(P)$ be the contribution to $E_1(P)$ from $|b_1'r_2|\leq \hat Y_1$
and write 
$E_{1,b}(P)$ for the corresponding  contribution from $|b_1'r_2|\geq  \max\{1,\hat Y_1\}
$.
In the first scenario it will be more efficient to apply Lemma \ref{lem:goal}, whereas 
Lemma \ref{lem:goal-easy} is sharper in the second scenario.

\subsubsection*{The contribution from $|b_1'r_2|\leq \hat Y_1$}

We may suppose that $\hat Y_1\geq 1$ since otherwise there is nothing to prove.
Note  that 
$|b_1'r_2||P||\c|\ll \hat Y^2$
for any $\c$ contributing to  $E_{1,a}(P)$. 
Since $F^*(\c)\neq 0$,
it therefore  follows from Lemma \ref{lem:goal} that
$$
\Sigma(Y,\theta)\ll \frac{|P|^\ve}{|b_1'r_2|^{1/2}} \frac{\hat Y^{1/2}}{J(\Theta)^{n/2-1/2}},
$$
 for any $\ve>0$.
There are $O(|P|^\ve)$ choices of 
$b_1'\in \cO^\sharp$ such that $\vp\mid b_1'\Rightarrow \vp\in S$.
Employing our bound for 
$\Sigma(Y,\theta)$
 in \eqref{eq:E(P)-2} we therefore obtain
\begin{align*}
E_{1,a}(P)
\ll~& \frac{|P|^{n+\ve}\hat Y^{1-n/2}\hat \Theta }{J(\Theta)^{n/2-1/2}}
\sum_{\substack{ \c\in \cO^n\\
F^*(\c)\neq 0\\
|\c|\ll \hat C 
}} 
\sum_{\substack{
r_2\in \cO\\ |r_2|\leq \hat Y_1}}
\frac{
|S_{r_2,M_2,\b_2
}(\c)|}
{|r_2|^{n/2+1}}.
\end{align*}
Decomposing $r_2$ as $b_2r_3$ it follows from Lemma \ref{lem:multi2} and 
\eqref{eq:alpha=2}
that 
\begin{equation}\label{eq:sock}
\sum_{b_2\leq \hat Y_1}
\frac{|S_{b_2,M_2,\b_2
}(\c)|}{|b_2|^{n/2+1}}
\ll |P|^{\ve} 
\sum_{|k_2|\leq \hat Y_1^{1/2}}
\frac{|(k_2, F^*(\c))|}{|k_2| }
\ll |P|^{\ve}.
\end{equation}
Hence
\begin{align*}
E_{1,a}(P)
\ll~& \frac{|P|^{n+\ve}\hat Y^{1-n/2}\hat \Theta }{J(\Theta)^{n/2-1/2}}
\sum_{\substack{ \c\in \cO^n\\
F^*(\c)\neq 0\\
|\c|\ll \hat C
}} 
\sum_{\substack{
r_3\in \cO\\ |r_3|\leq \hat Y_1}}
\frac{
|S_{r_3,M_3,\b_3
}(\c)|}
{|r_3|^{n/2+1}},
\end{align*}
for appropriate 
 $M_3\mid M$ and $\b_3\bmod{M}$. 
The following result is devoted to estimating  the inner sums over $\c$ and $r_3$.

\begin{lemma}\label{lem:shoe}
Let $R\geq 1$. There exists a constant $\delta>0$ depending only on $n$ such that 
$$
\sum_{\substack{ \c\in \cO^n\\
|\c|\ll \hat C
}} 
\sum_{\substack{
r_3\in \cO\\ |r_3|= \hat R}}
|S_{r_3,M_3,\b_3
}(\c)|
\ll |P|^\ve \hat R^{n/2+4/3-\delta} 
\left(\hat R^{n/3}+\hat C^n\right).
$$
\end{lemma}

\begin{proof}
To estimate this  we write $r_3=b_3r_4 $ and 
we  will need to argue differently according to the size of  $|r_4|$. 
For a 
parameter 
 $0<Z\leq R$, to be defined in due course, the  contribution to the inner sum from  
 $r_4$ such that $|r_4|=\hat Z$ is at most
\begin{equation}\label{eq:knee}
\begin{split}
\sum_{\substack{
r_3=b_3r_4\in \cO\\ |r_3|= \hat R\\
|r_4|=\hat Z
}}
\sum_{\substack{ \c\in \cO^n\\
|\c|\ll \hat C
}}
|S_{r_3,M_3,\b_3
}(\c)|
&\ll |P|^\ve 
\sum_{\substack{
r_3=b_3r_4\in \cO\\ |r_3|= \hat R\\
|r_4|=\hat Z
}}
|r_3|^{n/2+1}
\left(| r_3|^{n/3}+\hat C^n\right)\\
&\ll \frac{|P|^\ve \hat R^{n/2+4/3}}{\hat Z^{1/12}}
\left(\hat R^{n/3}+\hat C^n\right),
\end{split}
\end{equation}
by Lemma \ref{lem:square-full}. 
Alternatively, for appropriate $M_3',\b_3'$,  the second part of this same result gives
\begin{align*}
\sum_{\substack{
r_3=b_3r_4\in \cO\\ |r_3|= \hat R\\
|r_4|=\hat Z
}}
\sum_{\substack{ \c\in \cO^n\\
|\c|\ll \hat C
}}
|S_{r_3,M_3,\b_3
}(\c)|
&\ll 
\sum_{\substack{
r_3=b_3r_4\in \cO\\ |r_3|= \hat R\\
|r_4|=\hat Z
}}|r_4|^{n+1}
\sum_{\substack{ \c\in \cO^n\\
|\c|\ll \hat C
}}
|S_{b_3,M_3',\b_3'
}(\c)|\\
&\ll |P|^\ve \hat Z^{n+1}
\sum_{\substack{
r_3=b_3r_4\in \cO\\ |r_3|= \hat R\\
|r_4|=\hat Z
}} |b_3|^{n/2+2/3}
\left(| b_3|^{n/3}+\hat C^n\right)\\
&\ll |P|^\ve \hat R^{n/2+1}\hat Z^{n/2+1/4}
\left(\hat R^{n/3}+\hat C^n\right).
\end{align*}
Taking the minimum of these two estimates and summing over $q$-adic intervals for $Z$, we readily arrive at the statement of the 
lemma.
\end{proof}

Recalling the definitions \eqref{eq:C-hat},  \eqref{eq:Y1} of $\hat C$ and $\hat Y_1$,
and applying Lemma~\ref{lem:shoe} with $q$-adic ranges for $\hat R\ll \hat Y_1$,  
our work so far shows that 
\begin{align*}
E_{1,a}(P)
\ll~& \frac{|P|^{n+\ve}\hat Y^{4/3-n/2-\delta}\hat \Theta }{J(\Theta)^{n/2-1/6-\delta}}
\left\{
\left(
\frac{\hat Y}{J(\Theta)}\right)^{n/3}+\left(\hat Y |P|^{-1}J(\Theta)\right)^n\right\}\\
\ll~& |P|^\ve
\left\{
\frac{|P|^{n-3}\hat Y^{4/3-n/6-\delta}\hat \Theta |P|^3 }{J(\Theta)^{5n/6-1/6-\delta}}
+
\hat Y^{n/2+4/3-\delta} \hat \Theta J(\Theta)^{n/2+1/6+\delta}\right\},
\end{align*}
for a  constant $\delta>0$ depending only on $n$. Note that $4/3-n/6-\delta<0$ for $n\geq 8$. Hence, in view of \eqref{eq:hatY-size}, there exists $\delta'>0$ such that 
the first term is
\begin{equation}\label{eq:whistla}
\ll |P|^{n-3-\delta'}
\frac{\hat \Theta |P|^{3} }{J(\Theta)}\ll |P|^{n-3-\delta'}.
\end{equation}
This  is clearly satisfactory.
Applying 
\eqref{eq:conditions on Y}, the second term is seen to be
\begin{equation}\label{eq:IHP}
\begin{split}
&\ll
|P|^\ve\hat \Theta
\hat Y^{n/2+4/3-\delta}
+
|P|^{3n/2+1/2+3\delta+\ve} \hat \Theta^{n/2+7/6+\delta}
\hat Y^{n/2+4/3-\delta}\\
&\ll
|P|^\ve \hat Q^{n/2-2/3-\delta}
+
\frac{|P|^{3n/2+1/2+3\delta+\ve}\hat Y^{1/6-2\delta}}{   \hat Q^{n/2+7/6+\delta}}\\
&\ll
|P|^{3n/4-1-\delta'},
\end{split}
\end{equation}
for an appropriate constant $\delta'>0$ depending on $\delta$ and $\ve$.
This is satisfactory for $n\geq 8$.

\subsubsection*{The contribution from $|b_1'r_2|\geq  \max\{1,\hat Y_1\}$}

Let us put $\hat Y_2=\max\{1,\hat Y_1\}$ to ease notation.
In this case we deduce from Lemma \ref{lem:goal-easy} that 
$$
 \Sigma(Y,\theta)\ll \frac{|P|^\ve}{|b_1'r_2|} \frac{\hat Y}{J(\Theta)^{n/2}},
$$
since $\nabla F^*(\c)\neq \0$.
Applying this bound in  \eqref{eq:E(P)-2} we obtain
\begin{align*}
E_{1,b}(P)
\ll~& \frac{|P|^{n+\ve}\hat Y^{(3-n)/2}\hat \Theta }{J(\Theta)^{n/2}}
\sum_{\substack{ \c\in \cO^n\\
F^*(\c)\neq 0\\
|\c|\ll \hat C
}} 
\sum_{\substack{b_1'\in \cO^\sharp\\ \vp\mid b_1'\Rightarrow \vp\in S}}
\sum_{\substack{
r_2\in \cO\\  \hat Y_2< |b_1'r_2|\leq  \hat Y}}
\frac{
|S_{r_2,M_2,\b_2
}(\c)|}
{|b_1'|^{1/2}|r_2|^{(n+3)/2}}.
\end{align*}
Decomposing $r_2$ as $b_2r_3$, we find that 
$$
\sum_{\substack{
r_2\in \cO\\  \hat Y_2< |b_1'r_2|\leq  \hat Y}}
\frac{
|S_{r_2,M_2,\b_2
}(\c)|}
{|r_2|^{(n+3)/2}}
=
\sum_{\substack{
r_2=b_2r_3\in \cO\\  \hat Y_2< |b_1'b_2r_3|\leq  \hat Y}}
\frac{
|S_{b_2,M_2',\b_2'
}(\c)S_{r_3,M_3,\b_3}(\c)|}
{|b_2r_3|^{(n+3)/2}},
$$
for appropriate $M_2',M_3, \b_2',\b_3$,
Summing this over the relevant $\c$, we now apply
Lemma \ref{lem:shoe} for $q$-adic values of  $\hat R$ in the interval  $
\hat Y_2/|b_1'b_2|< \hat R \leq  \hat Y/|b_1'b_2|$ to conclude that
\begin{align*}
\sum_{\substack{ \c\in \cO^n\\
F^*(\c)\neq 0\\
|\c|\ll \hat C
}} 
\sum_{\substack{
r_3\in \cO\\  \hat Y_2< |b_1'b_2r_3|\leq  \hat Y}}
\frac{
|S_{r_3,M_3,\b_3
}(\c)|}
{|r_3|^{(n+3)/2}}
&\ll
|P|^\ve
 \left( \hat Y^{n/3-1/6-\delta}+
\frac{\hat C^n}{ (\hat Y_2/|b_1'b_2|)^{1/6+\delta}}\right).
\end{align*}
The sums over $b_1'$ and  $b_2$ are now easily estimated (with recourse to  \eqref{eq:sock} for the latter). 
Hence, recalling \eqref{eq:C-hat},  we obtain
\begin{equation}\label{eq:E1b}
\begin{split}
E_{1,b}(P)
\ll~& \frac{|P|^{n+\ve}\hat Y^{(3-n)/2}\hat \Theta }{J(\Theta)^{n/2} }
 \left( \hat Y^{n/3-1/6-\delta}+
\frac{(\hat Y |P|^{-1}J( \Theta))^n}{ \hat Y_2^{1/6+\delta}}\right)\\
\ll~& |P|^\ve \hat \Theta\left\{
\frac{|P|^{n}\hat Y^{4/3-n/6-\delta} }{J(\Theta)^{n/2} }
+
\frac{\hat Y^{n/2+3/2} J( \Theta)^{n/2}}{ \hat Y_2^{1/6+\delta}}
\right\}
\end{split}
\end{equation}
for some $\delta>0$.
Note that $4/3-n/6-\delta<0$ for $n\geq 8$, as before.  Hence, in view of \eqref{eq:hatY-size}, there exists $\delta'>0$ such that 
the first term is bounded by \eqref{eq:whistla},
which is satisfactory.
On the other hand, taking $\hat Y_2\geq \hat Y_1=\hat Y/J(\Theta)$, the second term is seen to be
\begin{align*}
&\ll
|P|^\ve\hat \Theta
\hat Y^{n/2+4/3-\delta}
+
|P|^{3n/2+1/2+3\delta+\ve} \hat \Theta^{n/2+7/6+\delta}
\hat Y^{n/2+4/3-\delta}.
\end{align*}
But this is  satisfactory for $n\geq 8$, by \eqref{eq:IHP}.

\subsection{Treatment of $E_2(P)$}

In this section we will assume that $n\geq 8$ (without any assumption on the parity)
and we will
 take $S$ to be the set of primes dividing $\Delta_F M \nabla F^*(\c)$. 
 There are $O(|P|^\ve)$ choices for $b_1'$ in \eqref{eq:E(P)-2}.
Applying Lemma 
\ref{lem:goal-easy}, we therefore obtain the bound
\begin{align*}
E_2(P)
\ll~& \frac{|P|^{n+\ve}\hat \Theta}{\hat Y^{(n-3)/2} J(\Theta)^{n/2}}
\sum_{\substack{ \c\in \cO^n\\
F^*(\c)=0\\
\nabla F^*(\c)\neq \0\\
|\c|\ll \hat C
}} 
\sum_{\substack{
r_2\in \cO\\ |r_2|\leq \hat Y}}
\frac{
|S_{r_2,M_2,\b_2
}(\c)|}
{|r_2|^{(n+3)/2}}.
\end{align*}
The argument used in 
\eqref{eq:sock} allows us to replace $r_2$ by $r_3$ in the inner sum, after adjusting the value of the parameter $\ve$ in the exponent of $|P|$.
We will need the following analogue of 
Lemma \ref{lem:shoe}.

\begin{lemma}\label{lem:shoe'}
Let $R\geq 1$ and put $\delta=\tfrac{1}{2(n-1)}$. Then
\begin{align*}
\sum_{\substack{ \c\in \cO^n\\
F^*(\c)= 0\\
|\c|\ll \hat C
}} 
\sum_{\substack{
r_3\in \cO\\ |r_3|= \hat R}}
\frac{|S_{r_3,M_3,\b_3
}(\c)|}{|r_3|^{(n+3)/2}}
\ll~& |P|^\ve
\left(
\frac{\hat R^{n/3}}{\hat R^{1/6+\delta/3}}+
\hat C^{n-1/2-\delta}
+
\hat R^{1/6}{\hat C}^{n-3/2}\right),
\end{align*}
\end{lemma}

\begin{proof}
To estimate this  we write $r_3=b_3r_4 $ and 
we  start by considering the 
contribution to the inner sum from  
 $r_4$ such that $|r_4|=\hat Z$, 
for a parameter 
 $0<Z\leq R$  to be defined in due course. 
The estimate \eqref{eq:knee} gives 
 \begin{align*}
\sum_{\substack{
r_3=b_3r_4\in \cO\\ |r_3|= \hat R\\
|r_4|=\hat Z
}}
\sum_{\substack{ \c\in \cO^n\\
F^*(\c)=0\\
|\c|\ll \hat C
}}
|S_{r_3,M_3,\b_3
}(\c)|
&\ll \frac{|P|^\ve \hat R^{n/2+4/3}}{\hat Z^{1/12}}
\left(\hat R^{n/3}+\hat C^n\right).
\end{align*}
Alternatively, we invoke Lemma 
\ref{lem:square-full 3}, which gives 
\begin{align*}
\sum_{\substack{ \c\in \cO^n\\
F^*(\c)=0\\
|\c|\ll \hat C
}}
|S_{r_3,M_3,\b_3
}(\c)|
&\ll
\hat R^{n/2+4/3+\ve}
\left(
|b_3|^{n/3-2/3}\hat Z^{n/2-5/6}+{\hat C}^{n-3/2}\right)\\
&=
\hat R^{n/2+4/3+\ve}
\left(
\hat R^{n/3-2/3}\hat Z^{n/6-1/6}+{\hat C}^{n-3/2}\right),
\end{align*}
for any $r_3=b_3r_4\in \cO$ such that $|r_3|= \hat R$ and $|r_4|=\hat Z$. There are clearly $O(\hat R^{1/3}\hat Z^{-1/12})$
such choices for $r_3.$
This therefore gives
\begin{align*}
\sum_{\substack{
r_3=b_3r_4\in \cO\\ |r_3|= \hat R\\
|r_4|=\hat Z
}}
\sum_{\substack{ \c\in \cO^n\\
F^*(\c)=0\\
|\c|\ll \hat C
}}
|S_{r_3,M_3,\b_3
}(\c)|
&\ll 
\frac{\hat R^{n/2+5/3+\ve}}{\hat Z^{1/12}}
\left(
\hat R^{n/3-2/3}\hat Z^{n/6-1/6}+{\hat C}^{n-3/2}\right)\\
&\ll  \hat R^{n/2+4/3+\ve}\left(
\hat R^{n/3-1/3}\hat Z^{n/6-1/4}
+
\hat R^{1/3}
{\hat C}^{n-3/2}\right).
\end{align*}
Taking the minimum of these two estimates gives
\begin{align*}
\sum_{\substack{
r_3=b_3r_4\in \cO\\ |r_3|= \hat R\\
|r_4|=\hat Z
}}
\sum_{\substack{ \c\in \cO^n\\
F^*(\c)=0\\
|\c|\ll \hat C
}}
|S_{r_3,M_3,\b_3
}(\c)|
\ll~& |P|^\ve \hat R^{n/2+4/3}\left(A+B+
\hat R^{1/3}{\hat C}^{n-3/2}\right),
\end{align*}
where 
\begin{align*}
A&=\min\{\hat R^{n/3-1/3}\hat Z^{n/6-1/4}, \hat Z^{-1/12}\hat R^{n/3}\},\\
B&=\min\{\hat R^{n/3-1/3}\hat Z^{n/6-1/4}, \hat Z^{-1/12}\hat C^{n}\}.
\end{align*}
We take $\min\{X,Y\}\leq X^{\delta}Y^{1-\delta}$ in both of these, 
with $\delta=\tfrac{1}{2(n-1)}$, 
to find that 
$A\leq \hat R^{n/3-\delta/3}$ and $B
\leq \hat R^{1/6} \hat C^{n-1/2-\delta}.
$
Summing over $q$-adic intervals for $\hat Z$, we quickly arrive at the statement of the 
lemma.
\end{proof}

Applying Lemma \ref{lem:shoe'} in our earlier bound for $E_2(P)$, with $1\leq \hat R\leq  \hat Y$, we are led to the conclusion that 
\begin{align*}
E_2(P)
\ll~& \frac{|P|^{n+\ve}\hat \Theta}{\hat Y^{(n-3)/2} J(\Theta)^{n/2}}
\left(
\hat Y^{n/3-1/6-\delta/3}+
 \hat C^{n-1/2-\delta}
+
\hat Y^{1/6}{\hat C}^{n-3/2}\right),
\end{align*}
where $\delta=\tfrac{1}{2(n-1)}$
The first term here is equal to the first term in the estimate \eqref{eq:E1b} for $E_{1,b}(P)$, with a different value of $\delta$,  and so makes a satisfactory overall contribution for $n\geq 8$.
Recalling the definition \eqref{eq:C-hat} of $\hat C$, the second   term contributes
\begin{align*}
&\ll\frac{|P|^{n+\ve}\hat \Theta (\hat Y |P|^{-1}J(\Theta))^{n-1/2-\delta}}{\hat Y^{(n-3)/2} J(\Theta)^{n/2}}\\
&= |P|^{1/2+\delta+\ve}\hat \Theta  \hat Y^{n/2+1-\delta}J(\Theta)^{n/2-1/2-\delta}\\
&\ll |P|^{1/2+\delta+\ve}\hat Q^{n/2-1-\delta}
+ |P|^{3n/2-1-2\delta+\ve}\hat Y^{n/2+1-\delta}\Theta^{n/2+1/2-\delta}\\
&\ll |P|^{3n/4-1-\delta/2+\ve},
\end{align*}
which is satisfactory for $n\geq 8$. 
Similarly, the contribution from the third term is seen to be 
\begin{align*}
&\ll\frac{|P|^{n+\ve}\hat \Theta (\hat Y |P|^{-1}J(\Theta))^{n-3/2}}{\hat Y^{n/2-5/3} J(\Theta)^{n/2}}\\
&= |P|^{3/2+\ve}\hat \Theta  \hat Y^{n/2+1/6}J(\Theta)^{n/2-3/2}\\
&\ll |P|^{3/2+\ve}\hat Q^{n/2-11/6}
+ |P|^{3n/2-3+\ve}\hat Y^{n/2+1/6}\Theta^{n/2-1/2}\\
&\ll |P|^{3n/4-5/4+\ve},
\end{align*}
which is also satisfactory for $n\geq 8$.

\subsection{Treatment of $E_3(P)$}

In this section we will assume that $n= 8$ and we 
take $S$ to be the set of primes dividing $\Delta_F M$. 
 We combine the second part of 
 Lemma 
\ref{lem:goal-easy} with the 
 argument used in 
\eqref{eq:sock} to  replace $r_2$ by $r_3$, to get 
\begin{align*}
E_3(P)
\ll~& \frac{|P|^{n+\ve}\hat \Theta}{\hat Y^{n/2-2} J(\Theta)^{n/2}}
\sum_{\substack{
r_3\in \cO\\ |r_3|\leq \hat Y}}
\sum_{\substack{ \c\in \cO^n\\
\nabla F^*(\c)= \0\\
0<|\c|\ll \hat C
}} 
\frac{
|S_{r_3,M_3,\b_3
}(\c)|}
{|r_3|^{n/2+2}},
\end{align*}
for appropriate $M_3\mid M$ and $\b_3 \bmod{M_3}$.
Our main tools to estimate the inner sum over $\c$ will be 
Lemma 
\ref{lem:square-full 5} and its corollary 
\eqref{eq:cor-square}, together with 
Lemma~\ref{lem:square-full}.
We begin with the following result. 

\begin{lemma}\label{lem:bon}
Let $n=8$ and let $\Delta>0$. Then 
$$
 \frac{|P|^{n+\ve}\hat \Theta \hat C^{6-\Delta}}{\hat Y^{n/2-2} J(\Theta)^{n/2}} \ll |P|^{5-\Delta/2+\ve}=|P|^{n-3-\Delta/2+\ve}
$$
\end{lemma}

\begin{proof}
Recalling the notation \eqref{eq:C-hat} for $\hat C$,
we take $n=8$ and see that the  left hand side is
\begin{align*}
&\ll |P|^{2+\Delta+\ve}\hat \Theta \hat Y^{4-\Delta}J(\Theta)^{2-\Delta}\\
&\ll |P|^{2+\Delta+\ve}\hat \Theta \hat Y^{4-\Delta}
+ |P|^{8-2\Delta+\ve}\hat \Theta^{3-\Delta} \hat Y^{4-\Delta}\\
&\ll |P|^{5-\Delta/2+\ve},
\end{align*}
as claimed.
\end{proof}

To begin with we dispatch the  contribution from $r_3$ for which $|b_3|>\hat Y^{1-\delta}$, for some small value of $\delta>0$ to be determined below.  In particular we must have $|r_4|<\hat Y^{\delta}$ in the decomposition $r_3=b_3r_4$.  In this setting
\eqref{eq:cor-square} gives the contribution 
\begin{align*}
&\ll \frac{|P|^{n+\ve}\hat \Theta}{\hat Y^{n/2-2} J(\Theta)^{n/2}}
\sum_{\substack{
r_3\in \cO\\ |r_3|\leq \hat Y\\ |b_3|>\hat Y^{1-\delta}}}
\left(\frac{\hat C^{n-5/2}}{|b_3|^{1/2}}+|b_3|^{4/3}|r_4|^{5/2}\right)\\
&\ll \frac{|P|^{n+\ve}\hat \Theta \hat Y^{O(\delta)}}{\hat Y^{n/2-2} J(\Theta)^{n/2}}
\left(\frac{\hat C^{n-5/2}}{\hat Y^{1/6}}+\hat Y^{5/3}\right),
\end{align*}
since there are $O(\hat Y^{1/3})$ available choices of $r_3$.
Assuming that $\delta$ is sufficiently small,
the first term makes a  satisfactory contribution, by Lemma \ref{lem:bon}.
On the other hand, taking $n=8$, the second term contributes
\begin{align*}
&\ll \frac{|P|^{n+\ve}\hat \Theta \hat Y^{-1/3+O(\delta)}}{ J(\Theta)^{n/2}}
= |P|^{n-3+\ve}\left(\frac{\hat \Theta |P|^{3} \hat Y^{-1/3+O(\delta)}}{ J(\Theta)^{n/2}}\right).
\end{align*}
This too is satisfactory, if $\delta$ is small enough,  since $\hat Y\gg |P|/J(\Theta)$.

We now turn to the contribution from $|r_3|\leq \hat Y$ such that $|b_3|\leq \hat Y^{1-\delta}$. 
There are clearly at most $O(\hat Y^{1/3-\delta/12})$ choices for $r_3$.
In fact the only place
we will need to use this inequality is when dealing with the term $|r_3|^{5n/6+1}$ that appears in Lemma \ref{lem:square-full}. 
Summing over the available $r_3$ the effect of this term is seen to be
\begin{align*}
&\ll \frac{|P|^{n+\ve}\hat \Theta}{\hat Y^{n/2-2} J(\Theta)^{n/2}}
\sum_{\substack{
r_3\in \cO\\ |r_3|\leq \hat Y\\ |b_3|\leq \hat Y^{1-\delta}}}
|r_3|^{n/3-1}
\ll \frac{|P|^{n+\ve}\hat \Theta}{\hat Y^{n/6-4/3+\delta/12} J(\Theta)^{n/2}}.
\end{align*}
The exponent of $\hat Y$ is strictly positive for $n=8$, which is  enough to conclude that this term makes a satisfactory overall contribution.

Applying Lemmas~\ref{lem:square-full 5} and 
\ref{lem:square-full}, the remaining contribution is found to be at most 
\begin{align*}
&\ll \frac{|P|^{n+\ve}\hat \Theta}{\hat Y^{n/2-2} J(\Theta)^{n/2}}
\left( H_1+H_2\right),
\end{align*}
where
$$
H_1=
\sum_{\substack{
r_3\in \cO\\ |r_3|\leq \hat Y}}
\min\left\{ \frac{\hat C^n}{|r_3|} ~,~ 
\frac{|G_1(r_3)|\hat C^{n-5/2}}{|r_3|^{n/2+2}}
\right\}, 
~~
H_2=
\sum_{\substack{
r_3\in \cO\\ |r_3|\leq \hat Y}}
\min\left\{ \frac{\hat C^n}{|r_3|} ~,~ 
\frac{|G_2(r_3)|}{|r_3|^{n/2+2}}\right\}.
$$
In the light of Lemma \ref{lem:bon} it suffices to show the existence of positive constants $\Delta_1,\Delta_2>0$ such that 
$
H_i\ll 
 \hat C^{6-\Delta_i}$, 
for $i=1,2$.
Beginning with $H_1$, 
we take $\min\{X,Y\}\leq X^{1/5-2\delta/5}Y^{4/5+2\delta/5}$
for a very small value of $\delta>0$.
Recalling that $n=8$ and then appealing to    \eqref{eq:G1}, we therefore find that 
$$
H_1
\leq 
 \hat C^{6-\delta}
\sum_{\substack{
r_3\in \cO\\ |r_3|\leq \hat Y}}
\frac{|r_3|^{O(\delta)}}
{|r_3|^{1/5}\left(|b_3|^{1/2}|b_4|^{3/8}|b_5|^{2/5}|b_6|^{1/12}|b_7|^{5/14}|r_8|^{1/4}\right)^{4/5}}.
$$
This therefore gives
$H_1
\ll  \hat C^{6-\delta}$, as required, since 
the sum over $r_3$ is absolutely convergent if $\delta$ is small enough.
Turning to  $H_2$, 
for a very small value of $\delta>0$, 
 we take $\min\{X,Y\}\leq X^{3/4-\delta/8}Y^{1/4+\delta/8}$.
 This time we appeal to   \eqref{eq:G2}, giving 
$$
H_2
\leq 
 \hat C^{6-\delta}
\sum_{\substack{
r_3\in \cO\\ |r_3|\leq \hat Y}}
\frac{|r_3|^{O(\delta)}\left(|b_3|^{4/3}|b_4| |b_5|^{9/5}|b_6|^{2}|b_7|^{2}|b_8|^{19/8} |r_9|^{5/2}\right)^{1/4}}{|r_3|^{3/4}}.
$$
This therefore gives
$H_2
\ll  \hat C^{6-\delta}$, as required, 
since 
the sum over $r_3$ is absolutely convergent if $\delta$ is small enough.

\end{document}